\documentclass[11pt]{amsart}

\usepackage[margin=1.05in,
           top=1.1in,
           bottom=1.1in,
           footskip=.25in]{geometry}
\usepackage{amssymb}
\usepackage{amscd}
\usepackage[all]{xy}
\usepackage{bbm}
\usepackage{mathrsfs}
\usepackage{enumerate}
\usepackage{tikz}
\usepackage{geometry}
\usepackage{stmaryrd}
\usepackage{etoolbox}
\usepackage{array}
\usepackage{hyperref}

\newcommand{\ric}{\operatorname{Ric}}
\newcommand{\N}{\mathbb{N}}
\newcommand{\Z}{\mathbb{Z}}

\newcommand{\R}{\mathbb{R}}
\newcommand{\inv}{^{-1}}
\newcommand{\eps}{\varepsilon}

\newcommand{\del}{\nabla}

\newcommand{\lap}{\Delta}
\newcommand{\bd}{\partial}
\newcommand{\cl}{\overline}
\newcommand{\ins}{\operatorname{int}}
\newcommand{\rest}{\scalebox{1.7}{$\llcorner$}}
\newcommand{\eval}{\bigg\vert}
\newcommand{\la}{\langle}
\newcommand{\ra}{\rangle}

\newcommand{\supp}{\operatorname{supp}}
\newcommand{\dist}{\operatorname{dist}}
\renewcommand{\div}{\operatorname{div}}

\newcommand{\RP}{{\R \text{\normalfont P}}}
\newcommand{\grad}{\del}
\newcommand{\vol}{\operatorname{Vol}}
\newcommand{\f}{\colon}
\newcommand{\area}{\operatorname{Area}}

\newcommand{\an}{\text{An}}
\newcommand{\h}{\mathcal H} 
\newcommand{\hv}{\mathfrak h}
\newcommand{\vc}{\operatorname{VC}}

\newcommand{\z}{\mathcal Z}
\newcommand{\B}{\mathcal B}
\newcommand{\C}{\mathcal C}

\newcommand{\vz}{\operatorname{VZ}}
\newcommand{\V}{\mathcal V}
\newcommand{\M}{\mathbf M}

\newcommand\numberthis{\addtocounter{equation}{1}\tag{\theequation}}

\theoremstyle{plain}
\newtheorem{theorem}{Theorem}[section]
\newtheorem{corollary}[theorem]{Corollary}
\newtheorem{prop}[theorem]{Proposition}
\newtheorem{lem}[theorem]{Lemma}
\newtheorem{conj}[theorem]{Conjecture}

\theoremstyle{definition}
\newtheorem{defn}[theorem]{Definition}

\theoremstyle{remark}
\newtheorem{rem}[theorem]{Remark}

\begin{document}

\title[Infinitely many half-volume CMCs]{Infinitely many half-volume constant mean curvature hypersurfaces via min-max theory}
\author{Liam Mazurowski}
\author{Xin Zhou}
\address{Cornell University, Department of Mathematics, Ithaca, New York 14850}
\email{lmm334@cornell.edu}
\address{Cornell University, Department of Mathematics, Ithaca, New York 14850}
\email{xinzhou@cornell.edu}

\begin{abstract}
Let $(M^{n+1},g)$ be a closed Riemannian manifold of dimension $3\le n+1\le 5$.  We show that, if the metric $g$ is generic or if the metric $g$ has positive Ricci curvature, then $M$ contains infinitely many geometrically distinct constant mean curvature hypersurfaces, each enclosing half the volume of $M$. As an essential part of the proof, we develop an Almgren-Pitts type min-max theory for certain non-local functionals of the general form  \[\Omega \mapsto \area(\bd \Omega) - \int_\Omega h + f(\vol(\Omega)).\]
\end{abstract}

\maketitle


\section{Introduction}
\label{section:introduction}

The volume spectrum of a closed, Riemannian manifold $M$ is a sequence of real numbers $\{\omega_p(M)\}_{p\in \N}$ associated to the geometry of $M$. The volume spectrum was introduced by Gromov \cite{gromov2006dimension} as a non-linear analog to the spectrum of the Laplacian. To define the volume spectrum, one first considers all $p$-sweepouts of $M$. Heuristically, a $p$-sweepout is a continuous family of hypersurfaces with the property that for any choice of points $x_1,\hdots,x_p\in M$, there is a hypersurface $\Sigma$ in the family with $x_1,\hdots,x_p\in \Sigma$. Then the $p$-width $\omega_p(M)$ is defined as the min-max value 
\[
\omega_p(M) = \inf_{p\text{-sweepouts } \Phi}\left[\sup_{\text{hypersurfaces } \Sigma \text{ in } \Phi} \area(\Sigma)\right]. 
\]
Liokumovich, Marques, and Neves \cite{liokumovich2018weyl} proved that the volume spectrum satisfies a Weyl law characterizing the asymptotic growth of $\omega_p(M)$ as a function of $p$.

In \cite{mazurowski2023half}, the authors introduced the half-volume spectrum $\{\tilde \omega_p(M)\}_{p\in \N}$. This is defined analogously to the usual volume spectrum, except that one restricts to $p$-sweepouts by hypersurfaces that are each required to enclose half the volume of $M$. The half-volume spectrum also satisfies a Weyl law \cite{mazurowski2023half}. In this paper, we develop an Almgren-Pitts type min-max theory for finding hypersurfaces associated to the half-volume spectrum. Our main result is the following theorem. 

\begin{theorem}
\label{main}
Assume $M^{n+1}$ is a closed manifold of dimension $3\le n+1\le 5$. Let $g$ be a generic Riemannian metric on $M$. Then for each $p\in \N$ there exists an open set $\Omega_p\subset M$ with $\vol(\Omega_p) = \frac{1}{2}\vol(M)$ such that $\bd \Omega_p$ is smooth and almost embedded, has non-zero constant mean curvature, and satisfies $\area(\bd \Omega_p) = \tilde \omega_p(M)$.
\end{theorem}

Combined with the Weyl law \cite[Theorem 1]{mazurowski2023half} for the half-volume spectrum, this has the following immediate corollary. 

\begin{corollary}
\label{main2}
Assume $M^{n+1}$ is a closed manifold of dimension $3\le n+1\le 5$. Then, for a generic Riemannian metric $g$ on $M$, there exist infinitely many geometrically distinct closed constant mean curvature hypersurfaces in $M$, each enclosing half the volume of $M$.  
\end{corollary}

\begin{rem}
    In Theorem \ref{main} and Corollary \ref{main2}, the word generic is to be interpreted in the sense of Baire category. In other words, there is a subset $\Gamma$ of the space of all smooth metrics on $M$ such that $\Gamma$ is of second Baire category and the conclusions of Theorem \ref{main} and Corollary \ref{main2} hold for all $g\in \Gamma$. 
\end{rem}

By approximating a metric $g$ with positive Ricci curvature by a sequence of generic metrics, we can show that the above conclusions also hold for metrics with positive Ricci curvature. 

\begin{theorem}
    \label{positive-Ricci}
Assume $M^{n+1}$ is a closed manifold of dimension $3\le n+1\le 5$. Assume $g$ is a Riemannian metric on $M$ with positive Ricci curvature. Then for each $p\in \N$ there exists an open set $\Omega_p\subset M$ with $\vol(\Omega_p) = \frac{1}{2}\vol(M)$ such that $\bd \Omega_p$ is smooth and almost embedded, has constant mean curvature (possibly equal to 0), and satisfies $\area(\bd \Omega_p) = \tilde \omega_p(M)$.\end{theorem}

\subsection{Background} Min-max theory is a powerful tool for finding critical points. In differential geometry, min-max theory has been used to great success to understand the existence of minimal surfaces, constant mean curvature surfaces, and more general prescribed mean curvature surfaces. In the early 1980s, Almgren \cite{almgren1965theory}, Pitts \cite{pitts2014existence}, and Schoen-Simon \cite{schoen1981regularity} developed a min-max theory for finding critical points of the area functional. Their combined work implies that every closed Riemannian manifold $M^{n+1}$ with dimension $3\le n+1\le 7$ contains a smooth, closed, embedded minimal hypersurface. Around the same time, Yau \cite{yau1982seminar} conjectured that every closed Riemannian 3-manifold should contain infinitely many closed minimal surfaces. Motivated by Yau's conjecture, Marques and Neves developed a program to understand the Morse theory of the area functional by further refining the Almgren-Pitts min-max theory. In \cite{marques2016morse}, Marques and Neves proved the following theorem.

\begin{theorem}[Marques-Neves \cite{marques2016morse}]
\label{theorem:MN1}
Let $M^{n+1}$ be a closed Riemannian manifold of dimension $3\le n+1\le 7$ and fix an integer $p\in \N$. Then there exists a collection of disjoint closed minimal hypersurfaces $\Sigma_1,\hdots,\Sigma_k$ in $M$ together with multiplicities $m_1,\hdots,m_k\in \N$ such that $\omega_p(M) = m_1\area(\Sigma_1) + \hdots + m_k\area(\Sigma_k)$. 
\end{theorem}

Theorem \ref{theorem:MN1} alone does not suffice to prove Yau's conjecture because of the possible appearance of multiplicities. Indeed, different $p$-widths could be achieved by the same underlying collection of minimal surfaces, but with different multiplicities.  Marques and Neves \cite{marques2021morse} conjectured that, for a generic choice of metric, all of the multiplicities in Theorem \ref{theorem:MN1} can be taken to be 1. In other words, for a generic choice of metric, every $\omega_p$ is equal to the area of some (possibly disconnected) minimal hypersurface $\Sigma_p$. This {\it Multiplicity One Conjecture} was proven by the second named author \cite{zhou2020multiplicity}, using the prescribed mean curvature regularization of the area functional developed by the second author and J. Zhu \cite{zhou2020existence}. The combined work of Marques and Neves together with the resolution of the Multiplicity One Conjecture gives the following strong answer to Yau's conjecture for generic metrics.

\begin{theorem}[\cite{marques2021morse},\cite{zhou2020multiplicity}]
\label{mult1}
Let $M^{n+1}$ be a closed manifold with dimension $3\le n+1\le 7$. For a generic choice of metric $g$ on $M$, there is a sequence of (possibly disconnected) closed minimal hypersurfaces $\Sigma_p$ such that $\omega_p(M) = \area(\Sigma_p)$.
\end{theorem}

We note that prior to Theorem \ref{mult1}, Irie, Marques, and Neves \cite{irie2018density} had already established the existence of infinitely many closed minimal hypersurfaces for generic metrics. Moreover, Marques, Neves, and Song \cite{marques2019equidistribution} proved that for a generic choice of metric on $M$, some sequence of minimal hypersurfaces becomes equidistributed in $M$. Finally, Song \cite{song2023existence} proved that {\it every} metric $g$ on $M$ admits infinitely many minimal hypersurfaces, thus completely resolving Yau's conjecture. For a more detailed overview of this subject, we refer to the survey articles \cite{Marques-Neves21} and \cite{Zhou-ICM22}. 

There is a parallel min-max procedure for constructing minimal hypersurfaces in Riemannian manifolds based on the theory of phase transitions and the Allen-Cahn equation. In this context, Gaspar and Guaraco \cite{gaspar2018allen} defined an Allen-Cahn analog to the volume spectrum $\{c_p(M)\}_{p\in N}$ and proved a corresponding version of Theorem \ref{theorem:MN1}. They also showed that $c_p(M)$ satisfies a Weyl law \cite{gaspar2019weyl}.  In the Allen-Cahn setting, Chodosh and Mantoulidis \cite{chodosh2020minimal} proved that the Multiplicity One Conjecture holds in ambient dimension $n+1=3$. In \cite{mazurowski2023half}, the authors also defined an Allen-Cahn version of the half-volume spectrum $\{\tilde c_p(M)\}_{p\in \N}$. 

\vspace{1em}
In the context of constant mean curvature hypersurfaces, the $A^c$ min-max theory of the second author and J. Zhu \cite{zhou2019min} can be used to find closed CMC hypersurfaces with prescribed mean curvature $c$. This min-max theory gives no control over the enclosed volume. It is natural to wonder if there is a corresponding min-max theory for producing {\it unstable} CMC hypersurfaces with prescribed volume (but with no control over the value of the mean curvature). Of course, by the solution of the isoperimetric problem, $M$ admits a {\it stable} CMC enclosing volume $v$ for every prescribed volume $v$. We are asking if there is a meaningful higher parameter version of this. 

Intuitively, one expects that the half-volume spectrum should detect critical points of the area functional with a half-volume constraint. 
In \cite{mazurowski2023half}, the authors used the Allen-Cahn min-max theory and the work of Bellettini-Wickramasekera \cite{bellettini2020inhomogeneous} to show that each Allen-Cahn half-volume spectrum $\tilde c_p$ is achieved by a constant mean curvature hypersurface enclosing half the volume of $M$, together with a collection of minimal hypersurfaces with even multiplicities.  This is an analog to Theorem \ref{theorem:MN1} in the half-volume Allen-Cahn setting.

In this paper, we prove a corresponding theorem for the Almgren-Pitts half-volume spectrum $\tilde \omega_p$. In fact, we are able to use a prescribed mean curvature regularization to obtain a multiplicity one type result. Thus we obtain Theorem \ref{main}, which is an analog to Theorem \ref{mult1} in the half-volume Almgren-Pitts setting. As a corollary, we prove the existence of infinitely many half-volume CMCs for generic metrics. As part of the proof, we develop an Almgren-Pitts type min-max theory for certain non-local functionals. These arguments may be of interest in their own right, and we hope that they will prove useful in other contexts. 

\begin{rem}
Most of our arguments work in the larger dimension range $3\le n+1 \le 7$, but the extra restriction $3\le n+1\le 5$ is needed at exactly one point in the proof. This is related to the fact that stable CMCs are known to have diameter bounds in these dimensions \cite{elbert2007stable}. 
\end{rem}

\subsection{Motivation and Conjectures} 
A well-known conjecture of Arnold \cite{arnold2004arnold} states that every Riemannian 2-sphere $(S^2,g)$ admits at least two distinct closed curves with constant geodesic curvature $\kappa$ for every $\kappa > 0$. The following ``{\it Twin Bubble Conjecture}'' \cite{Zhou-ICM22} can be seen as a higher dimensional analog of Arnold's conjecture; see also \cite{Dey23,Mazurowski22} for some partial results. 

\begin{conj}
\label{two-cmcs}
    Every closed Riemannian manifold $(M^{n+1},g)$ with $3\leq n+1\leq 7$ admits at least two distinct closed hypersurfaces with constant mean curvature $c$ for every constant $c > 0$. 
\end{conj}

Our initial motivation for studying half-volume CMCs came from Conjecture \ref{two-cmcs}.  Naively, one expects that the two CMCs in Conjecture \ref{two-cmcs} come from ``sweeping out $M$ in opposite directions.'' Of course, topologically all sweepouts of $M$ are homotopic, and so it does not make sense to distinguish the direction of different sweepouts. Nevertheless, by considering the functional 
\[
\Omega \mapsto \max\{\area(\bd \Omega) - c\vol(\Omega), \area(\bd \Omega) - c\vol(M\setminus \Omega)\}, 
\]
one can artificially attempt to see both a sweepout and ``the sweepout in the opposite direction.'' Performing min-max for this functional over all sweepouts of $M$, we expect to detect either a hypersurface with constant mean curvature $c$, or a half-volume CMC (with no control over the value of the mean curvature). 
The first case should correspond to the second $c$-CMC hypersurface in addition to \cite{zhou2019min}. Our current work provides a thorough understanding of the latter case via restriction of this functional to all half-volume hypersurfaces.

We also record the following counterpart to Conjecture \ref{two-cmcs}, where prescribed mean curvature is replaced by a volume constraint. 

\begin{conj}
    Every closed Riemannian manifold $(M^{n+1},g)$ with $3\leq n+1\leq 7$ admits at least two distinct closed CMC hypersurfaces enclosing volume $v$ for every $v\in (0,\vol(M))$. 
\end{conj}

\subsection{Outline of Proof}

In the remainder of the introduction, we will give a sketch of the proof of Theorem \ref{main}. This section is more technical than the rest of the introduction. We will assume that the reader is familiar with some of the notions from geometric measure theory introduced in Section \ref{section:preliminaries}, as well as the basic ideas in the Almgren-Pitts min-max theory.

Let $M^{n+1}$ be a closed Riemannian manifold. The Almgren-Pitts min-max theory works with homotopy classes of maps into the space $\mathcal Z_n(M,\Z_2)$ of flat cycles mod $2$. Because we work with a volume constraint, we will restrict to the space $\mathcal B(M,\Z_2)$ consisting of those $T \in \mathcal Z_n(M,\Z_2)$ such that $T = \bd \Omega$ for some Caccioppoli set $\Omega$.  Within $\mathcal B(M,\Z_2)$, we can further consider the space of half-volume cycles 
\[
\mathcal H(M,\Z_2) = \{T \in \mathcal B(M,\Z_2): T= \bd \Omega \text{ and } \vol(\Omega) = \frac{1}{2}\vol(M)\}. 
\]
At first, one might hope to develop a min-max theory for the area functional directly on the space $\mathcal H(M,\Z_2)$. 

However, a serious obstacle arises when attempting to design a pull-tight operation for the area functional restricted to $\mathcal H(M,\Z_2)$.  In the unconstrained Almgren-Pitts min-max scheme, one first obtains a weak solution $V$ in the space of varifolds $\mathcal V(M)$ as a limit of a  min-max sequence.  The pull-tight operation is certain pseudo-gradient flow that ensures that this weak solution $V$ is stationary. The fact that $V$ is stationary (or really just that $V$ has bounded first variation) is then used in an essential way to prove the regularity of $V$. 

In our case, the volume constraint does not make sense on the space of varifolds. Nevertheless, we can replace the space of varifolds by Almgren's $\vz$ space, as in \cite{wang2023existence}. The space $\vz(M,\Z_2)$ consists of pairs $(V,T)\in \mathcal V(M)\times \mathcal B(M,\Z_2)$ such that there is a sequence $T_i \in \mathcal B(M,\Z_2)$ with $T_i \to T$ and $\vert T_i\vert \to V$. In particular, $\vz(M,\Z_2)$ is the natural space in which to consider the limit of a min-max sequence. The volume constraint makes sense in $\vz(M,\Z_2)$ and so we can look for a weak solution $(V,T) \in \vz(M,\Z_2)$ with $T = \bd \Omega$ and $\vol(\Omega) = \frac{1}{2}\vol(M)$. 

It is at this point the difficult arises. One would like to design a pull-tight operation which ensures that the weak solution $(V,T)\in \vz(M,\Z_2)$ with $T = \bd \Omega$ is stationary for the area functional with half-volume constraint, i.e., 
\begin{equation}
\label{equation:eq1} 
\delta V(X) = 0, \quad \text{for all } X\in \mathfrak X(M) \text{ with } \int_\Omega \div(X) = 0.
\end{equation}
Here $\mathfrak X(M)$ denotes the space of $C^1$ vector fields on $M$. It does not seem straightforward to achieve this. Moreover, even if such a pull-tight operation does exist, the condition (\ref{equation:eq1}) does not seem sufficient to prove regularity. Indeed, condition (\ref{equation:eq1}) gives no a priori bound on the first variation of $V$.

Thus we adopt a different approach. Rather than performing min-max theory for the area functional directly on $\mathcal H(M,\Z_2)$, we first perform min-max on all of $\mathcal B(M,\Z_2)$ with a functional $E$ given by the area plus a term that penalizes the distance to the space of half-volume cycles. Let $\hv = \frac{1}{2}\vol(M)$. Then, given $k\in \N$, define 
\[
E_k(T) = \M(T) + k\left(\vol(\Omega) - \hv\right)^2
\]
where $\Omega\in \C(M)$ satisfies $\bd \Omega = T$.  It is easy to see that $E_k(T)$ does not depend on the choice of $\Omega$.  Note that if $\Sigma = \bd \Omega$ is a smooth critical point of $E_k$ then $\Sigma$ has constant mean curvature 
\begin{equation}
\label{equation:eq2}
\vert H_\Sigma\vert  = 2 k \left\vert\vol(\Omega) - \hv\right\vert \le k \vol(M).  
\end{equation}
Our idea to prove Theorem \ref{main} is to first use min-max methods to construct a critical point of $E_k$ for each $k\in \N$, and then to take a limit of these critical points as $k \to \infty$.

\subsubsection{Min-Max Theory} 
Now assume $3\le n+1 \le 7$. Fix some $k\in \N$. We would like to use min-max theory to find a critical point of $E_k$. We employ Almgren's $\vz$ space $\vz(M,\Z_2)$ as discussed above. Note that the $E_k$ functional extends naturally to $\vz(M,\Z_2)$ and that
we can define the first variation $\delta E_k|_{(V,T)}$ at a point $(V,T)\in \vz(M,\Z_2)$. As a first step, we obtain a weak solution $(V,T)\in \vz(M,\Z_2)$ as a limit of a min-max sequence. We then design a suitable pull-tight argument which ensures that 
$
\delta E_k|_{(V,T)} = 0.
$
As suggested by (\ref{equation:eq2}), one can then verify that $V$ has $k\vol(M)$-bounded first variation. Thus working with $E_k$ has the major advantage that $V$ has a priori bounded first variation. 

The next step in the usual Almgren-Pitts construction is to use a combinatorial argument to show that some weak solution $V$ is almost minimizing in annuli for the area functional. This almost minimizing property is defined in terms of $(\eps,\delta)$-deformations which, roughly speaking, are deformations which decrease the area of a cycle by at least $\eps$ without ever increasing it by more than $\delta$.  Because the combinatorial argument involves simultaneously applying many different localized $(\eps,\delta)$-deformations to a cycle $T$, it is essential at this point that area functional is {\it local}. In other words, given an open set $U\subset M$ and two cycles $S,T \in \B(M,\Z_2)$ which agree outside of $U$, one has 
\[
\M(T) - \M(S) = \M(T\rest U) - \M(S\rest U). 
\]
This property is used to ensure that if a cycle $T$ is modified in several disjoint open sets, the total change in the mass of $T$ is equal to the sum of the local changes. 

The $E_k$ functional, on the other hand, is {\it non-local}. If a cycle $T = \bd \Omega$ is modified in several disjoint open sets, the individual modifications may interact in a complicated way through the volume term. To circumvent this, we observe that the $E_k$ functional still satisfies a certain quasi-locality property. Namely, if $T$ is modified in several disjoint open sets, and each modification individually changes area by much more than it changes volume, then the net change in $E_k$ resulting from performing all the modifications is approximately additive. Based on this, we define a define a new notion of almost-minimizing for the $E_k$ functional. Heuristically, in our  definition, we consider only those $(\eps,\delta)$-deformations which additionally change the area by much more than the volume at all times. 

This new almost minimizing property still suffices to construct replacements. The first step is to solve a suitable constrained minimization problem. Because $(\eps,\delta)$-deformations which do not change the volume are always admissible, solutions to the constrained minimization problem are smooth and volume preserving stable for the area functional. By a result of Bellettini-Chodosh-Wickramasekera \cite{bellettini2019curvature}, the volume preserving stability implies a curvature estimate. This is then enough to construct a replacement as a limit of solutions to the constrained minimization problem. The smooth part of the replacement is either a multiplicity one CMC with non-zero mean curvature, or a collection of minimal hypersurfaces with multiplicities. 
Finally, one uses replacements in overlapping annuli to prove the regularity of $(V,T)$. Here the crucial point is to show that replacements in overlapping annuli will either both be minimal, or both have the same non-zero constant mean curvature. Using this, we are able to show that the support of $T$ is a (possibly minimal) CMC bounding a region $\Omega$, and that the remainder of $\supp \|V\|$ consists of minimal hypersurfaces. 

In fact, our min-max idea can handle more general functions of volume with very little extra work. See Section \ref{section:preliminaries} for precise definitions. 

\begin{theorem}
\label{E-min-max}
Let $(M^{n+1},g)$ be a closed Riemannian manifold of dimension $3\le n+1\le 7$.  Let $f\colon [0,\vol(M)] \to \R$ be a smooth function satisfying 
\[
f\left(\hv-t\right) = f\left(\hv + t\right)
\]
for all $t\in [0,\hv]$ where $\hv = \frac{1}{2}\vol(M)$. Define $E\colon \B(M,\Z_2)\to \R$ by 
\[
E(T) = \M(T) + f(\vol(\Omega))
\]
where $\Omega$ satisfies $\bd \Omega = T$. Let $\Pi$ be the $X$-homotopy class of a map $\Phi_0\colon X\to (\B(M,\Z_2),\mathbf F)$ and assume that $L^E(\Pi) > 0$. Then there exists $(V,T)\in \vz(M,\Z_2)$ with $E(V,T) = L^E(\Pi)$. 
Choose a set $\Omega$ with $\bd \Omega = T$ and let $H = -f'(\vol(\Omega))$. 
\begin{itemize}
    \item[(i)] If $H \neq 0$, then there exists a smooth, almost-embedded, (not necessarily connected) closed CMC hypersurface $\Lambda = \bd \Omega$, which has mean curvature $H$ with respect to the normal pointing into $\Omega$. Moreover, there exist a (possibly empty) collection of closed embedded minimal hypersurfaces $\Sigma_1,\hdots,\Sigma_k$ and a collection of multiplicities $m_1,\hdots,m_k\in \N$ such that 
\[
V =  \vert \Lambda\vert + \sum_{i=1}^k m_i \vert \Sigma_i\vert.
\]
The hypersurfaces $\Lambda,\Sigma_1,\hdots,\Sigma_k$ are all disjoint.
\item[(ii)] If $H=0$, then there exists a collection of closed, embedded minimal hypersurfaces $\Lambda_1,\hdots,\Lambda_q$ such that $\bd \Omega = \Lambda_1\cup \hdots \cup \Lambda_q$. Moreover, there exist a (possibly empty) collection of closed, embedded minimal hypersurfaces $\Sigma_1,\hdots,\Sigma_k$ and a collection of multiplicities $\ell_1,\hdots, \ell_q$, $m_1,\hdots,m_k\in \N$ such that 
\[
V = \sum_{i=1}^q \ell_i\vert \Lambda_i\vert + \sum_{i=1}^k m_i \vert \Sigma_i\vert.
\]
The hypersurfaces $\Lambda_1,\hdots,\Lambda_q,\Sigma_1,\hdots,\Sigma_k$ are all disjoint.
\end{itemize}
\end{theorem}

\subsubsection{Regularization} At this point, we employ a prescribed mean curvature regularization to obtain improved regularity as in \cite{zhou2020multiplicity}. Again consider a fixed $k\in \N$. Given a smooth function $h\colon M\to \R$ and $\eps > 0$, define $F_{k,\eps}\f \C(M)\to \R$ by
\[
F_{k,\eps}(\Omega) = \M(\bd \Omega) - \eps \int_\Omega h + k\left(\vol(\Omega)-\hv\right)^2. 
\]
For a good choice of $h$, we use similar arguments to find critical points of $F_{k,\eps}$. Analogous to the PMC min-max theory \cite{zhou2020existence}, these critical points are smooth, almost-embedded multiplicity one hypersurfaces $\Sigma = \bd \Omega$ with mean curvature $H = \eps h + h_0$, where $h_0 = -2k(\vol(\Omega) - \hv)$ is a constant. 

As above, we can actually handle more general functionals with very little extra effort. See Section \ref{section:preliminaries} for precise definitions. In particular, see Definition \ref{property (T)} for the definition of property (T), which is a condition on the mean curvature of the level sets of $h$ that ensures unique continuation. It is similar to, but stronger than, condition ($\dagger$) in \cite{zhou2020existence}.
\begin{theorem}
\label{F-min-max} 
Let $(M^{n+1},g)$ be a closed Riemannian manifold of dimension $3\le n+1\le 7$. Let $f\colon [0,\vol(M)]\to\R$ be an arbitrary smooth function, and let $h\colon M\to \R$ be a smooth Morse function satisfying property $\operatorname{(T)}$; see Definition \ref{property (T)}. 
Define $F\colon \C(M)\to \R$ by 
\[
F(\Omega) = \M(\bd \Omega) - \int_\Omega h + f(\vol(\Omega)).
\]
Let $\Pi$ be the $(X,Z)$-homotopy class of a map $\Phi_0\colon X\to (\mathcal C(M),\mathbf F)$. Assume that 
\[
L^F(\Pi) > \sup_{z\in Z} F(\Phi_0(z)).
\]
Then there exists a smooth, almost-embedded hypersurface $\Sigma = \bd \Omega$ satisfying 
\[
L^F(\Pi) = F(\Omega).
\]
The mean curvature of $\Sigma$ is given by $H = h - f'(\vol(\Omega))$, and the touching set of $\Sigma$ is contained in a countable union of $(n-1)$-dimensional manifolds. 
\end{theorem} 

\subsubsection{Half-Volume CMCs} 
Fix some $p\in \N$. Consider a sequence of half-volume $p$-sweepouts $\Phi_k:X_k \to \mathcal H(M,\Z_2)$ with 
\[
\lim_{k\to \infty}\left[ \sup_{x\in X_k} \M(\Phi_k(x))\right] \to \tilde \omega_p.
\]
Let $\Pi_k$ be the homotopy class of $\Phi_k$. Then 
\[
L^{E_k}(\Pi_k)\to \tilde \omega_p. 
\]
We can apply the min-max theory for $E_{k}$ in the homotopy class of $\Pi_k$. We then follow the ideas in \cite{zhou2020multiplicity}. After regularizing to $F_{k,\eps}$, we obtain smooth, multiplicity one hypersurfaces $\Sigma_{k,\eps} = \bd \Omega_{k,\eps}$ with 
\[
F_{k,\eps}(\Omega_{k,\eps}) \to L^{E_k}(\Pi_k), \quad\text{as } \eps \to 0. 
\]
Moreover, the hypersurfaces $\Sigma_{k,\eps}$ have mean curvature $H(\Sigma_{k,\eps}) = \eps h + h_{k,\eps}$ where $h_{k,\eps}$ is a constant.  

Sending $\eps \to 0$, we obtain weak convergence $\Sigma_{k,\eps}\to V_k$ as varifolds and $\Omega_{k,\eps}\to \Omega_k$ as Caccioppoli sets. In fact, for fixed $k$, the surfaces $\Sigma_{k,\eps}$ all satisfy a certain almost minimizing property, and this is enough to obtain the regularity of $V_k$. This fact has already been observed in \cite{li2023existence} and \cite{wang2023existence}. 
In our setting, there is a subtle extra challenge in this step which we address in Section \ref{S:compactness}. 
It follows that $V_k$ is induced by either 
\begin{itemize}
\item[(i)] an almost-embedded, multiplicity one CMC hypersurface $\Sigma_k = \bd \Omega_k$ with non-zero mean curvature, or
\item[(ii)] a collection of smooth, embedded minimal hypersurfaces with multiplicities. 
\end{itemize}
Moreover, in case (ii), some subcollection of the minimal hypersurfaces bounds $\Omega_k$. 
Note that we do not a priori rule out case (ii). The primary benefit of the regularization is that it rules out the possibility of obtaining a mixture of CMC components with non-zero mean curvature and minimal components.

Next, we can send $k\to \infty$, and obtain convergence $V_k\to V$ as varifolds and $\Omega_k\to \Omega$ as Caccioppoli sets. It is straightforward to check that $\|V\|(M) \le \tilde \omega_p$ and $\vol(\Omega) = \hv$. We would like to show that $V$ is induced by a multiplicity one CMC. Again, each $V_k$ satisfies a certain almost-minimizing property. However, unlike for $\Sigma_{k,\eps}$, the varifolds $V_k$ do not have a uniform a priori bound on their first variation. If we can show that the first variation stays bounded, then we can use the almost minimizing property as before to prove regularity of $V$. In this case, $V$ is induced by either 
\begin{itemize}
\item[(i)] an almost-embedded, multiplicity one CMC hypersurface $\Sigma = \bd \Omega$ with non-zero mean curvature, or
\item[(ii)] a collection of smooth, embedded minimal hypersurfaces with multiplicities. 
\end{itemize}
In the second case, some subcollection of the minimal hypersurfaces bounds $\Omega$. We are able to show that, for a generic metric $g$, no collection of minimal hypersurfaces bounds a region with volume equal to $\hv$. Thus, for generic metrics, the second possiblity cannot occur, and we obtain a multiplicity one CMC hypersurface $\Sigma = \bd \Omega$ enclosing half the volume of $M$ and satisfying $\area(\Sigma) \le \tilde \omega_p$.
 
To show that the mean curvature $H_k$ of $\Sigma_k = \supp \|V_k\|$ is uniformly bounded, we can argue as follows.   
Supposing instead that $H_k\to \infty$, we can rescale the surfaces $\Sigma_k$ to obtain new surfaces $\Sigma_k'$ with constant mean curvature 1. Moreover, we can use the almost minimizing property to ensure that the surfaces $\Sigma_k'$ are volume preserving stable on a large set. It is at that this point that we need to further restrict to dimensions $3\le n+1\le 5$.  In these dimensions, we can take advantage of the fact that stable 1-CMC surfaces are known to have diameter upper bounds by the work of Elbert-Nelli-Rosenberg \cite{elbert2007stable}. 
This implies that $\Sigma_k'$ resembles the union of many unit balls when $k$ is large. At the original scale, this means that $\Sigma_k$ is the union of many tiny components. However, this structure is incompatible with the half-volume constraint and the uniform area bound. 

To complete the proof, it remains to verify the energy identity $\area(\Sigma) = \tilde \omega_p$. Since we know $E_k(V_k,\bd \Omega_k)\to \tilde \omega_p$, it is equivalent to show that 
\[
k(\vol(\Omega_k)-\hv)^2 \to 0, \quad \text{as } k\to \infty. 
\]
In fact, this also follows from the fact that the mean curvature $H_k$ is uniformly bounded. Indeed, since $(V_k,\bd \Omega_k)$ is stationary for $E_k$, we have 
$
\vert H_k\vert = 2 k \vert \vol(\Omega_k)-\hv\vert. 
$
Hence
\[
k(\vol(\Omega_k)-\hv)^2 = \frac{\vert H_k\vert^2}{4 k},
\]
and this goes to 0 as $k\to \infty$, as needed.

\subsection{Outline of the paper}
We collect preliminary materials in Section \ref{section:preliminaries}. Section \ref{S:relative min-max for F} and Section \ref{S:absolute min-max for E} are respectively devoted to the min-max theory for the $F$ and $E$ functionals, and the proof of Theorem \ref{F-min-max} and Theorem \ref{E-min-max}. In Section \ref{S:compactness}, we prove several compactness results for min-max solutions associated with the $E$ and $F$ functionals. In Section \ref{S:constructiong half-volume CMCs}, we apply the min-max theory for $E$ and $F$ and associated compactness results to prove our main Theorem \ref{main} for half volume spectrum as well as Theorem \ref{positive-Ricci}. In Appendix \ref{h-generic}, we prove that condition (T) is $C^\infty$ generic among Morse functions, and in Appendix \ref{generic metrics}, we prove a collection of results about closed CMC hypersurfaces in a generic metric.

\subsection*{Acknowledgement} 
L.M. acknowledges the support of an AMS Simons Travel Grant. X.Z. acknowledges the support by NSF grant DMS-1945178, an Alfred P. Sloan Research Fellowship, and a grant from the Simons Foundation (1026523, XZ). X.Z. also thanks the support of the math department at Princeton University and the Institute of Advanced Study, where part of the work was carried out.

\section{Preliminaries} 
\label{section:preliminaries} 

In this section, we begin by reviewing some concepts from geometric measure theory needed in the paper. Then we introduce the functionals we will consider and related min-max notions. 

We will need the following concepts from geometric measure theory; see \cite{simon1983lectures}.  Let $(M^{n+1},g)$ be a closed Riemannian manifold. 
\begin{itemize}
\item Let $\C(M)$ denote the space of all Caccioppoli sets in $M$. 
\item Let $\mathcal V(M)$ denote the space of all $n$-dimensional varifolds on $M$. 
\item Let $\z(M,\Z_2)$ denote the space of $n$-dimensional flat cycles in $M$ mod 2. 
\item Given $\Omega \in \C(M)$, the notation $\bd \Omega$ denotes the element of $\z(M,\Z_2)$ induced by the boundary of $\Omega$. 
\item Given $T \in \z(M,\Z_2)$, the notation $\vert T\vert$ stands for the varifold induced by $T$. 
\item Let $\B(M,\Z_2)$ denote the space of all $T\in \z(M,\Z_2)$ such that $T = \bd \Omega$ for some $\Omega \in \C(M)$. 
\item Let $\h(M,\Z_2)$ denote the space of all $T\in \z(M,\Z_2)$ such that $T = \bd \Omega$ for some $\Omega \in \C(M)$ with $\vol(\Omega) = \frac 1 2 \vol(M)$. 
\item We use $\mathcal F$ to denote the flat topology, $\mathbf F$ to denote the $\mathbf F$-topology, and $\mathbf M$ to denote the mass topology.  By convention, the $\mathbf F$ topology on $\C(M)$ is 
\[
\mathbf F(\Omega_1,\Omega_2) = \mathcal F(\Omega_1,\Omega_2) + \mathbf F(\vert \bd \Omega_1\vert,\vert \bd \Omega_2\vert).
\]
\item Let $\vz(M,\Z_2)$ denote Almgren's $\vz$ space (see \cite{almgren1965theory}).
\item Let $\vc(M)$ denote Almgren's $\vc$ space (see \cite{almgren1965theory} and \cite{wang2023existence}).  
\end{itemize} 

Because the reader is likely less familiar with Almgren's $\vz$ \cite{almgren1965theory} and $\vc$ spaces \cite[Definition 1.3]{wang2023existence}, we include the relevant background below.  The set $\vz(M,\Z_2)$ consists of all pairs $(V,T) \in \mathcal V(M) \times \mathcal \B(M,\Z_2)$ such that there is a sequence $T_k \in \mathcal \B(M,\Z_2)$ with $\vert T_k\vert \to V \in  \V(M)$ and $T_k\to T \in \mathcal B(M,\Z_2)$.  Note that it may or may not be true that $V = \vert T \vert$, but it is always true that $\|\, \vert T\vert \, \|\leq \|V\|$ as measures.  We can endow $\vz(M, \Z_2)$ with the product metric, so that for any $(V, T), (V', T')\in \vz(M, \Z_2)$, the $\mathscr{F}$-distance between them is
\[ \mathscr{F}\big( (V, T), (V', T') \big) = \mathbf{F}(V, V') + \mathcal{F}(T, T'). \]
We will use $\vz(M, \mathscr{F}, \Z_2)$ if we wish to emphasize the metric $\mathscr{F}$.

Given $(V,T)\in \vz(M,\Z_2)$ and a $C^1$ diffeomorphism $\phi \f M\to M$, define the pushforward $\phi_\sharp(V,T) = (\phi_\sharp V, \phi_\sharp T)$. It is easy to check that $\phi_\sharp (V,T) \in \vz(M,\Z_2)$ and that the mapping $\phi_\sharp\f \vz(M,\Z_2)\to \vz(M,\Z_2)$ is continuous.  

\begin{prop}
Given any constant $L > 0$, the space
\begin{gather}\label{eq:Y bounded mass pairs}
    Y_{z, L} = \{(V,T)\in \vz(M,\Z_2): \|V\|(M) \le L\}
\end{gather}
is a compact metric space. 
\end{prop} 

\begin{proof}
The flat topology on $\mathcal B(M,\Z_2)$ is metrizable.  Moreover, the weak topology on the space $X = \{V\in \V(M): \|V\|(M)\le L\}$ is also metrizable.  Therefore $Y$ is a metric space. To see that $Y$ is compact, assume that $(V_i,T_i)$ is a sequence in $Y$.  Since $X$ is compact, by passing to a subsequence if necessary, we can assume that $V_i \to V$ for some $V \in X$.  Since $\M(T_i) \le L$ for all $i\in \N$, by passing to a further subsequence, we can assume that $T_i \to T$ for some $T \in \B(M,\Z_2)$. 
It remains to show that $(V,T)\in \vz(M,\Z_2)$.  For each $i\in \N$, there is a sequence $T_{i,k} \in \B(M,\Z_2)$ such that $T_{i,k} \to T_i$ and $\vert T_{i,k}\vert \to V_i$ as $k\to \infty$.  Since all spaces in question are metrizable, the diagonal sequence $T_{i,i}$ therefore satisfies $T_{i,i} \to T$ and $\vert \bd T_{i,i}\vert \to V$ as $i\to \infty$.  Thus $(V,T) \in \vz(M,\Z_2)$, and it follows that $Y$ is compact. 
\end{proof}

The VC space is entirely analogous but with $\mathcal B(M,\Z_2)$ replaced by $\C(M)$. The set $\vc(M)$ consists of all pairs $(V,\Omega) \in \mathcal V(M) \times \mathcal C(M)$ such that there is a sequence $\Omega_k \in \mathcal C(M)$ with $\vert \bd \Omega_k\vert \to V \in  \V(M)$ and $\Omega_k\to \Omega \in \mathcal C(M)$.    We similarly endow $\vc(M)$ with the product metric, so that for any $(V, \Omega), (V', \Omega')\in \vc(M)$, the $\mathscr{F}$-distance between them is
\[ \mathscr{F}\big( (V, \Omega), (V', \Omega') \big) = \mathbf{F}(V, V') + \mathcal{F}(\Omega, \Omega'). \]
Again, we will write $\vc(M, \mathscr{F})$ if we wish to emphasize the metric $\mathscr{F}$. Given $(V,\Omega)\in \vc(M)$ and a $C^1$ diffeomorphism $\phi \f M\to M$, define the pushforward $\phi_\sharp(V,\Omega) = (\phi_\sharp V, \phi_\sharp \Omega)$. It is easy to check that $\phi_\sharp (V,\Omega) \in \vc(M)$ and the mapping $\phi_\sharp\f \vc(M)\to \vc(M)$ is continuous.  

\begin{prop}
Given any constant $L > 0$, the space
\begin{gather}\label{eq:Y bounded mass pairs2}
    Y_{c, L} = \{(V,\Omega)\in \vc(M): \|V\|(M) \le L\}
\end{gather}
is a compact metric space. 
\end{prop}

\subsection{The E Functional} 

Fix a closed, Riemannian manifold $M$. 

\begin{defn}
Let $\hv = \frac{1}{2}\vol(M)$. 
\end{defn}

Let $f\colon [0,\vol(M)]\to\R$ be a non-constant smooth function and assume that $f$ is ``even'' in the sense that 
\begin{equation}
\label{f is even}
f(\hv - t) = f(\hv + t)
\end{equation}
for all $t\in [0,\hv]$. Given $T\in \B(M,\Z_2)$, we write $f(\vol(T))$ to mean $f(\vol(\Omega))$ where $\Omega\in \C(M)$ satisfies $\bd \Omega = T$. This is well-defined since $\vert \vol(\Omega) - \hv\vert = \vert \vol(M\setminus \Omega)-\hv\vert$ and $f$ satisfies (\ref{f is even}). 

\begin{defn}
Define the function $E^f\colon \B(M,\Z_2)\to \R$ by 
$
E^f(T) = \M(T) + f(\vol(T)).
$
In the following, we will sometimes just write $E$ instead of $E^f$ if the choice of $f$ is clear. 
\end{defn}

Assume $\Sigma = \bd \Omega$ is a smooth hypersurface in $M$.  Let $\nu$ denote the inward pointing normal vector to $\Sigma$, and let $H$ be the mean curvature of $\Omega$ with respect to $\nu$.  Let $X$ be a $C^1$ vector field on $M$ and let $\phi_t$ be the associated flow.  

\begin{prop}
The first variation of $E$ is given by
\[
\delta E|_\Sigma(X) = \frac{d}{dt}\eval_{t=0} E(\phi_t(\Sigma)) = -\int_\Sigma H \la X,  \nu \ra - f'(\vol(\Omega)) \int_{\Sigma} \la X,\nu\ra.
\]
Thus $\Sigma$ is critical for $E$ if and only if $\Sigma$ has constant mean curvature $H = -f'(\vol(\Omega))$. 
\end{prop}

Note that the $E$ functional extends to a functional on $\vz(M,\Z_2)$ by 
\[
E(V,T) = \|V\|(M) + f(\vol(T)).
\]
Moreover, if $T_i\in \z(M,\Z_2)$ is a sequence with $T_i\to T$ and $\vert T_i\vert \to V$ then $E(T_i) \to E(V,T)$.  It is also possible to define the first variation of $E$ on $\vz(M,\Z_2)$.

\begin{prop} 
The first variation of $E$ on $\vz(M,\Z_2)$ is given by 
\begin{align*}
\delta E|_{(V,T)}(X) &= \frac{d}{dt}\eval_{t=0} E((\phi_t)_\sharp (V,T)) \\ &= \delta V(X) + f'(\vol(\Omega)) \int_{\Omega}\div(X) \numberthis \label{eq:1st variation for E}
\end{align*}
where $\Omega \in \C(M)$ satisfies $\bd \Omega = T$. 
\end{prop} 
A pair $(V,T)\in \vz(M,\Z_2)$ is called stationary for $E$ if $\delta E|_{(V,T)}(X) = 0$ for all $C^1$ vector fields $X$ on $M$. 

\begin{prop}
\label{prop:continuity of first variation of E}
Fix a $C^1$ vector field $X$ on $M$. Assume that $(V_i,T_i) \to (V,T) \in \vz(M,\Z_2)$. Then 
$
\delta E_{(V_i,T_i)}(X) \to \delta E_{(V,T)}(X). 
$
\end{prop}

\begin{proof}
Let $(V_i,T_i)$ and $(V,T)$ be as in the statement of the theorem.  It is possible to choose Caccioppoli sets $\Omega_i$ and $\Omega$ so that $\bd \Omega_i = T_i$ and $\bd \Omega = T$ and $\Omega_i \to \Omega$ as Caccioppoli sets. Since $X$ is fixed, we have 
\[
\int_{G_n(M)} \div_P X(x)\, dV_i(x,P)  \to \int_{G_n(M)} \div_P X(x)\, dV(x,P) 
\]
as $i \to \infty$.  
Again since $X$ is fixed we have $\vol(\Omega_i) \to \vol(\Omega)$ and 
\[
\int_{\Omega_i} \div(X) \to \int_{\Omega} \div(X). 
\]
The result now follows from the first variation formula. 
\end{proof}

\begin{corollary}
The set $\{(V,T)\in \vz(M,\Z_2):\, (V,T) \text{ is stationary for } E\}$ is closed as a subset of $\vz(M,\Z_2)$. 
\end{corollary}

\subsection{The F Functional} 
\label{the F functional}

Fix a smooth, non-constant function $f\colon [0,\vol(M)]\to \R$. 
Also fix a smooth Morse function $h\colon M\to \R$.  

\begin{defn} Given a regular point $x\in M$ for $h$, let $\Gamma(x)$ be the level set of $h$ passing through $x$. Then define $v(h,x)$ to be the vanishing order at $x$ of the mean curvature $H_{\Gamma(x)}$, regarded as a function on $\Gamma(x)$.   
\end{defn} 

\begin{defn}
\label{property (T)} Let $h\colon M\to \R$ be a smooth Morse function. We say that $h$ satisfies property (T) provided 
\begin{itemize}
\item[(T)] For every regular point $x$ of $h$, we have $v(h,x) < \infty$. 
\end{itemize}
\end{defn}
In the following, fix a smooth Morse function $h\colon M\to \R$ satisfying property (T). In Appendix \ref{h-generic}, we show that the set of smooth Morse functions satisfying property (T) is dense in $C^\infty(M)$.

\begin{defn} Define the functional $A^h\colon \C(M) \to \R$ by 
\[
A^h(\Omega) = \M(\bd \Omega) - \int_\Omega h. 
\]
Then define the functional $F^{h,f}\f \C(M) \to \R$ by 
\[
F^{h,f}(\Omega) = A^h(\Omega) + f(\vol(\Omega)). 
\]
In the following, we will sometimes just write $F$ instead of $F^{h,f}$ if the choice of $h$ and $f$ is clear. 

\end{defn}

Assume $\Sigma = \bd \Omega$ is a smooth hypersurface in $M$.  Let $\nu$ denote the inward pointing normal vector to $\Sigma$, and let $H$ be the mean curvature of $\Sigma$ with respect to $\nu$.  Let $X$ be a $C^1$ vector field on $M$ and let $\phi_t$ be the associated flow.  

\begin{prop}
The first variation of $F$ is given by
\[
\delta F|_\Omega(X) = \frac{d}{dt}\eval_{t=0} F(\phi_t(\Omega)) = \int_\Sigma \big(h-H - f'(\vol(\Omega)\big) \la X, \nu \ra. 
\]
Thus $\Omega$ is a critical point for $F$ if and only if $\Sigma$ has constant mean curvature given by $H = h - f'(\vol(\Omega))$. 
\end{prop} 

Property (T) is used to control the touching set of almost-embedded hypersurfaces with mean curvature $h + h_0$ for some constant $h_0$. 

\begin{prop}\label{prop:touching set}
    Assume that $h$ satisfies property $\operatorname{(T)}$. Let $\Sigma$ be an almost-embedded hypersurface with mean curvature $H = h + h_0$ for some constant $h_0$. Then the touching set $S(\Sigma)$ is contained in a countable union of $(n-1)$-dimensional manifolds. 
\end{prop}

\begin{proof}
    For every regular point $x\in M$ for $h$, the proof of \cite[Theorem 3.11]{zhou2020existence} shows that there is some $r > 0$ for which $\mathcal S(\Sigma) \cap B_r(x)$ is contained in a countable union of $(n-1)$-dimensional manifolds. Let $C$ be the finite set of all critical points of $h$. Let $G_n$ be a sequence of closed subsets of $M\setminus C$ with $\cup_{n=1}^\infty G_n = M\setminus C$. Then each $G_n$ is covered by finitely many balls $B$ with the property that $\mathcal S(\Sigma)\cap B$ is contained in a countable union of $(n-1)$-dimensional submanifolds. Taking the union of these balls over all $n$, it follows that $\mathcal S(\Sigma)\cap (M\setminus C)$ is contained in a countable union of $(n-1)$-dimensional manifolds. Of course, the finite set $C$ is also contained in a finite union of $(n-1)$-dimensional manifolds, and the result follows. 
\end{proof}

Note that the $F$ functional extends to a functional on $\vc(M)$ by 
\[
F(V,\Omega) = \|V\|(M) - \int_\Omega h + f(\vol(\Omega)).
\]
Moreover, if $\Omega_i\in \C(M)$ is a sequence with $\Omega_i\to \Omega$ and $\vert \bd\Omega_i\vert \to V$ then $F(\Omega_i) \to F(V,\Omega)$.  It is also possible to define the first variation of $F$ on $\vc(M)$.

\begin{prop}
The first variation of $F$ on $\vc(M)$ is given by 
\begin{align*}
\delta F|_{(V,\Omega)}(X) &= \frac{d}{dt}\eval_{t=0} F((\phi_t)_\sharp (V,\Omega)) \\ &= \delta V(X) - \int_\Omega \div(hX) + f'(\vol(\Omega)) \int_{\Omega}\div(X). \numberthis \label{eq:1st variation for F}
\end{align*}
\end{prop} 
A pair $(V,\Omega)\in \vc(M)$ is called stationary for $F$ if $\delta F_{(V,\Omega)}(X) = 0$ for all $C^1$ vector fields $X$ on $M$. 

\begin{prop}
\label{prop:continuity of first variation of F}
Fix a $C^1$ vector field $X$ on $M$. Assume that $(V_i,\Omega_i) \to (V,\Omega) \in \vc(M)$. Then 
$
\delta F_{(V_i,\Omega_i)}(X) \to \delta F_{(V,\Omega)}(X). 
$
\end{prop}

\begin{corollary}
The set $\{(V,\Omega)\in \vc(M):\, (V,\Omega) \text{ is stationary for } F\}$ is closed in $\vc(M)$. 
\end{corollary}

\subsection{Min-Max Notions} 

Let $X$ be a cubical subcomplex of the cell complex $I(m, k_0)$ for some $m, k_0 \in \mathbb N$. Here $I(m, k)=I(1, k)\otimes \cdots I(1, k)$ ($m$-times), where $I(1, k)$ denotes the complex on $I=[0,1]$ whose $1$-cells and $0$-cells are, respectively,
\[ [1, 3^{-k}], [3^{-k}, 2\cdot 3^{-k}], \cdots, [1-3^{-k}, 1] \text{ and } [0], [3^{-k}], \cdots, [1-3^{-k}], [1].\]
We refer to \cite[Appendix A]{zhou2020multiplicity} for a summary of notions; (see also \cite[Section 2.1]{marques2017existence}).  We will define both absolute and relative homotopy classes. 

\subsubsection{Absolute Homotopy Classes} 

Given a continuous map $\Phi_0: X\to (\mathcal B(M,\Z_2),\mathbf F)$, we let $\Pi$ be the collection of all sequences of continuous maps $\{\Phi_i: X\to (\B(M,\Z_2),\mathbf F)\}$ such that, for each $i$, there exists a flat continuous homotopy map 
\begin{gather*}
H_i\colon X\times [0,1] \to (\B(M,\Z_2),\mathcal F),\\
H_i(x,0) = \Phi_0(x), \\
H_i(x,1) = \Phi_i(x).
\end{gather*}

\begin{defn}
Such a sequence $\{\Phi_i\}_{i\in\N}$ is called an $X$-homotopy sequence of mappings into $\B(M,\Z_2)$, and $\Pi$ is called the $X$-homotopy class of $\Phi_0$.
\end{defn}

\begin{defn}
    Fix a functional $E\f \B(M,Z_2)\to \R$. Define the min-max value of $\Pi$ with respect to $E$ by
    \[
    L^{E}(\Pi) = \inf_{\{\Phi_i\}\in \Pi} \limsup_{i\to\infty}\left[\sup_{x\in X} E(\Phi_i(x))\right]. 
    \]
    A sequence $\{\Phi_i\}\in \Pi$ is called a critical sequence for $E$ if 
    \[
    L^{E}(\{\Phi_i\}) := \limsup_{i\to \infty} \left(\sup_{x\in X} E(\Phi_i(x))\right) = L^{E}(\Pi). 
    \]
  The critical set  $\mathcal K(\{\Phi_i\})$ associated to a critical sequence $\{\Phi_i\}$ is the set of all $(V,T)\in \vz(M, \Z_2)$ such that there exist $x_{i_j}\in X$ with  $\vert \Phi_{i_j}(x_{i_j})\vert \to V$, and $\Phi_{i_j}(x_{i_j})\to \Omega$, and $E(\Phi_{i_j}(x_{i_j})) \to L^{E}(\Pi)$.  In the following, we will sometimes abbreviate $L^{E}$ to $L$ when the choice of $E$ is clear.  
\end{defn}

\begin{prop}
    The critical set $\mathcal K(\{\Phi_i\})$ is a compact subset of $\vz(M, \mathscr{F}, \Z_2)$. 
\end{prop}

\subsubsection{Relative Homotopy Classes} 
Let $Z \subset X$ be a cubical subcomplex. 
Given a continuous map $\Phi_0: X\to (\C(M),\mathbf F)$, we let $\Pi$ be the collection of all sequences of continuous maps $\{\Phi_i: X\to (\C(M),\mathbf F)\}$ such that, for each $i$, there exists a flat continuous homotopy map 
\begin{gather*}
H_i\colon X\times [0,1] \to (\C(M),\mathcal F),\\
H_i(x,0) = \Phi_0(x), \\
H_i(x,1) = \Phi_i(x),
\end{gather*}
and, moreover,
\[
\limsup_{i\to \infty} \left[\sup_{(z,t)\in Z\times [0,1]} \mathbf F(H_i(z,t), \Phi_0(z))\right] \to 0
\]
as $i\to \infty$. 

\begin{defn}
Such a sequence $\{\Phi_i\}_{i\in\N}$ is called an $(X,Z)$-homotopy sequence of mappings into $\C(M)$, and $\Pi$ is called the $(X,Z)$-homotopy class of $\Phi_0$.
\end{defn}

\begin{defn}
    Fix a functional $F\f \C(M)\to \R$. Define the min-max value of $\Pi$ with respect to $F$ by
    \[
    L^{F}(\Pi) = \inf_{\{\Phi_i\}\in \Pi} \limsup_{i\to\infty}\left[\sup_{x\in X} F(\Phi_i(x))\right]. 
    \]
    A sequence $\{\Phi_i\}\in \Pi$ is called a critical sequence for $F$ if 
    \[
    L^{F}(\{\Phi_i\}) := \limsup_{i\to \infty} \left(\sup_{x\in X} F(\Phi_i(x))\right) = L^{F}(\Pi). 
    \]
    The critical set  $\mathcal K(\{\Phi_i\})$ associated to a critical sequence $\{\Phi_i\}$ is the set of all $(V,\Omega)\in \vc(M)$ such that there exist $x_{i_j}\in X$ with  $\vert \Phi_{i_j}(x_{i_j})\vert \to V$, $\Phi_{i_j}(x_{i_j})\to \Omega$, and $F(\Phi_{i_j}(x_{i_j})) \to L^{F}(\Pi)$.  In the following, we will sometimes abbreviate $L^{F}$ to $L$ when the choice of $F$ is clear.  
\end{defn}

\begin{prop}
    The critical set $\mathcal K(\{\Phi_i\})$ is a compact subset of $\vc(M, \mathscr{F})$. 
\end{prop}

\subsection{Regularity and Curvature Estimates}

When developing a min-max theory for the $F$ functional, we will need to analyze volume constrained local minimizers of $A^h$.

\begin{prop}
    \label{regularity-for-minimizers}
    Fix $3\le n+1\le 7$. Assume that $\Omega$ is a local minimizer for $A^h$ in an open set $U$ subject to a volume constraint. In other words, for any $\Theta$ with $\supp(\Omega-\Theta)\subset U$ and $\vol(\Theta)=\vol(\Omega)$, we have $A^h(\Theta)\ge A^h(\Omega)$. Then $\bd \Omega$ is smooth and embedded in $U$. 
\end{prop}

\begin{proof}
It suffices to verify condition (2) of \cite[Proposition 3.1]{morgan2003regularity}. As in \cite{morgan2003regularity}, we can find two points $p_1,p_2\in U$ and deformations $\Omega_{i,t}$ such that 
\begin{itemize}
\item[(i)] $\Omega_{i,0}=\Omega$, $\Omega_{i,t}\supset \Omega$ for $t > 0$, and $\Omega_{i,t}\subset \Omega$ for $t<0$
\item[(ii)] $\supp(\Omega_{i,t}-\Omega)$ is contained in a small ball centered at $p_i$,
\item[(iii)] $\vert \mathcal H^{n}(\bd \Omega_{i,t}) - \mathcal H^n(\bd \Omega)\vert \le C_1 \vert \vol(\Omega_{i,t})-\vol(\Omega)\vert$
\end{itemize}
for $i=1,2$. Condition (i) implies that $\vert \vol(\Omega \operatorname{\Delta} \Omega_{i,t}) \vert = \vert \vol(\Omega) - \vol(\Omega_{i,t})\vert$. Therefore condition (iii) implies that 
\begin{align*}
    \vert A^h(\Omega_{i,t})-A^h(\Omega)\vert &\le \vert \mathcal H^n(\bd \Omega_{i,t})-\mathcal H^n(\bd \Omega)\vert + \left\vert \int_\Omega h - \int_{\Omega_{i,t}} h\right\vert \\
    &\le C_1 \vert \vol(\Omega_{i,t})-\vol(\Omega)\vert + (\sup \vert h\vert) \vol(\Omega \operatorname{\Delta} \Omega_{i,t}) \\
    &\le C_2 \vert \vol(\Omega) - \vol(\Omega_{i,t})\vert. 
\end{align*}
Now we can argue as in \cite{zhou2020existence}.

Let $p\in U$ and choose a small number $r > 0$ so that $B_r(p)\subset U$. We can suppose $r$ is small enough that $B_r(p)$ is far away from a neighborhood of $p_i$ for some choice of $i\in \{1,2\}$. Consider any $\Theta$ with $\supp(\Omega-\Theta)\subset B_r(p)$. Then it must be the case that 
\begin{equation}
\label{delta-a}
A^h(\Theta) - A^h(\Omega) \ge -C_2 \vert \vol(\Theta) - \vol(\Omega)\vert. 
\end{equation}
Otherwise, we could restore volume near $p_i$ using the deformation above, and this would create a new competitor with the same volume as $\Omega$ but lower $A^h$. Let $A = \mathcal H^n(\bd \Omega)$ and $A' = \mathcal H^n(\bd \Theta)$. Then (\ref{delta-a}) implies that 
\[
A' - A \ge -C_2 \vert \vol(\Theta)-\vol(\Omega)\vert - \left\vert \int_\Theta h - \int_\Omega h\right\vert \ge -C_3 \vol(\Theta \operatorname{\Delta} \Omega). 
\]
As in \cite{zhou2020existence}, this inequality implies condition (2) of \cite[Proposition 3.1]{morgan2003regularity} provided $r$ is sufficiently small. The regularity then follows. 
\end{proof}

Bellettini, Chodosh, and Wickramasekera proved curvature estimates for volume preserving stable CMCs. In particular, they proved the following Bernstein type theorem.

\begin{theorem}[\cite{bellettini2019curvature} Proposition 3]
\label{Bernstein} 
Fix $3\le n+1\le 7$. Assume that $\Sigma\subset \R^{n+1}$ is a connected, embedded, volume preserving stable minimal hypersurface with no singularities. Further suppose that $\mathcal H^n(\Sigma \cap B_r)\le \Lambda r^n$ for some constant $\Lambda \ge 1$ and all $r\ge 0$. Then $\Sigma$ is a hyperplane. 
\end{theorem}

For the min-max theory, we need to extend the Bellettini-Chodosh-Wickramasekera curvature estimates to local minimizers of $A^h$ with a volume constraint.

\begin{prop}
\label{curvature-estimates}
Fix an open set $U\subset M$. Assume that $\Sigma = \bd \Omega$ is smooth and properly embedded in $U$ with mean curvature $H_\Sigma = h + h_0$ for some constant $h_0$. Assume that $\Omega$ is volume preserving stable for $A^{h+h_0}$ in $U$.  Then there is a constant $C$ depending only on $U$, $\|h\|_{C^3}$, an upper bound for $\area(\Sigma)$, and an upper bound for $\|H_\Sigma\|_\infty$ such that 
\[
\|A_\Sigma\|^2(x) \le \frac{C}{\dist(x,\bd U)^2}
\]
for all $x\in \Sigma\cap U$.
\end{prop}

\begin{proof}
This follows from Theorem \ref{Bernstein} by a standard argument. 
\end{proof}

Finally we record the following Simons type theorem for volume preserving stable cones. 

\begin{prop}
\label{simons}
Assume that $\Sigma$ is a smooth, embedded, volume preserving stable minimal cone in $\R^{n+1}\setminus \{0\}$ with $3\le n+1 \le 7$. Then $\Sigma$ is a hyperplane. 
\end{prop}

\begin{proof}
Note that $\Sigma$ must be strongly stable outside some ball $B(0,R)$. Since $r \Sigma = \Sigma$ for every $r > 0$, it follows that $\Sigma$ is strongly stable outside $B(0,rR)$ for every $r > 0$. Thus $\Sigma$ is strongly stable in $\R^{n+1}\setminus \{0\}$ and the result follows from the usual Simons theorem. 
\end{proof}

\section{Relative Min-Max for F}\label{S:relative min-max for F}

The goal of this section is to develop a relative min-max theory for the $F$ functional. Throughout this section, $M$ is a fixed Riemannian manifold, $f\colon [0,\vol(M)]\to \R$ is an non-constant smooth function, $h\colon M\to \R$ is a smooth Morse function satisfying property (T), and we abbreviate $F = F^{h,f}$.

\subsection{Pull Tight}  
 We will employ a pull-tight operation to ensure that every point in the critical set is stationary for $F$.  The construction is relatively standard, and so we shall be brief.  The goal is to prove the following proposition.

\begin{prop}
\label{F-pt}
Let $\Pi$ be the $(X,Z)$-homotopy class of a continuous map $\Phi_0:X \to \mathcal (\C(M),\mathbf F)$. Assume that 
\[
L^F(\Pi) > \sup_{z\in Z} F(\Phi_0(z)).
\]
Let $\{\Phi_i\}$ be a critical sequence for $\Pi$.  There is another critical sequence $\{\Xi_i\}$ such that $\mathcal K(\{\Xi_i\}) \subset \mathcal K(\{\Phi_i\})$ and every $(V,\Omega)\in \mathcal K(\{\Xi_i\})$ is stationary for $F$. 
\end{prop}

\begin{proof}
Let $L = L^F(\{\Phi_i\}) + 1$. Recall that $Y = Y_{c, L}$ defined in \eqref{eq:Y bounded mass pairs2} is the set of pairs in $\vc$ with mass bounded by $L$. Let 
\[ Y_0 = \{(V, \Omega)\in Y: \text{$(V, T)$ is stationary for $F$}\} \cup \Phi_0(Z). \]
Consider the concentric annuli of $Y_0$ under the $\mathscr{F}$-metric:
\[
\begin{aligned}
    & Y_1 = \{(V, \Omega)\in Y: \mathscr{F}\big((V, T), Y_0\big)\geq \frac{1}{2}\},\\
    & Y_j = \{(V, \Omega)\in Y: \frac{1}{2^j}\leq \mathscr{F}\big((V, \Omega), Y_0\big) \leq \frac{1}{2^{j-1}}\},\, j\in\N,\, j\geq 2.
\end{aligned}
\]
By compactness of $Y_j$ under the $\mathscr{F}$-metric and the continuity of first variation of $(V,\Omega)\to \delta F_{V, \Omega}(\mathcal X)$ for a fixed vector field $\mathcal X$ (Proposition \ref{prop:continuity of first variation of F}), for each $j\in \N$, we can find $c_j>0$, such that for all $(V, \Omega)\in Y_j$, there exists a $C^1$ vector field $\mathcal X_{V, \Omega}$ on $M$ with 
\[ \|\mathcal X_{V, \Omega}\|_{C^1}\leq 1,\quad \delta F_{V, \Omega}(\mathcal X_{V, T})\leq - c_j <0. \]
Next, we can follow the same procedure as in \cite[Step 2 in Section 2.2]{wang2023existence} to construct a mapping
\[ \mathcal X: Y \to \mathfrak X(M),\]
where $\mathfrak X(M)$ denotes the space of $C^1$ vector fields on $M$, such that the following lemma holds.
\begin{lem}
    The map $\mathcal X$ is continuous under the $C^1$ topology on $\mathfrak X(M)$. Moreover, there exist continuous functions $g\colon (0,\infty)\to(0,\infty)$ and $\rho\colon(0,\infty)\to(0,\infty)$ with $\lim_{t\to 0} g(t) = 0$ and $\lim_{t\to 0} \rho(t) = 0$ such that 
    \[
    \delta F_{V',\Omega'}(\mathcal X(V,\Omega)) \le - g(\mathscr F((V,\Omega),Y_0))
    \]
    for all $(V',T'),(V,T)\in Y$ satisfying $\mathscr F\big((V',\Omega'),(V,\Omega)\big) \le \rho\big(\mathscr F((V,\Omega),Y_0)\big)$.
\end{lem}

The next step is to construct a map into the space of isotopies. Again this can be done as in \cite[Step 3 in Section 2.2]{wang2023existence}. Given $(V,\Omega)\in Y$ and $t > 0$, define $(V_t,\Omega_t) = \phi_{V,\Omega}(t)_\sharp (V,\Omega)$ where $\phi_{V,\Omega}(t)$ denotes the time $t$ flow of the vector field $\mathcal X(V,\Omega)$.  

\begin{lem}
There exist continuous functions $\mathcal T\colon(0,\infty)\to (0,\infty)$ and $\mathcal L\colon (0,\infty)\to(0,\infty)$ with $\lim_{t\to 0} \mathcal T(t) = 0$ and $\lim_{t\to 0} \mathcal L(t) = 0$ such that for all $(V,\Omega)\in Y$ it holds that
\[
F(V_{\mathcal T(\gamma)},\Omega_{\mathcal T(\gamma)}) \le F(V,\Omega) - \mathcal L(\gamma),
\]
where we have abbreviated $\gamma = \mathscr F((V,\Omega),Y_0)$. 
\end{lem}

Define a map $\Psi\colon Y \times [0,1]\to Y$ by 
\[
\Psi((V,\Omega),t) = (V_{\mathcal T(\gamma) t}, \Omega_{\mathcal T(\gamma) t}),
\]
where again we have abbreviated $\gamma = \mathscr F((V,\Omega),Y_0)$. Then, for each $i$, define $\Xi_i\colon X \to \C(M)$ by $\Xi_i(x) = \pi\circ \Psi((\vert \Phi_i(x)\vert, \Phi_i(x)),1)$ where $\pi(V,\Omega)=\Omega$. It is straightforward to verify from the properties of $\Psi$ that the sequence $\{\Xi_i\}$ is as required. 
\end{proof}

\subsection{Replacements and the Almost Minimizing Property} 

In this section, we define a suitable almost-minimizing property and then use it to construct replacements.  Note that previous Almgren-Pitts min-max schemes have always used for a local functional.  
For example, in the case of the $A^h$ functional, if a competitor $\Omega$ is modified in a small open set $U$, it is possible to determine the change in $A^h$ without knowledge of $\Omega \rest (M- U)$.  However, this is not the case for $E$.  If a competitor $\Omega$ is modified in a small open set $U$, it is impossible to determine the change in $f(\vol(\Omega))$ without knowledge of $\Omega\rest (M- U)$.  
The key observation is that the functional $F$ still satisfies a quasi-locality type property if we restrict to modifications that change $A^h$ faster than volume. 

\begin{defn}
Define $a = \max_{v\in [0,\vol(M)]} \vert f'(v)\vert$.  Note that $a > 0$ since we assume $f$ is not constant. 
\end{defn}

\begin{defn}
Define $b = \sup_M \vert h\vert$. 
\end{defn}

\begin{prop}
\label{F-ql}
Assume $U_1,U_2,\hdots,U_N$ are disjoint open sets in $M$. Assume that $\Omega\in \mathcal C(M)$. Fix a small $\delta > 0$.  Suppose there are sets $\Theta_k$ with $\supp(\Omega-\Theta_k)\subset U_k$ and 
\[
2a \vert \vol(\Omega)-\vol(\Theta_i)\vert \le A^h(\Omega) - A^h(\Theta_i) + \delta
\]
for $k=1,2,\hdots,N$.  Let 
\[
\Omega^* = \Omega \rest (M\setminus (U_1\cup \hdots \cup U_k)) + \sum_{k=1}^N \Theta_k \rest U_k. 
\]
Then there is an estimate 
\[
F(\Omega^*) \le F(\Omega)  +  \frac{1}{2}\sum_{i=1}^N \big(A^h(\Theta_k) - A^h(\Omega)\big) + \frac{N\delta}{2}. 
\]
\end{prop}

\begin{proof}
Note that 
\begin{align*}
&\vol(\Omega^*)  = \vol(\Omega\setminus (U_1\cup \hdots \cup U_k)) + \sum_{k=1}^N \vol(\Theta_k\cap U_k).
\end{align*}
Observe that 
\[
\vol(\Theta_k\cap U_k) - \vol(\Omega\cap U_k) = \vol(\Theta_k)-\vol(\Omega) \le \frac{ A^h(\Omega)-A^h(\Theta_k)}{2a} + \frac{\delta}{2a}. 
\]
for $k=1,\hdots,N$.
It follows that 
\[
\vert \vol(\Omega^*) - \vol(\Omega)\vert \le \frac{N\delta}{2a}+ \sum_{k=1}^N \frac{ A^h(\Omega)-A^h(\Theta_k)}{2a}.
\]
Next observe that 
\begin{align*}
f(\vol(\Omega^*)) &\le f(\vol(\Omega)) + a \vert \vol(\Omega^*) - \vol(\Omega)\vert \\
&\le f(\vol(\Omega)) + \frac{N\delta}{2} +  \sum_{k=1}^N \frac{A^h(\Omega)-A^h(\Theta_k)}{2}. 
\end{align*}
Thus we have 
\begin{align*}
F(\Omega^*) &= A^h(\Omega) + \sum_{k=1}^N \big({A^h(\Theta_k)-A^h(\Omega)}\big) + f(\vol(\Omega^*))\\
&\le A^h(\Omega) + \frac{N\delta}{2} +  \frac{1}{2}\sum_{k=1}^N \big(A^h(\Theta_k) - A^h(\Omega)\big) + f(\vol(\Omega))\\
&= F(\Omega) + \frac{1}{2}\sum_{k=1}^N\big (A^h(\Theta_k) - A^h(\Omega)\big) + \frac{N\delta}{2}. 
\end{align*}
This proves the result. 
\end{proof}

\subsubsection{The Almost-Minimizing Property} 
The above quasi-locality property is the motivation for the following definition of $F$ almost-minimizing sets. 

\begin{defn}
\label{F-am}
Let $U$ be an open subset of $M$.  Let $\nu$ denote either the $\mathcal F$, $\mathbf F$, or $\M$ norm.  Fix constants $\eps,\delta > 0$ and $\gamma \ge 0$.  Fix an element $\Omega \in \C(M)$.  Let $\{\Omega_i\}_{i=1,\hdots,k}$ be a sequence in $\C(M)$.  Assume that  
\begin{itemize}
\item[(i)] $\Omega_1 = \Omega$,
\item[(ii)] $\operatorname{supp}(\Omega-\Omega_i) \subset U$ for all $i$,
\item[(iii)] $2 a \vert \vol(\Omega)-\vol(\Omega_i)\vert  \le  A^h(\Omega) - A^h(\Omega_i) + \gamma$ for all $i$, 
\item[(iv)] $\mathcal \nu(\bd \Omega_i,\bd \Omega_{i+1}) < \delta$ for all $i$,
\item[(v)]  $A^h(\Omega_i) \le A^h(\Omega) + \delta$ for all $i$.
\end{itemize}
We say $\Omega$ is $(F,\eps,\delta,\gamma,\nu)$-almost-minimizing in $U$ if for all sequences $\{\Omega_i\}_{i=1,\hdots,k}$ as above, we have $A^h(\Omega_k) \ge A^h(\Omega) - \eps$. 
\end{defn}

\begin{rem}
Note that we allow $\gamma = 0$.  When $\nu = \mathcal F$, we will always choose $\gamma = 0$.  When $\nu = \M$, we will always choose $\gamma = \delta > 0$ to allow room for interpolation. 
\end{rem}

\begin{defn} 
\label{F-am2} 
We define  $F$-almost-minimizers. Let $(V,\Omega)\in \vc(M)$ and let $U$ be an open subset of $M$.
\begin{itemize}
\item[(a)] We say $(V,\Omega)$ is $(F,\mathcal F)$-almost-minimizing in $U$ if there is a sequence $\Omega_i \in C(M)$ such that $\Omega_i$ is $(F,\eps_i,\delta_i,0,\mathcal F)$-almost-minimizing in $U$, $\vert \bd \Omega_i \vert \to V$, $\Omega_i \to \Omega$, and $\eps_i ,\delta_i\to 0$. 
\item[(b)] We say $(V,\Omega)$ is $(F,\M)$-almost-minimizing in $U$ if there is a sequence $\Omega_i \in \C(M)$ such that $\Omega_i$ is $(F,\eps_i,\delta_i,\delta_i,\M)$ almost-minimizing in $U$, $\vert \bd \Omega_i\vert \to V$, $\Omega_i\to \Omega$, and $\eps_i,\delta_i\to 0$. 
\end{itemize}
\end{defn}

\begin{prop}
\label{F-cc}
Assume that $\Omega$ is $(F,\eps,\rho,\rho,\M)$-almost-minimizing in $U$. If $\delta \le \rho$ is small enough, then $\Omega$ is $(F,\eps,\delta, 0,\mathcal F)$-almost-minimizing in any open set $W$ compactly contained in $U$. 
\end{prop}

\begin{proof}
We rely on an interpolation process from \cite{zhou2019min} and \cite{zhou2020existence}.  Assume $\delta$ is small enough that we can apply \cite[Lemma A.1]{zhou2019min} with $\eta = \min\{\rho/(2a + 1), \rho/2\}$ and $c = \max\{1,b\}$.  We can assume that $\delta \le \eta$.   Assume that $\Omega_1$ and $\Omega_2$ satisfy $\mathcal F(\bd \Omega_1,\bd \Omega_2) < \delta$ and $\supp(\Omega-\Omega_i)\subset W$ and $A^h(\Omega_i) \le A^h(\Omega) + \delta$ and 
\[
2a \vert \vol(\Omega) - \vol(\Omega_i)\vert \le  A^h(\Omega) - A^h(\Omega_i)
\]
for $i=1,2$. By \cite[Lemma A.1]{zhou2019min} (also see \cite[Lemma A.1]{zhou2020existence}), there exists a sequence $\Omega_1 = \Lambda_0,\Lambda_1,\hdots,\Lambda_m = \Omega_2\in \C(M)$ such that for each $j=1,\hdots,m-1$ we have 
\begin{itemize}
\item[(a)] $\supp(\Lambda_j - \Omega) \subset U$,
\item[(b)] $\M(\bd \Lambda_k - \bd \Lambda_{j+1}) \le \eta$,
\item[(c)] $A^h(\Lambda_j) \le \max\{A^h(\Omega_1),A^h(\Omega_2)\} + {\eta}$,
\item[(d)] $\vert \vol(\Lambda_j) -\vol(\Omega_i)\vert \le \eta$, for $i=1,2$.  
\end{itemize}
One verifies immediately that (ii), (iv), and (v) hold along the sequence $\Lambda_j$ with $\nu = \M$ and with $\delta$ replaced by $\rho$. Choose $k$ so that $\max\{A^h(\Omega_1),A^h(\Omega_2)\} = A^h(\Omega_k)$.  Observe that 
\begin{align*}
2a \vert \vol(\Omega)-\vol(\Lambda_j)\vert &= 2a \vert \vol(\Omega) -\vol(\Omega_k) + \vol(\Omega_k)-\vol(\Lambda_j)\vert \\
&\le 2a \vert \vol(\Omega)-\vol(\Omega_k)\vert + 2a \eta,
\end{align*}
and therefore 
\begin{align*}
 A^h(\Omega) - A^h(\Lambda_j)&=  A^h(\Omega) - A^h(\Omega_k) + A^h(\Omega_k) - A^h(\Lambda_j) \\
&\ge 2a \vert \vol(\Omega) - \vol(\bd \Omega_k)\vert   - \eta\\
&\ge 2a \vert \vol(\Omega)-\vol(\Lambda_j)\vert - (2a + 1) \eta  \\
&\ge 2a \vert \vol(\Omega)-\vol(\Lambda_j)\vert - \rho.
\end{align*}
Thus (iii) also holds along the sequence $\Lambda_j$ with $\gamma = \rho$.  It is easy to see that the above construction can be used to prove that if $\Omega$ is not $(F,\eps,\delta, 0, \mathcal F)$ almost-minimizing in $W$ then $\Omega$ is not $(F,\eps,\rho,\rho,\M)$ almost-minimizing in $U$. 
\end{proof} 

\begin{prop}
\label{prop:F-M a-m equivalentto F-F a-m}
If $(V,\Omega)\in \vc(M)$ is $(F,\M)$-almost-minimizing in $U$ then $(V,\Omega)$ is $(F,\mathcal F)$-almost-minimizing in any open set $W$ compactly contained in $U$. 
\end{prop}  

\begin{proof}
Assume that $(V,\Omega)$ is $(F,\M)$-almost-minimizing in $U$ and let $W$ be an open set compactly contained in $U$.  According to the definition, there is a sequence $\Omega_i\in \C(M)$ such that $\Omega_i$ is $(F,\eps_i,\rho_i,\rho_i,\M)$ almost-minimizing in $U$ and $\vert \bd \Omega_i\vert\to V$ and $\Omega_i\to \Omega$ and $\eps_i,\rho_i\to 0$. By Proposition \ref{F-cc}, for each $i$, there is a $\delta_i \le \rho_i$ such that $\Omega_i$ is $(F,\eps_i,\delta_i,0,\mathcal F)$-almost-minimizing in $W$. Since $\eps_i,\delta_i\to 0$, this witnesses that $(V,\Omega)$ is $(F,\mathcal F)$-almost-minimizing in $W$. 
\end{proof}

Recall that $a = \sup \vert f'\vert$ and $b = \sup \vert h\vert$. 

\begin{prop}
\label{F-ambv}
Assume that $(V,\Omega)\in \vc(M)$ is $(F,\mathcal F)$-almost-minimizing in $U$. Then $V$ has $(2a+b)$-bounded first variation in $U$. 
\end{prop}

\begin{proof}
We will prove the contrapositive. Suppose that $V$ does not have $(2a+b)$-bounded first variation in $U$.  Then there is a smooth vector field $X$ compactly supported in $U$ such that 
\[
\delta V(X) < -(2a + b + \eps_0)   \int_M \vert X\vert \, d\mu_V
\]
for some $\eps_0 > 0$. 
By continuity and the first variation formula, there is an $\eps_1 > 0$ such that for all $\Omega\in \C(M)$ with  $\mathbf F(\vert \bd \Omega\vert,V) < 2\eps_1$ we have
\begin{gather*}
\delta \vert \bd \Omega \vert (X) + 2a \int_M \vert X\vert \, d\mu_{\bd \Omega} \le 0,\\
\delta A^h\vert_\Omega(X) \le - \frac{\eps_0}{2} \int_M \vert X\vert\, d\mu_{\bd \Omega}. 
\end{gather*}
Fix some $\Omega$ with $\mathbf F(\vert \bd \Omega \vert,V) < \eps_1$. Deforming $\Omega$ along the flow of $X$ for a uniform small time $\tau$ yields an $\mathbf F$-continuous family $(\Omega_t)_{t\in [0,\tau]}$ such that $\mathbf F(\vert \bd \Omega_t\vert , V)<2\eps_1$ for all $t\in[0,\tau]$. Observe that $t \mapsto A^h(\Omega_t)$ is decreasing for $t\in [0,\tau]$. We also have  
\begin{align*}
A^h(\Omega_\tau) \le A^h(\Omega) - \eps_2
\end{align*}
where $\eps_2 > 0$ is a uniform constant that does not depend on the choice of $\Omega$. Finally, note that for any $t\in [0,\tau]$ we have 
\begin{align*}
2a \vert \vol(\Omega) - \vol(\Omega_t)\vert &= 2a \left\vert \int_0^t \delta \vol |_{\Omega_t}(X) \, dt\right\vert \\
&\le 2a \int_0^t  \left(\int_{\bd \Omega_t} \vert X\vert\right)  dt \le  -\int_0^t  \delta \vert \bd \Omega_t\vert(X)\, dt. 
\end{align*}
Therefore $\area(\bd \Omega)-\area(\bd \Omega_t) \ge 2a \vert \vol(\Omega)-\vol(\Omega_t)\vert$ for all $t\in [0,\tau]$. Discretizing the flow witnesses that $\Omega$ cannot be $(F,\eps_2,\delta,0,\mathcal F)$ almost-minimizing in $U$ for any $\delta > 0$. 
\end{proof} 

\begin{rem}
    \label{rem:am-stability}
    Note that we cannot use a similar argument to prove that $(V, \Omega)$ is stable in a suitable sense for $A^h$ in $U$ as in \cite{wang2023existence}, mainly due to the volume constraint in Definition \ref{F-am}(iii). This will bring in extra challenges to show compactness in Section \ref{S:compactness}.
\end{rem}

\begin{defn}
An element $(V,\Omega)\in \vc(M)$ is called $F$-almost-minimizing in annuli if for each $x\in M$ there is a number $\rho(x) > 0$ such that $(V,\Omega)$ is $(F,\mathcal F)$ almost-minimizing in every annulus $\an(x,s,r)$ with $s < r < \rho(x)$. 
\end{defn}

\begin{prop}
\label{F-ca}
Let $\Pi$ be the $(X,Z)$-homotopy class of a continuous map $\Phi_0\colon X\to (\C(M),\mathbf F)$, where $X$ is a cubical subcomplex of $I(m,k)$ for some $m,k\in \N$. Assume that 
\[
L^F(\Pi) > \sup_{z\in Z} F(\Phi_0(z)). 
\]
Choose a pulled-tight critical sequence $\{\Phi_i\}$ and let $\mathcal K =\mathcal K(\{\Phi_i\})$ be the critical set.  Then there is an element $(V,\Omega)\in \mathcal K$ which is $F$-almost minimizing in annuli. In fact, there exists some $(V,\Omega)\in \mathcal K$ satisfying the following stronger property: 
\begin{itemize}
    \item[(R)]\label{item:property R for F} There is a number $N = N(m)$ depending only on $m$ such that for any collection of $N$ concentric annuli $\an(x,s_1,r_1)$, $\hdots$, $\an(x,s_N,r_N)$ with $2r_j < s_{j+1}$, $(V,\Omega)$ is $(F,\mathcal F)$ almost minimizing in at least one of the annuli. 
\end{itemize}
\end{prop}

\begin{proof} 
Applying discretization \cite[Theorem 1.11]{zhou2020multiplicity}, we can find a homotopy sequence $\{\phi_i\}$ such that $L(\{\phi_i\})= L(\Pi)$ and $\mathcal K(\{\phi_i\}) = \mathcal K(\{\Phi_i\})$. 
Assume for contradiction that no element in $\mathcal K$ satisfies property (R). Then by Proposition \ref{prop:F-M a-m equivalentto F-F a-m}, no element in $\mathcal K$ satisfies property (R) with $(E,\mathcal F)$ almost-minimizing replaced by $(E,\M)$ almost-minimizing. Thus we can apply Parts 1-19 of the Almgren-Pitts combinatorial construction \cite[Theorem 4.10]{pitts2014existence} to $\{\phi_i\}$. This produces a homotopic sequence $\{\psi_i\}$.  

Note that in Part 20, one can use Proposition \ref{F-ql} to verify that 
\[
L(\{\psi_i\}) < L(\{\phi_i\}) - \eps
\]
for some $\eps > 0$. Let us be more precise. Consider disjoint annuli $A_1,A_2,\hdots,A_K$. Fix $\Omega\in \C(M)$. As part the the combinatorial construction, one obtains small constants $\eps,\delta>0$ and Caccioppoli sets $\Omega_k(j)$, $k\in \{1,\hdots,K\}$, $j\in \{1,\hdots,J\}$ satisfying
\begin{itemize}
\item[(i)] $\Omega_k(1) = \Omega$,
\item[(ii)] $\supp(\Omega_k(j)-\Omega) \subset A_k$ for all $j$ and $k$,
\item[(iii)] $2a\vert \vol(\Omega)-\vol(\Omega_k(j))\vert \le A^h(\Omega)-A^h(\Omega_k(j)) + \delta$ for all $j$ and $k$,
\item[(iv)] $\M(\bd \Omega_k(j)-\bd\Omega_k(j+1))< \delta$ for all $j$ and $k$,
\item[(v)] $A^h(\Omega_k(j))\le A^h(\Omega) + \delta$ for all $j$ and $k$,
\item[(vi)] $A^h(\Omega_k(J))\le A^h(\Omega) - \eps$. 
\end{itemize}
For a choice of integers $j_1,\hdots,j_N\in \{1,\hdots,J\}$, one then considers a set of the form 
\[
\Omega^* = \Omega\rest (M\setminus (A_1\cup \hdots \cup A_N)) + \sum_{k=1}^K \Omega_k(j_k)\rest A_k
\]
and needs to estimate the difference $F(\Omega) - F(\Omega^*)$. This can be accomplished by applying Proposition \ref{F-ql} with $U_k = A_k$ and $\Theta_k = \Omega_k(j_k)$.

Applying interpolation \cite[Theorem 1.12]{zhou2020multiplicity} to $\{\psi_i\}$ then gives a sequence 
\[
\{\Psi_i \colon X\to (\C(M),\mathbf F)\}
\]
in the $(X,Z)$-homotopy class $\Pi$ with $\sup_{x\in X} F(\Psi_i(x)) < L(\Pi) - \eps/2$ for all large $i$. This is a contradiction. 
\end{proof}

\subsubsection{A Constrained Minimization Problem}  Consider the following constrained minimization problem.  Assume that $\Omega$ is $(F,\eps,\delta,0,\mathcal F)$ almost minimizing in $U$.  Let $K$ be a compact subset of $U$.   Let $\mathcal A = \mathcal A(\Omega,K,\delta)$ be the set of all $\Theta\in \mathcal C(M)$ such that there is a sequence $\{\Omega_i\}_{i=1,\hdots,k}$ in $\C(M)$ satisfying 
\begin{itemize}
\item[(i')] $\Omega_1 = \Omega$,
\item[(ii')] $\operatorname{supp}(\Omega-\Omega_i) \subset K$ for all $i$,
\item[(iii')] $2 a  \vert \vol(\Omega)-\vol(\Omega_i)\vert \le  A^h(\Omega) - A^h(\Omega_i)$ for all $i$, 
\item[(iv')] $\mathcal F(\bd \Omega_i,\bd\Omega_{i+1}) < \delta$ for all $i$,
\item[(v')]  $A^h(\Omega_i) \le A^h(\Omega) + \delta$ for all $i$,
\end{itemize}
and $\Theta = \Omega_k$.
We aim to minimize $A^h$ among elements of $\mathcal A$. 

\begin{prop}
There exists $\Theta\in \mathcal A$ such that $A^h(\Theta)= \inf_{\Xi \in \mathcal A} A^h(\Xi)$. 
\end{prop}

\begin{proof}
Choose a minimizing sequence $\Theta_j \in \mathcal A$.  
Since the mass of $\bd \Theta_j$ is uniformly bounded, by passing to a subsequence if necessary, we can assume that $\Theta_j\to \Theta$ in $\C(M)$. Since $A^h$ is lower-semicontinuous in the flat topology, one has 
\[
A^h(\Theta) \le \inf_{\Xi \in \mathcal A} A^h(\Xi).
\]
To complete the proof, it remains to show that $\Theta\in \mathcal A$.  Select $J$ large enough so that $\mathcal F(\bd \Theta_J,\bd \Theta) < \delta$.  There is a sequence $\{\Omega_i\}_{i=1,\hdots,k}$ satisfying (i')-(v') and such that $\Theta_J = \Omega_k$. Consider the extended sequence $\{\Omega_j\}_{j=1,\hdots,k+1}$ where $\Omega_{k+1}=\Theta$. It is easy to see that (i'),(ii'),(iv'),(v') hold for the extended sequence.   Note that 
\[
2a \vert \vol(\Omega)-\vol(\Theta_j)\vert \le  A^h(\Omega) - A^h(\Theta_j) 
\]
for all $j$.   Passing to the limit as $j\to \infty$ and using the lower semi-continuity of $A^h$, it follows that  
\[
2a \vert \vol(\Omega)-\vol(\Theta)\vert \le A^h(\Omega) - A^h(\Theta).
\]
Thus (iii') holds as well, and it follows that $\Theta\in \mathcal A$, as needed.
\end{proof} 

\begin{prop} 
\label{F-mam} 
Assume $\Omega$ is $(F,\eps,\delta,0,\mathcal F)$ almost-minimizing in $U$ and let $K$ be a compact subset of $U$.  Let $\Theta$ be a solution to the constrained minimization problem. Then $\Theta$ is $(F,\eps,\delta,0,\mathcal F)$ almost-minimizing in $U$. 
\end{prop}

\begin{proof}
Since $\Theta\in \mathcal A$, there is a sequence $\{\Omega_i\}_{i=1,\hdots,k}$ satisfying (i')-(v') such that $\Theta = \Omega_k$.  Now consider any sequence $\{\Theta_i\}_{i=1,\hdots,j}$ such that 
\begin{itemize}
\item[(i'')] $\Theta_1 = \Theta$,
\item[(ii'')] $\operatorname{supp}(\Theta-\Theta_i) \subset U$ for all $i$,
\item[(iii'')] $2 a \vert \vol(\Theta)-\vol(\Theta_i)\vert  \le  A^h(\Theta) - A^h(\Theta_i)$ for all $i$, 
\item[(iv'')] $\mathcal F(\bd \Theta_i,\bd \Theta_{i+1}) < \delta$ for all $i$,
\item[(v'')]  $A^h(\Theta_i) \le A^h(\Theta) + \delta$ for all $i$.
\end{itemize}
Now consider the concatenated sequence $\Omega_1,\hdots,,\Omega_{k-1},\Omega_k=\Theta_1,\Theta_2,\hdots,\Theta_k$. It is clear that (i), (ii), (iv), and (v) hold along the extended sequence. Finally note that 
\begin{align*}
2a \vert \vol(\Omega)-\vol(\Theta_i)\vert &= 2a \vert \vol(\Omega) - \vol(\Theta) + \vol(\Theta) - \vol(\Theta_i)\vert \\
&\le 2a \vert \vol(\Omega)-\vol(\Theta)\vert + 2a \vert \vol(\Theta)-\vol(\Theta_i)\vert \\
& \le A^h(\Omega)-A^h(\Theta) + A^h(\Theta) - A^h(\Theta_i) \\
& = A^h(\Omega)-A^h(\Theta_i) 
\end{align*} 
for $i=2,\hdots,j$. Thus (iii) also holds along the extended sequence. Since $\Omega$ is almost-minimizing, it follows that $A^h(\Theta_k) \ge A^h(\Omega) - \eps \ge A^h(\Theta)-\eps$. 
\end{proof}

\begin{prop}
Let $\Theta$ be a solution to the constrained minimization problem.   Then $\bd \Theta$ locally minimizes $A^h$ with respect to volume preserving modifications in the interior of $K$.  In the interior of $K$, $\bd \Theta$ is smooth with mean curvature $H = h + h_0$ for some constant $h_0$ and, moreover, $\vert H\vert \le 2a + b$.  The surface $\bd \Theta$ is a volume preserving stable critical point for $A^{h+h_0}$ in the interior of $K$. 
\end{prop}

\begin{proof} There is a sequence $\{\Omega_i\}_{i=1,\hdots,k}$ satisfying (i')-(v') such that $\Theta = \Omega_k$. We will show that if any of the claimed properties fail, it is possible to extend this sequence so as to contradict that $\Theta$ solved the constrained minimization problem. 

Consider an open set $W\subset K$ with $\vol(W) < \delta$.   Assume for contradiction that $\bd \Theta$ does not minimize area with respect to volume preserving modifications in $W$.  Then there exists $\Xi\in \C(M)$ such that $\supp(\Xi-\Theta) \subset W$ and $\vol(\Xi) = \vol(\Theta)$ and $A^h(\Xi) < A^h(\Theta)$.    Consider the extended sequence $\{\Omega_i\}_{i=1,\hdots,k+1}$ where $\Omega_{k+1}=\Xi$. It is easy to see that the extended sequence satisfies (i'), (ii'), and (v').  We have $\mathcal F(\bd \Omega_k,\bd \Omega_{k+1}) \le \vol(W) < \delta$ and so (iv') holds. Finally observe that 
\begin{align*}
A^h(\Omega) - A^h(\Xi)  &\ge A^h(\Omega) - A^h(\Theta)  \\
&\ge 2a \vert \vol(\Omega)-\vol(\Theta)\vert \\
&= 2a \vert \vol(\Omega) - \vol(\Xi)\vert
\end{align*}
and so (iii') holds.  It follows that $\Xi\in \mathcal A$, and this is a contradiction. 

Proposition \ref{regularity-for-minimizers} implies that $\bd \Theta$ is induced by a multiplicity one, smooth, embedded surface with mean curvature $H = h + h_0$ in the interior of $K$.  Here $h_0$ is a constant.  We claim that $\vert H\vert \le 2a+ b$. Suppose for contradiction that $\vert H\vert > 2a + b$ at some point $p\in \bd \Theta \cap \text{int}(K)$. Fix a small number $r > 0$.  {The assumption on the mean curvature gives} the existence of $\Xi\in \C(M)$ with $\supp(\Xi-\Theta)\subset B_r(p)$ and 
\[
\M(\bd \Theta)-\M(\bd \Xi) \ge (2a + b) \vert \vol(\Theta)-\vol(\Xi)\vert > 0.
\]
Moreover, we can ensure that $\bd \Xi$ lies to one side of $\bd \Theta$ so that we have 
$
\vol(\Xi \operatorname{\Delta} \Theta) = \vert \vol(\Theta) - \vol(\Xi)\vert. 
$
Note that this implies 
\begin{align*}
A^h(\Theta) - A^h(\Xi) &\ge (2a + b) \vert \vol(\Theta) - \vol(\Xi)\vert - \int_\Theta h + \int_\Xi h\\
&\ge (2a + b) \vert \vol(\Theta) - \vol(\Xi) \vert - b \vol(\Theta \operatorname{\Delta} \Xi) = 2a \vert \vol(\Theta) - \vol(\Xi)\vert. 
\end{align*}
Consider the extended sequence $\{\Omega_i\}_{i=1,\hdots,k+1}$ where $\Omega_{k+1}=\Xi$.  It is easy to see that (i'), (ii'), (iv'), and (v') hold provided $r$ is small enough.  It remains to note that 
\begin{align*}
A^h(\Omega) - A^h(\Xi)  &= A^h(\bd \Omega) - A^h(\Theta)  + A^h(\Theta) - A^h(\Xi) \\
&\ge 2a \vert \vol(\Omega) - \vol(\Theta)\vert + 2a \vert \vol(\Theta)-\vol(\Xi)\vert \\
&\ge 2a \vert \vol(\Omega) - \vol(\Xi)\vert.
\end{align*}
Thus (iii') holds so $\Xi\in \mathcal A$ and this is a contradiction. 

It remains to show that $\Theta$ is a volume preserving stable critical point for $A^{h+h_0}$ in the interior of $K$. Note that it is equivalent to show that $\Theta$ is a volume preserving stable critical point for $A^h$.  Suppose this is not the case.  Then there is a smooth function $\varphi\f \bd \Theta\to \R$ with compact support in $\ins(K)$ such that 
\[
\int_{\bd \Theta\cap K} \varphi  = 0, \text{ and } \int_{\bd \Theta \cap K}  \vert \grad \varphi\vert^2  - (\vert A\vert^2 + \ric(\nu,\nu) - \bd_\nu h)\varphi^2 < 0.
\]
According to \cite[Lemma 2.2]{barbosa2012stability}, it is possible to find a smooth family $\bd \Theta_t$, $t\in [-\eps_0,\eps_0]$ such that 
\begin{itemize}
\item[(a)] $\Theta_0 = \Theta$ 
\item[(b)] $\supp(\Theta_t-\Theta)\subset \ins(K)$ for all $t\in[-\eps_0,\eps_0]$ 
\item[(c)] $\vol(\Theta_t) = \vol(\Theta)$ for all $t\in[-\eps_0,\eps_0]$,
\item[(d)] one has
\[
\frac{d}{dt}\eval_{t=0} \bd \Theta_t = \varphi \cdot \nu
\]
where $\nu$ is the unit normal pointing into $\Theta$.
\end{itemize}
{By the second variation of $A^{h}$ and the volume preserving condition}, there is an $\eps_1 < \eps_0$ such that  $A^h(\Theta_t) < A^h(\Theta)$ for all $t\in (0,\eps_1)$.  Fix a very small $t_0 > 0$, and consider the extended sequence $\{\Omega_i\}_{i=1,\hdots,k+1}$ where $\Omega_{k+1} = \Theta_{t_0}$. It is easy to verify that the extended sequence satisfies (i')-(v') provided $t_0$ is small enough, and this contradicts that $\Theta$ solved the constrained minimization problem.
\end{proof}

\subsubsection{Replacements}

Next we turn to the construction of replacements. 

\begin{prop}
Assume that $V$ has $(2a+b)$-bounded first variation, and also that $(V,\Omega)$ is $(F,\mathcal F)$-almost-minimizing in an open set $U$. Let $K$ be a compact subset of $U$.  There exists an element $(V^*,\Omega^*)\in \vc(M)$ called a replacement for $(V,\Omega)$ such that 
\begin{itemize}
\item[(i)] $(V, \Omega) \rest (M-K) = (V^*, \Omega^*)\rest (M-K)$,
\item[(ii)] $A^h(V,\Omega) = A^h(V^*,\Omega^*)$,
\item[(iii)] $(V^*,\Omega^*)$ is $(F,\mathcal F)$ almost-minimizing in $U$,
\item[(iv)] $V^*$ has $(2a+b)$-bounded first variation,
\item[(v)] $(V^*,\Omega^*)$, when restricted to the interior of $K$, is a limit of some solutions to the constrained minimization problems. 
\end{itemize}
\end{prop}

\begin{proof}
Let $(V,\Omega)$ and $U$ and $K$ be as in the assumptions of the theorem.  By definition, there exists a sequence $\Omega_i$ such that $\vert \bd \Omega_i\vert\to \vert V\vert$ and $\Omega_i\to \Omega$ and $\Omega_i$ is $(F,\eps_i,\delta_i,0,\mathcal F)$-almost-minimizing in $U$ and $\eps_i,\delta_i\to 0$.  Let $\Theta_i$ be the solution to a corresponding constrained minimization problem, i.e., 
\[
A^h(\Theta_i) = \inf_{\Theta \in \mathcal A(\Omega_i,K,\delta_i)} A^h(\Theta).
\]
By compactness, passing to a subsequence if necessary, we can assume that $\vert \bd \Theta_i\vert \to V^*$ and $\Theta_i \to \Omega^*$. We claim that $(V^*,\Omega^*)$ has the necessary properties. 

Properties (i) and (v) are clear.  To check property (iii), note that by Proposition \ref{F-mam}, each $\Theta_i$ is $(F,\eps_i,\delta_i,0,\mathcal F)$-almost-minimizing in $U$.  Since $\vert \bd \Theta_i\vert \to V^*$ and $\Theta_i\to \Omega^*$, it follows that $(V^*,\Omega^*)$ is $(F,\mathcal F)$-almost-minimizing in $U$.  To see that (ii) holds, observe that 
\[
A^h(\Omega_i) - \eps_i \le A^h(\Theta_i) \le A^h(\Omega_i)
\]
for all $i$. We can let $i\to \infty$ in the above equation to deduce that 
\[
A^h(V,\Omega) = A^h(V^*,\Omega^*)
\]
which is property (iii).  It remains to check (iv). Since $V$ has $(2a+b)$-bounded first variation, it follows that that $V^*$ has $(2a+b)$-bounded variation on $M-K$. Also $V^*$ has $(2a+b)$-bounded variation on $U$ by Proposition \ref{F-ambv}. This implies that $V^*$ has $(2a+b)$-bounded variation on all of $M$.
\end{proof}

\begin{rem}
Note that since the replacement $(V^*,\Omega^*)$ is still almost-minimizing in $U$, it is therefore possible to obtain a replacement $(V^{**},\Omega^{**})$ for $(V^*,\Omega^*)$ in any compact set $K' \subset U$, and so on. 
\end{rem}

To prove the regularity of replacements, we need to use the curvature estimates for volume preserving stable critical points of $A^{h+h_0}$. 

\begin{prop}
Let $(V^*,\Omega^*)$ be a replacement for $(V,\Omega)$ in $K$.  Then, in the interior of $K$, the varifold $V^*$ is induced by a smooth, almost-embedded hypersurface $\Sigma$ with multiplicity one. The surface $\Sigma$ has mean curvature $H = h + h_0$ for some constant $h_0$ and its touching set is contained in a countable union of $(n-1)$-dimensional submanifolds.  
The surface $\Sigma$ coincides with $\bd \Omega^*$ in the interior of $K$. The surface $\Sigma$ is volume preserving stable as an immersion for $A^{h+h_0}$ in the interior of $K$.  
Finally, in the interior of $K$, there is a curvature estimate 
\[
\|A_\Sigma\|^2(x) \le \frac{C}{\dist(x,\bd K)^2}
\]
where $C$ is a constant that depends only on $M$, $a$, $b$, and $\|V\|(M)$. 
\end{prop}

\begin{proof}
In the interior of $K$, the replacement $V^*$ is a limit of smooth, embedded, volume preserving stable critical points $\Sigma_k$ for $A^{h+h_k}$. Here $\{h_k\}$ is a sequence of constants. Moreover, the surfaces $\Sigma_k$ have uniformly bounded mean curvature. Hence the regularity follows from the curvature estimates of Proposition \ref{curvature-estimates}. By Proposition \ref{prop:touching set}, the touching set is contained in a countable union of $(n-1)$-dimensional submanifolds since $h$ satisfies property (T). 
\end{proof}

We conclude by investigating the tangent cones to almost-minimizers.  We show that if $(V,\Omega)$ is $(F,\mathcal F)$-almost-minimizing in annuli, then every tangent cone to $V$ is an integer multiple of a plane. In the following, $\eta_{p,r}$ denotes the map which rescales by a factor of $1/r$ centered at $p$. 

\begin{prop}
Assume that $(V,\Omega)$ has $(2a+b)$-bounded first variation and is $(F,\mathcal F)$-almost minimizing in a set $U$. Then for any sequence $p_i \to p \in U$ and any sequence of scales $r_i \to 0$, every varifold limit $\overline V = (\eta_{p_i,r_i})_\sharp V$ is an integer multiple of a complete, embedded minimal hypersurface. Moreover, every varifold tangent to $V$ is an integer multiple of a hyperplane. 
\end{prop}

\begin{proof} 
Choose a sequence of points $p_i\to p$ and a sequence of scales $r_i\to 0$.  Consider a varifold limit of the form
 \[
 \overline V  = \lim (\eta_{p_i,r_i})_\sharp V.
 \]
We can argue exactly as in \cite[Lemma 5.10]{zhou2019min} to show that $\overline V$ admits volume preserving stable minimal replacements in annuli. More precisely, for any annulus $\an \subset T_pM$, there exists a varifold $\overline V^*$ on $T_pM$ such that 
\begin{itemize}
\item[(i)] $\overline V \rest (T_pM \setminus \an) = \overline V^* \rest (T_pM \setminus \an)$,
\item[(ii)] $\|\overline V\|(B) = \|\overline V^*\|(B)$ for a large ball $B$ containing $\an$,
\item[(iii)] the varifold $\overline V^*$ is induced by a smooth, embedded, volume preserving stable minimal hypersurface with multiplicity in $\an$.
\end{itemize}
In fact, this process can be iterated any number of times to construct replacements for $\overline V^*$, replacements for these replacements, and so on. 

It is well known that the existence of replacements implies the regularity of $\overline V$ when volume preserving stability is replaced by strong stability in (iii), c.f. \cite[Appendix C]{zhou2019min}. However, strong stability is used only to get curvature estimates and to apply Simons theorem. Since volume preserving stable minimal hypersurfaces also have curvature estimates and satisfy  Simons theorem (Proposition \ref{simons}), the existence of volume preserving stable replacements is also enough to prove the regularity of $\overline V$. 

Finally, once the regularity of all such $\overline V$ is known, we can argue exactly as in \cite[Proposition 5.11]{zhou2019min} to show that every varifold tangent to $V$ is an integer multiple of a hyperplane.
\end{proof}

\subsection{Regularity} 

\label{section:regularity}

The goal of this subsection is to prove the regularity of the min-max pair $(V,\Omega)$. Given a smooth almost-embedded hypersurface $\Sigma$ with mean curvature $h+h_0$, let $\mathcal R(\Sigma)$ be the set of embedded points of $\Sigma$ and let $\mathcal S(\Sigma)$ denote the touching set of $\Sigma$.

Next we construct replacements for $(V,\Omega)$ on overlapping annuli.  To prove regularity, it is essential that the consecutive replacements can be selected to have matching mean curvature. The next proposition shows that this is always possible. 

\begin{prop}\label{prop:mean curvature match1}
Assume $(V,\Omega)$ is $F$-almost-minimizing in annuli and fix a point $p\in \supp \|V\|$.  
Fix sufficiently small numbers $s_1 < r_1$ and let $(V^*,\Omega^*)$ be a replacement for $(V,\Omega)$ in $\an(p,s_1,r_1)$. Let $\Sigma^*$ be the smooth, almost-embedded hypersurface inducing $V^*$ in $\an(p,s_1,r_1)$ and suppose $\Sigma^*$ has mean curvature $h+h_0$. Choose $s_1 < r_2 < r_1$ so that $\bd B_{r_2}(p)$ intersects $\Sigma^*$ and the countable union of manifolds containing $\mathcal S(\Sigma^*)$ transversally.  Fix any $s < s_1$ and let $(V^{**}_s,\Omega^{**}_s)$ be a replacement for $(V^*,\Omega^*)$ in $\an(p,s,r_2)$. Let $\Sigma^{**}_s$ be the smooth, almost-embedded hypersurface inducing $V^{**}$ in $\an(p,s,r_2)$. Then $\Sigma^{**}_s$ has mean curvature $H = h + h_0$ for the same constant $h_0$. 
\end{prop} 

\begin{proof} 
Choose $s < s_1 < r_2 < r_1$ as in the statement of the proposition. Observe that the maximum principle ensures that $\Sigma^*$ and $\Sigma^{**}_s$ are both non-empty provided $r_1$ is sufficiently small.  Note that $\Sigma^{**}_s$ has mean curvature $H = h + h_1$ for some constant $h_1$, and we need to show that $h_0 = h_1$. 

Suppose to the contrary that $h_0 \neq h_1$. Pick a point $p \in \mathcal R(\Sigma^*)\cap \an(p,r_2,r_1)$ and pick a point $q\in \mathcal R(\Sigma^{**}_s)\cap \an(p,s,r_2)$. Let $\nu$ be the unit normal pointing into $\Omega^{**}_s$. Let $X$ be a vector field supported in a small neighborhood of $\{p,q\}$ such that 
\[
\int_{\Sigma^*} \la X,\nu\ra = -\int_{\Sigma^{**}_s} \la X,\nu\ra \neq 0. 
\]
Note that this implies $\delta\vol|_{\Omega_s^{**}}(X)=0$. Since $h_0\neq h_1$, after replacing $X$ by $-X$ if necessary, we obtain that 
\[
\delta A^h|_{V^{**}_s,\Omega^{**}_s}(X) < 0. 
\]
Hence, by continuity, there are $\eps_1,\eps_2 > 0$ such that 
\[
\delta A^h|_\Theta(X) < -\eps_2, \quad \vert \delta \vol|_\Theta (X) \vert \le \frac{\eps_2}{4a}
\]
for all $\Theta\in \mathcal C(M)$ with $\mathscr F((\vert \bd \Theta\vert,\Theta), (V^{**}_s,\Omega^{**}_s)) < 2\eps_1$.  It follows that for any $\Theta\in \mathcal C(M)$ with $\mathscr F((\vert \bd \Theta\vert,\Theta), (V^{**}_s,\Omega^{**}_s)) < \eps_1$, we can flow $\Theta$ along the flow of $X$ for a uniform short time $\tau$ to obtain a family $\Theta_t$ such that 
\begin{itemize}
\item[(i)] $A^h(\Theta_t) \le A^h(\Theta)$ for all $t\in [0,\tau]$,
\item[(ii)] $2a \vert \vol(\Theta) - \vol(\Theta_t)\vert \le A^h(\Theta) - A^h(\Theta_t)$ for all $t\in [0,\tau]$,
\item[(iii)] $A^h(\Theta_\tau) \le A^h(\Theta) - \eps_3$.
\end{itemize}
Here $\eps_3>0$ is a uniform constant that does not depend on $\Theta$. Discretizing this family shows that $\Theta$ is not $(F,\eps_3,\delta,0,\mathcal F)$-almost-minimizing for any $\delta > 0$. Since this is true for all $\Theta$ which are $\mathscr F$-close to $(V^{**}_s,\Omega^{**}_s)$, this contradicts that $(V^{**}_s,\Omega^{**}_s)$ is $(F, \mathcal F)$-almost-minimizing. 
\end{proof}

We can now prove the regularity. 

\begin{prop}
\label{F-regularity}
Assume that $(V,\Omega)$ is stationary for $F$. Further suppose that $(V,\Omega)$ is $F$-almost-minimizing in annuli. Then $V$ is induced by a smooth, closed, almost-embedded hypersurface $\Sigma$ with multiplicity one. The touching set of $\Sigma$ is contained in a countable union of $(n-1)$-dimensional submanifolds. Moreover, $\Sigma$ coincides with the boundary of $\Omega$. Finally, the mean curvature of $\Sigma$ satisfies $H = h -f'(\vol(\Omega))$. 
\end{prop}

\begin{proof}
Once we know that successive replacements on overlapping annuli must have the same mean curvature, then we can proceed exactly as in \cite{zhou2020existence} to deduce local regularity: for each $p\in \supp \|V\|$ there is an $r > 0$ such that $V$ is induced by a smooth, almost-embedded hypersurface $\Sigma(p,r)$ with multiplicity one in $B_r(p)$. Moreover, the touching set of $\Sigma(p,r)$ is contained in a countable union of $(n-1)$-dimensional manifolds.  Finally, $\Sigma(p,r)$ has mean curvature $h + h_0(p,r)$ for some constant $h_0(p,r)$ that a priori depends on $p$ and $r$.  

The local regularity implies that $V$ is induced by finitely many smooth, almost-embedded, multiplicity one  components $(\Gamma_i)_{i=1}^k$ with mean curvature $H_i = h + h_i$, where $h_i$ is a constant that depends a priori on the component. Since $\| \, \vert \bd \Omega \vert\, \| \le \|V\|$, the constancy theorem implies that 
\[
\vert \bd \Omega\vert = \sum_{i=1}^k a_i \vert \Gamma_i\vert,
\]
where each $a_i$ is either 0 or 1. 
We claim that in fact $a_i = 1$ for all $i$. Indeed, this follows from the fact that $(V,\Omega)$ is stationary for $F$. If some $a_i$ was equal to 0, then since the mean curvature of $\Gamma_i$ is not identically 0, we could construct a local deformation near $\Gamma_i$ decreasing the area to first order, and leaving the region $\Omega$ unchanged. But such a deformation would also decrease $F$ to first order. Thus $a_i$ must equal 1, as claimed.  

Finally, the fact that $(V,\Omega)$ is stationary for $F$ implies that the mean curvature of $\Gamma_i$ with respect to the normal vector pointing into $\Omega$ must equal $h -f'(\vol(\Omega))$ for every component $\Gamma_i$. Otherwise it would again be possible to construct a deformation decreasing the $F$ functional to first order. This completes the proof.
\end{proof}

Theorem \ref{F-min-max} now follows by combining Proposition \ref{F-pt}, Proposition \ref{F-ca}, and Proposition \ref{F-regularity}.

\section{Absolute Min-Max for E}\label{S:absolute min-max for E}

The goal of this section is to develop a min-max theory for the $E$ functional. Throughout this section, we fix a closed manifold $M$, and a smooth function $f\colon[0,\vol(M)]\to \R$ satisfying (\ref{f is even}), and we abbreviate $E = E^f$. 

\subsection{Pull Tight}  
 We will employ a pull-tight operation to ensure that every point in the critical set is stationary for $E$.

\begin{prop}
\label{E-pt}
Let $\Pi$ be the homotopy class of an $\mathbf F$ continuous map $\Phi_0:X \to \mathcal B(M,\Z_2)$. Let $\{\Phi_i\}$ be a critical sequence for $\Pi$.  There is another critical sequence $\{\Xi_i\}$ such that $\mathcal K(\{\Xi_i\}) \subset \mathcal K(\{\Phi_i\})$ and every $(V,T)\in \mathcal K(\{\Xi_i\})$ is stationary for $E$. 
\end{prop}

\begin{proof} The proof is almost identical to Proposition \ref{F-pt}, except that one replaces the $\vc$ space by the $\vz$ space and lets $Z = \emptyset$. 
\end{proof}

\subsection{Replacements and the Almost Minimizing Property} 

In this section, we define a suitable almost-minimizing property and then use it to construct replacements.  Most of the constructions are very similar to those for the $F$ functional, although at some points we need to exercise care to ensure that nothing depends on whether we consider $T = \bd \Omega$ or $T = \bd(M\setminus \Omega)$. We will omit proofs that are essentially identical to those for $F$-almost minimizers. 

As in the previous section, the key point is that the functional $E$ still satisfies a quasi-locality type property if we restrict to modifications that change area faster than volume. As before, let $a = \sup \vert f'\vert$.

\begin{prop}
\label{E-ql}
Assume $U_1,U_2,\hdots,U_N$ are disjoint open sets in $M$. Assume that $T\in \mathcal B(M,\Z_2)$ and choose $\Omega\in \mathcal C(M)$ with $\bd \Omega = T$.  Let $\delta > 0$ be given.  Suppose there are $S_k = \bd \Theta_k\in \mathcal B(M,\Z_2)$ with $\supp(\Omega-\Theta_k)\subset U_k$ and 
\[
2a \vert \vol(\Omega)-\vol(\Theta_i)\vert \le \M(\bd \Omega) - \M(\bd \Theta_i) + \delta
\]
for $k=1,2,\hdots,N$.  Let 
\[
T^* = T\rest (M\setminus (U_1\cup \hdots \cup U_k)) + \sum_{k=1}^N S_k \rest U_k. 
\]
Then there is an estimate 
\[
E(T^*) \le E(T)  + \frac{1}{2} \sum_{i=1}^N \big(\M(S_k) - \M(T)\big) + \frac{N\delta}{2}. 
\]
\end{prop}

\subsubsection{The Almost-Minimizing Property} 
The above quasi-locality property is the motivation for the following definition of $E$ almost-minimizing cycles. 

\begin{defn}
\label{E-am}
Let $U$ be an open subset of $M$.  Let $\nu$ denote either the $\mathcal F$, $\mathbf F$, or $\M$ norm.  Fix constants $\eps,\delta > 0$ and $\gamma \ge 0$.  Fix an element $T \in \mathcal \B(M,\Z_2)$.  Choose $\Omega\in \C(M)$ with $\bd \Omega = T$. Let $(\Omega_i)_{i=1,\hdots,k}$ be a sequence in $\C(M)$.  Assume that  
\begin{itemize}
\item[(i)] $\Omega_1 = \Omega$,
\item[(ii)] $\operatorname{supp}(\Omega-\Omega_i) \subset U$ for all $i$,
\item[(iii)] $2 a  \vert \vol(\Omega)-\vol(\Omega_i)\vert  \le  \M(\bd \Omega) - \M(\bd \Omega_i) + \gamma$ for all $i$, 
\item[(iv)] $\mathcal \nu(\bd \Omega_i,\bd \Omega_{i+1}) < \delta$ for all $i$,
\item[(v)]  $\M(\bd \Omega_i) \le \M(\bd \Omega) + \delta$ for all $i$.
\end{itemize}
We say $T$ is $(E,\eps,\delta,\gamma,\nu)$-almost-minimizing in $U$ if for all sequences $(\Omega_i)_{i=1,\hdots,k}$ as above, we have $\M(\bd \Omega_k) \ge \M(T) - \eps$. 
\end{defn}

\begin{rem}
The above definition does not depend on the choice of $\Omega$.  Indeed if $(\Omega_i)$ is a sequence satisfying (i)-(v) with respect to $\Omega$ then $M-\Omega_i$ is a sequence satisfying (i)-(v) with respect to $M- \Omega$. 
\end{rem}

\begin{defn} We define $E$ almost-minimizers. Let $(V,T)\in \vz(M,\Z_2)$ and let $U$ be an open subset of $M$.
\begin{itemize}
\item[(a)] We say $(V,T)$ is $(E,\mathcal F)$ almost-minimizing in $U$ if there is a sequence $T_i \in \B(M,\Z_2)$ such that $T_i$ is $(E,\eps_i,\delta_i,0,\mathcal F)$-almost-minimizing in $U$ and $\vert T_i \vert \to V$ and $T_i \to T$ and $\eps_i ,\delta_i\to 0$.  
\item[(b)] We say $(V,T)$ is $(E,\M)$ almost-minimizing in $U$ if there is a sequence $T_i \in \B(M,\Z_2)$ such that $T_i$ is $(E,\eps_i,\delta_i,\delta_i,\M)$ almost-minimizing in $U$ and $\vert T_i\vert \to V$ and $T_i\to T$ and $\eps_i,\delta_i\to 0$. 
\end{itemize}
\end{defn}

\begin{prop}
\label{E-cc}
Assume that $T$ is $(E,\eps,\rho,\rho,\M)$ almost-minimizing in $U$. If $\delta \le \rho$ is small enough, then $T$ is $(E,\eps,\delta, 0,\mathcal F)$ almost-minimizing in any open set $W$ compactly contained in $U$. 
\end{prop}

\begin{prop}
If $(V,T)\in \vz(M,\Z_2)$ is $(E,\M)$ almost-minimizing in $U$ then $(V,T)$ is $(E,\mathcal F)$ almost-minimizing in any open set $W$ compactly contained in $U$. 
\end{prop}

\begin{prop}
\label{E-ambv}
Assume that $(V,T)\in \vz(M)$ is $(E,\mathcal F)$ almost-minimizing in $U$. Then $V$ has $2a$-bounded first variation in $U$. 
\end{prop}

\begin{defn}
An element $(V,T)\in \vz(M,\Z_2)$ is called $E$-almost-minimizing in annuli if for each $x\in M$ there is a number $\rho(x) > 0$ such that $(V,T)$ is $(E,\mathcal F)$ almost-minimizing in every annulus $\an(x,s,r)$ with $s < r < \rho(x)$. 
\end{defn}

\begin{prop}
\label{E-ca}
Let $\Pi$ be the $X$-homotopy class of a map $\Phi_0\colon X\to \mathcal (B(M,\Z_2),\mathbf F)$, where $X$ is a cubical subcomplex of $I(m,k)$ for some $m,k\in \N$.  Choose a pulled-tight critical sequence $\{\Phi_i\}$ and let $\mathcal K =\mathcal K(\{\Phi_i\})$ be the critical set.  Then there is an element $(V,T)\in \mathcal K$ which is $(E,\mathcal F)$ almost minimizing in annuli. In fact, there exists some $(V,T)\in \vz(M,\Z_2)$ satisfying the following stronger property: 
\begin{itemize}
    \item[(R)]\label{item:property R for E} There is a number $N = N(m)$ depending only on $m$ such that for any collection of $N$ concentric annuli $\an(x,s_1,r_1)$, $\hdots$, $\an(x,s_N,r_N)$ with $2r_j < s_{j+1}$, $(V,T)$ is $(E,\mathcal F)$ almost minimizing in at least one of the annuli. 
\end{itemize}
\end{prop}

\begin{proof} Applying discretization \cite[Theorem 1.11]{zhou2020multiplicity}, we can find a homotopy sequence $\{\phi_i\}$ such that $L(\{\phi_i\})= L(\Pi)$ and $\mathcal K(\{\phi_i\}) = \mathcal K(\{\Phi_i\})$. 
Assume for contradiction that no element in $\mathcal K$ satisfies property (R). Then by Proposition \ref{E-cc}, no element in $\mathcal K$ satisfies property (R) with $(E,\mathcal F)$ almost-minimizing replaced by $(E,\M)$ almost-minimizing. Thus we can apply Parts 1-19 of the Almgren-Pitts combinatorial construction \cite[Theorem 4.10]{pitts2014existence} to $\{\phi_i\}$ . 
This produces a homotopic sequence $\{\psi_i\}$.  

Let us give more details. In Part 9, given $\sigma$ and a vertex of $\sigma$, one defines $T(j,1)$ and $T(j,2)$ in the same way as Pitts. Note that 
\[
\M(T(j,1)-T(j,2)) < \delta_i. 
\]
Choose $\Omega_1$ so that $\bd \Omega_1 = T(j,1)$.  Then by the isoperimetric theorem, there is a {\it unique} choice of $\Omega_2$ so that $\bd \Omega_2 = T(j,2)$ and $\vert \vol(\Omega_1)-\vol(\Omega_2)\vert < \frac{\delta_i}{2a}$. {Moreover, $\supp(\Omega_2-\Omega_1)\subset A_k$.} 
Then, using the fact that $T(j,2)$ is not $(E,\eps,\delta_i,\delta_i,\mathcal M)$ almost minimizing in an annulus $a_k\subset A_k$, we obtain a sequence $(\Omega_q)_{q=2}^{3^{N_1}}$ such that 
\begin{itemize}
\item[(i)] $\operatorname{supp}(\Omega_2-\Omega_q) \subset a_k$ for all $q$,
\item[(ii)] $2 a  \vert \vol(\Omega_2)-\vol(\Omega_q)\vert  \le  \M(\bd \Omega_2) - \M(\bd \Omega_q) + \delta_i$ for all $i$, 
\item[(iii)] $\mathcal \M(\bd \Omega_2,\bd \Omega_{q}) < \delta_i$ for all $q$,
\item[(iv)]  $\M(\bd \Omega_q) \le \M(\bd \Omega_2) + \delta_i$ for all $q$,
\item[(v)] $\M(\bd \Omega_{3^{N_1}}) \le \M(\bd\Omega_2)-\eps$. 
\end{itemize}
Note then that 
\begin{itemize}
\item[(i)] $\operatorname{supp}(\Omega_1-\Omega_q) \subset A_k$ for all $q$,
\item[(ii)] $2 a  \vert \vol(\Omega_1)-\vol(\Omega_q)\vert  \le  \M(\bd\Omega_1) - \M(\bd\Omega_q) + 3\delta_i$ for all $i$, 
\item[(iii)] $\mathcal \M(\bd \Omega_1,\bd \Omega_{q}) < 3\delta_i$ for all $q$,
\item[(iv)]  $\M(\bd \Omega_q) \le \M(\bd \Omega_1) + 3\delta_i$ for all $q$,
\item[(v)] $\M(\bd \Omega_{3^{N_1}}) \le \M(\bd \Omega_1)-\eps/2$. 
\end{itemize}
We set $T(j,q) = \bd \Omega_q$. 
Again, we emphasize that all properties in the above list continue to hold if we replace $\Omega_q$ by $M\setminus \Omega_q$ for $q=1,\hdots,3^{N_1}$. 
One then continues through Part 19 as in \cite{pitts2014existence}.

Now in Part 20, one can use Proposition \ref{E-ql} to verify that 
\[
L(\{\psi_i\}) < L(\{\phi_i\}) - \eps_1
\]
for some $\eps_1 > 0$. Again let us be more precise. Consider disjoint annuli $A_1,A_2,\hdots,A_I$. Fix $T\in \B(M,\Z_2)$.  
For each annulus, consider a sequence $T_k(q) = T_k(j,q)$ as above. 
For a choice of integers $q_1,\hdots,q_N\in \{1,\hdots,3^{N_1}\}$, one then considers a cycle of the form 
\[
T^* = T\rest (M\setminus (A_1\cup \hdots \cup A_I)) + \sum_{k=1}^I T_k(q_k)\rest A_k
\]
and needs to estimate the difference $E(T) - E(T^*)$. Choose $\Omega$ so that $T = \bd \Omega$. Then, for each $k$ and $q$, let $\Omega_k(q)$ be the sets that were selected above which satisfy $\bd \Omega_k(q) = T_k(q)$ and were used to define $T_k(q)$. If $\Omega_k(1) = \Omega$, then set $\Theta_k(q) = \Omega_k(q)$ for all $q$. Otherwise, if $\Omega_k(1) = M\setminus \Omega$, then set $\Theta_k(q) = M\setminus \Omega_k(q)$ for all $q$. Then we can apply Proposition \ref{E-ql} with $S_k = T_k(q_k)$ and $\Theta_k = \Omega_k(q_k)$ to estimate $E(T) - E(T^*)$, as needed. 

Applying interpolation \cite[Theorem 1.12]{zhou2020multiplicity} to $\{\psi_i\}$ then gives a sequence 
\[
\{\Psi_i \colon X\to (\C(M),\mathbf F)\}
\]
in the $(X,Z)$-homotopy class $\Pi$ with $\sup_{x\in X} F(\Psi_i(x)) < L(\Pi) - \eps_2$ for all large $i$. This is a contradiction. 
\end{proof}

\subsubsection{A Constrained Minimization Problem}  Consider the following constrained minimization problem.  Assume that $T$ is $(E,\eps,\delta,0,\mathcal F)$ almost minimizing in $U$.  Let $K$ be a compact subset of $U$. Choose a set $\Omega\in \mathcal C(M)$ such that $\bd \Omega = T$.  Let $\mathcal A = \mathcal A(T,K,\delta)$ be the set of all $S\in \B(M,\Z_2)$ such that there is a sequence $(\Omega_i)_{i=1,\hdots,k}$ in $\C(M)$ satisfying 
\begin{itemize}
\item[(i')] $\Omega_1 = \Omega$,
\item[(ii')] $\operatorname{supp}(\Omega-\Omega_i) \subset K$ for all $i$,
\item[(iii')] $2 a \vert \vol(\Omega)-\vol(\Omega_i)\vert \le  \M(\bd \Omega) - \M(\bd \Omega_i)$ for all $i$, 
\item[(iv')] $\mathcal F(\bd \Omega_i,\bd\Omega_{i+1}) < \delta$ for all $i$,
\item[(v')]  $\M(\bd \Omega_i) \le \M(\bd \Omega) + \delta$ for all $i$,
\end{itemize}
and $S = \bd \Omega_k$.
We aim to minimize area among elements of $\mathcal A$.  

\begin{prop}
There exists $S\in \mathcal A$ such that $\M(S)= \inf_{R\in \mathcal A} \M(R)$. 
\end{prop}

\begin{prop} 
\label{E-mam} 
Assume $T=\bd \Omega$ is $(E,\eps,\delta,0,\mathcal F)$ almost-minimizing in $U$ and let $K$ be a compact subset of $U$.  Let $S$ be an area minimizer in $\mathcal A$. Then $S$ is $(E,\eps,\delta,0,\mathcal F)$ almost-minimizing in $U$. 
\end{prop}

\begin{prop}
Let $S$ be an area minimizer in $\mathcal A$.   Then $S$ locally minimizes area with respect to volume preserving modifications in the interior of $K$. Moreover, $S$ is smooth with constant mean curvature $\vert H\vert \le 2a$ in $K$.  The surface $S$ is volume preserving stable for the area functional in the interior of $K$. 
\end{prop}

\subsubsection{Replacements}

Next we turn to the construction of replacements. 

\begin{prop}
Assume that $V$ has $2a$-bounded first variation, and also that $(V,T)$ is $(E,\mathcal F)$ almost-minimizing in an open set $U$. Let $K$ be a compact subset of $U$.  There exists an element $(V^*,T^*)\in \vz(M)$ called a replacement for $(V,T)$ such that 
\begin{itemize}
\item[(i)] $V \rest (M-K) = V^*\rest (M-K)$,
\item[(ii)] $\|V\|(M) = \|V^*\|(M)$,
\item[(iii)] $(V^*,T^*)$ is $(E,\mathcal F)$ almost-minimizing in $U$,
\item[(iv)] $V^*$ has $2a$-bounded first variation,
\item[(v)] $(V^*,T^*)$, when restricted to the interior of $K$, is a limit of some solutions to the constrained minimization problems. 
\end{itemize}
\end{prop}

\begin{prop}
Let $(V^*,T^*)$ be a replacement for $(V,T)$ in $K$.  Then in the interior of $K$, either
\begin{itemize}
\item[(i)] the varifold $V^*$ is induced by a smooth, almost-embedded CMC hypersurface $\Sigma$ with {$|H|\leq 2a$} and multiplicity one; or
\item[(ii)] the support of $V^*$ is a smooth, embedded (not necessarily connected) minimal hypersurface $\Sigma$ and $V^*$ is induced by the connected components of $\Sigma$ equipped with some integer multiplicities. 
\end{itemize}
Moreover, if $\Omega^*\in \mathcal C(M)$ satisfies $\bd \Omega^* = T^*$, then $\bd \Omega^*$ coincides with the odd multiplicity components of $V^*$ in the interior of $K$. 
{Finally, in the interior of $K$, there is a curvature estimate 
\[
\|A_\Sigma\|^2(x) \le \frac{C}{\dist(x,\bd K)^2}
\]
where $C$ is a constant that depends only on $M$, $a$, and an upper bound for $\|V\|(M)$. }
\end{prop}

\begin{proof}
In the interior of $K$, the replacement $V^*$ is a limit of smooth, embedded, volume preserving stable CMC hypersurfaces with uniformly bounded mean curvature. Hence the regularity follows from the curvature estimates for volume preserving stable CMCs.
\end{proof}

We conclude by investigating the tangent cones to $E$ almost-minimizers.  We show that if $(V,T)$ is $(E,\mathcal F)$ almost-minimizing in annuli, then every tangent cone to $V$ is an integer multiple of a plane.

\begin{prop}
Assume that $(V,T)$ has $2a$-bounded first variation and is $(E,\mathcal F)$-almost minimizing in a set $U$. Then for any sequence $p_i \to p \in U$ and any sequence of scales $r_i \to 0$, every varifold limit $\overline V = (\eta_{p_i,r_i})_\sharp V$ is an integer multiple of a complete, embedded minimal hypersurface. Moreover, every varifold tangent to $V$ is an integer multiple of a hyperplane. 
\end{prop}

\subsection{Regularity} 

\label{section:E-regularity}

The goal of this subsection is to prove the regularity of the min-max pair $(V,T)$. As in the case of the $F$ functional, the key point is to ensure that successive replacements on overlapping annuli have the same mean curvature.

Given a smooth almost-embedded surface $\Sigma$ with constant mean curvature, recall that $\mathcal R(\Sigma)$ and $\mathcal S(\Sigma)$ denote respectively the set of embedded points and the touching set of $\Sigma$.

The next proposition shows that, once a replacement has non-zero mean curvature, all subsequent replacements must also have matching non-zero mean curvature. It does not seem straightforward to show that if the initial replacement is minimal with multiplicity, then subsequent replacements will also be minimal with multiplicity. Fortunately, the following is already enough to prove regularity. 

\begin{prop}
\label{prop:mean curvature match2}
Assume $(V,\Omega)$ is $E$ almost-minimizing in annuli and fix a point $p\in \supp \|V\|$.  Fix sufficiently small numbers $s_1 < r_1$ and let $(V^*,\Omega^*)$ be a replacement for $(V,\Omega)$ in $\an(p,s_1,r_1)$. Let $\Sigma^*$ be the smooth, almost-embedded hypersurface inducing $V^*$ in $\an(p,s_1,r_1)$ and assume that $\Sigma^*$ has non-zero mean curvature $h_0$.  Choose $s_1 < r_2 < r_1$ so that $\bd B_{r_2}(p)$ intersects $\Sigma^*$ and the countable union of manifolds containing $\mathcal S(\Sigma^*)$ transversally and so that $\bd B_{r_2}(p) \cap \mathcal R(\Sigma^*)$ is non-empty. Fix any $s < s_1$ and let $(V^{**}_s,\Omega^{**}_s)$ be a replacement for $(V^*,\Omega^*)$ in $\an(p,s,r_2)$. Let $\Sigma^{**}_s$ be the smooth hypersurface inducing $V^{**}_s$ in $\an(p,s,r_2)$. Then $\Sigma^{**}_s$ also has mean curvature $h_0$. 
\end{prop}

\begin{proof} There are several other possibilities to rule out. First, suppose for contradiction that $\Sigma^{**}_s$ is a multiplicity one CMC with constant mean curvature $h_1\neq h_0$.  Then we can argue as in Proposition \ref{prop:mean curvature match1} to construct a volume preserving deformation which decreases the area. This contradicts the $E$-almost-minimizing property. 

Second, suppose for contradiction that $\Sigma^{**}_s$ is minimal. If the multiplicity of $\Sigma^{**}_s$ is one, then $\Sigma^{**}_s$ coincides with the boundary of $\Omega^{**}_s$ and we can again construct a volume preserving deformation which decreases the area. As before, this violates the almost-minimizing property. Finally suppose that $\Sigma^{**}_s$ has higher multiplicity. Then the multiplicity function $\theta$ for $V^{**}_r$ satisfies $\theta(x) \ge 2$ for all $x\in \Sigma^{**}_s$. By assumption, there is a point $y \in \bd B_{r_2}\cap \mathcal R(\Sigma^*)$. We know any varifold tangent to $V^{**}_s$ at $y$ is an integer multiple of a plane. By considering $\Sigma^*$, we see that in fact the only possible varifold tangent is $T_y\Sigma^*$ with multiplicity one.  Thus, by transversality, there must be a sequence of points $x_i\in \Sigma^{**}_s$ with $x_i \to y$. But, by upper semi-continuity of the density, this implies that $\theta_{V^{**}_s}(y) \ge 2$, contradicting that the varifold tangent $T_y\Sigma^*$ has multiplicity one. Therefore this also cannot occur.
\end{proof}

We can now complete the regularity argument.

\begin{prop}
\label{E-regularity}
Assume that $(V,T)$ satisfies $\delta E(V,T) = 0$. Further suppose that $(V,T)$ is $(E,\mathcal F)$ almost-minimizing in annuli. Choose a set $\Omega$ with $\bd \Omega = T$ and let $H = -f'(\vol(\Omega))$. 
\begin{itemize}
    \item[(i)] If $H \neq 0$, then there exists a smooth, almost-embedded (not necessarily connected) CMC hypersurface $\Lambda = \bd \Omega$, which has mean curvature $H$ with respect to the normal pointing into $\Omega$. Moreover, there exist a (possibly empty) collection of minimal hypersurfaces $\Sigma_1,\hdots,\Sigma_k$ and a collection of multiplicities $m_1,\hdots,m_k\in \N$ such that 
\[
V =  \vert \Lambda\vert + \sum_{i=1}^k m_i \vert \Sigma_i\vert.
\]
The hypersurfaces $\Lambda,\Sigma_1,\hdots,\Sigma_k$ are all disjoint.
\item[(ii)] If $H=0$, then there exists a collection of smooth, embedded minimal hypersurfaces $\Lambda_1,\hdots,\Lambda_q$ such that $\bd \Omega = \Lambda_1\cup \hdots \cup \Lambda_q$. Moreover, there exist a (possibly empty) collection of minimal hypersurfaces $\Sigma_1,\hdots,\Sigma_k$ and a collection of multiplicities $\ell_1,\hdots,\ell_q$, $m_1,\hdots,m_k\in \N$ such that 
\[
V = \sum_{i=1}^q \ell_i\vert \Lambda_i\vert + \sum_{i=1}^k m_i \vert \Sigma_i\vert.
\]
The hypersurfaces $\Lambda_1,\hdots,\Lambda_q,\Sigma_1,\hdots,\Sigma_k$ are all disjoint.
\end{itemize}
\end{prop}

\begin{proof}
The argument is slightly more complicated than for the $F$ functional, because Proposition \ref{prop:mean curvature match2} is weaker than Proposition \ref{prop:mean curvature match1}.  To start, choose a point $p\in \supp \|V\|$. Choose very small numbers $s_1 < r_1$ and let $(V^*,\Omega^*)$ be a replacement for $(V,\Omega)$ in $\an(p,s_1,r_1)$. There are now several cases to consider. 

First, suppose that $(V^*,\Omega^*)$ is induced by a smooth, almost-embedded CMC hypersurface $\Sigma^*$ with non-zero mean curvature in $\an(p,s_1,r_1)$. Then choose $s_1 < r_2< r_1$  so that $\bd B_{r_2}(p)$ is transverse to $\Sigma^*$ and the union of manifolds containing $\mathcal S(\Sigma^*)$ and so that $\bd B_{r_2}(p)\cap \mathcal R(\Sigma^*)$ is non-empty.  Then for $s < s_1$ let $(V^{**}_s,\Omega^{**}_s)$ be a replacement for $(V^*,\Omega^*)$ in $\an(p,s,r_2)$. According to Proposition \ref{prop:mean curvature match2}, $V^{**}_s$ must be induced by a multiplicity one hypersurface $\Sigma^{**}_s$ whose mean curvature matches that of $\Sigma^*$.  One can now proceed exactly as in \cite{zhou2019min} to deduce that $V$ is induced by a smooth, almost-embedded CMC hypersurface with multiplicity one in a small neighborhood of $p$. 

Second, suppose that $(V^*,\Omega^*)$ is induced by a smooth, embedded, minimal hypersurface $\Sigma^*$, possibly with multiplicity. Choose $s_1< r_2< r_1$ so that $\bd B_{r_2}(p)$ intersects $\Sigma^*$ transversally. Then for $s<s_1$ let $(V^{**}_s,\Omega^{**}_s)$ be a replacement for $(V^*,\Omega^*)$ in $\an(p,s,r_2)$. If for every $s < s_1$, $V^{**}_s$ is induced by a smooth minimal surface with multiplicity, then we can argue exactly as in the usual Almgren-Pitts construction to deduce that $V$ is induced by a smooth, embedded minimal hypersurface with multiplicity in a small neighborhood of $p$. 

The remaining possibility is that for some choice of $s_2< s_1$, the varifold $V^{**}_{s_2}$ is induced by a smooth, almost-embedded CMC hypersurface $\Sigma^{**}_{s_2}$ with non-zero mean curvature. In this case, we will take a {\it third} replacement. Choose $s_2 < r_3< r_2$ so that $\bd B_{r_3}(p)$ intersects $\Sigma^{**}_{s_2}$ and the countable union of manifolds containing $\mathcal S(\Sigma^{**}_{s_2})$ transversally and so that $\bd B_{r_3}(p)\cap \mathcal R(\Sigma^{**}_{s_2})$ is non-empty. Then for any $s< s_2$, let $(V^{***}_s,\Omega^{***}_s)$ be a replacement for $(V^{**}_{s_2},\Omega^{**}_{s_2})$ in $\an(p,s,r_3)$. According to Proposition \ref{prop:mean curvature match2} (applied with $(V,\Omega) = (V^{**}_{s_2},\Omega^{**}_{s_2})$), the varifold $V^{***}_{s}$ is induced by a smooth, almost-embedded hypersurface $\Sigma^{***}_s$ in $\an(p,s,r_3)$ with non-zero mean curvature matching that of $\Sigma^{**}_{s_2}$. Arguing exactly as in \cite{zhou2019min}, one can show that $\Sigma^{**}_{s_2}$ and $\Sigma^{***}_s$ glue smoothly along $\bd B_{r_3}(p)$. One can then let $s\to 0$ and continue to argue exactly as in \cite{zhou2019min} to deduce that $V$ is induced by a smooth, almost-embedded CMC hypersurface with multiplicity one in a small neighborhood of $p$. 

We have now proven the following local regularity: for every $p\in \supp \|V\|$ there is an $r > 0$ such that in $B_r(p)$ the varifold $V$ is induced by either 
    \begin{itemize}
        \item[(i)] a smooth, almost-embedded, multiplicity one CMC surface $\Lambda$ with non-zero mean curvature, or
        \item[(ii)] a collection of smooth, embedded minimal hypersurfaces with integer multiplicities. 
    \end{itemize}
The local regularity then implies that $V$ is induced by a collection of multiplicity one, almost-embedded CMCs together with a collection of smooth, embedded minimal hypersurfaces with multiplicities. 

Next note that since $\|\, \vert T\vert\, \| \le \|V\|$, the constancy theorem implies that $\vert T\vert$ is induced by some subcollection of the surfaces inducing $V$. Hence we can list the components of $\supp \|V\|$ as $\Lambda_1,\hdots,\Lambda_\ell,\Sigma_1,\hdots,\Sigma_k$ so that 
\begin{gather*}
\vert T\vert = \sum_{i=1}^\ell \vert \Lambda_i\vert,\\
V = \sum_{i=1}^\ell \ell_i \vert \Lambda_i\vert + \sum_{j=1}^k m_j \vert \Sigma_j\vert.
\end{gather*}
Here $\ell_i,m_j\in \N$ are some multiplicities. 

Now we use the fact that $(V,T)$ is stationary for $E$ to deduce the remaining claims of the proposition. First note that every surface $\Sigma_j$ must be minimal. Indeed, if some surface $\Sigma_j$ was not minimal, then we could construct a local deformation near $\Sigma_j$ decreasing the area and leaving $T$ unchanged. 
Such a deformation would decrease $E$ to first order. Therefore each $\Sigma_j$ must be minimal. 

Choose a set $\Omega$ with $\bd \Omega = T$. It remains to verify that all the surfaces $\Lambda_i$ must have mean curvature $-f'(\vol(\Omega))$, computed with respect to the normal pointing into $\Omega$. If, to the contrary, $\Lambda_i$ had a different mean curvature, then we could construct a local deformation near $\Lambda_i$ decreasing the $E$ functional to first order. 
This completes the proof of the proposition. 
\end{proof}

Theorem \ref{E-min-max} now follows by combining Proposition \ref{E-pt}, Proposition \ref{E-ca}, and Proposition \ref{E-regularity}.

\section{Compactness}\label{S:compactness}

In this section, we record some compactness properties satisfied by the surfaces produced by the min-max theorems.

\subsection{Compactness for F} Assume that $f\colon [0,\vol(M)] \to \R$ satisfies \eqref{f is even} and that $h\colon M\to \R$ is a smooth Morse function satisfying property (T) \ref{property (T)}. 
Choose a sequence $\eps_k\to 0$ and let 
\[
F_k(\Omega) = \M(\bd \Omega) - \eps_k \int_\Omega h + f(\vol(\Omega)).
\]
Also define 
\[
E(T) = \M(\bd \Omega) + f(\vol(T)). 
\]
We would like to prove certain 
compactness for a sequence $(V_k,\Omega_k)$ of critical points for $F_k$ as $k\to \infty$. 

The basic strategy is to try and show that the almost-minimizing property implies volume preserving stability. As noted in Remark \ref{rem:am-stability}, it seems we cannot adapt the argument in \cite[Lemma 3.3]{wang2023existence} because of the extra volume constraint in the definition of almost-minimizers. Thus we will instead argue that the min-max hypersurfaces coincide with their replacements, which are known to be volume preserving stable.

\begin{prop}
\label{F-gluing}
Assume that $\Sigma = \bd \Omega$ is a smooth, almost-embedded hypersurface with mean curvature $\eps_k h + h_0$ for some constant $h_0$. Assume that $(\vert\Sigma\vert, \Omega)$ is $(F_k,\mathcal F)$ almost-minimizing in an open set $U$ and that $\an(p,s,r)\subset U$ is a closed annulus. Let $(V^*,\Omega^*)$ be a replacement for $(\vert \Sigma\vert,\Omega)$ in $\an(p,s,r)$. Let $\Sigma^*$ be the smooth, almost-embedded hypersurface inducing $V$ in $\an(p,s,r)$. Then $\Sigma \cap (M\setminus \an(p,s,r))$ and $\Sigma^*$ have the same mean curvature, and glue smoothly along the boundary of $\an(p,s,r)$. 
\end{prop}

\begin{proof}
    Consider any point $q\in \bd \an(p,s,r)$. Note that $(V^*,\Omega^*)$ is $(F_k,\mathcal F)$-almost-minimizing in small annuli centered at $q$. Therefore, by the same local regularity argument as in Proposition \ref{F-regularity}, it follows that $V^*$ is induced by a smooth, almost-embedded hypersurface in a neighborhood of $q$. Therefore, $\supp \|V^*\|$ is smooth in a neighborhood of $\bd \an(p,s,r)$. This implies the conclusion of the proposition.
\end{proof}

\begin{rem}
An almost-embedded hypersurface in $M$ can be viewed as an immersion $\Sigma \to M$. In the following, by a component of an almost-embedded hypersurface, we mean the image of a connnected component of $\Sigma$ under this immersion. Thus two spheres touching tangentially at a point define two components.
\end{rem}

\begin{corollary}
\label{F-am implies stability}
    Assume that $\Sigma = \bd \Omega$ is a smooth, almost-embedded hypersurface with mean curvature $\eps_k h + h_0$ for some constant $h_0$. Assume that $(\vert \Sigma\vert, \Omega)$ is $(F_k,\mathcal F)$ almost-minimizing in an open set $U$ and that $\an(p,s,r) \subset U$ is a closed annulus. Assume that a local embedded sheet $\Gamma$ of $\Sigma$ crosses the boundary of $\an(p,s,r)$. Then the component of $\Sigma$ containing this sheet is volume preserving stable for $A^{\eps_k h + h_0}$ in the interior of $\an(p,s,r)$. 
\end{corollary}

\begin{proof}
Consider a replacement $(V^*,\Omega^*)$ for $(\vert \Sigma\vert, \Omega)$ in $\an(p,s,r)$. Then $V^*$ is induced by a smooth, almost-embedded hypersurface $\Sigma^*$ in the interior of $\an(p,s,r)$. 
According to Proposition \ref{F-gluing}, $\Sigma \cap (M\setminus \an(p,s,r))$ glues smoothly to $\Sigma^*$ along the boundary of $\an(p,s,r)$. 
Now let $\Sigma_1$ be the component of $\Sigma$ containing $\Gamma$. Then by unique continuation, $\Sigma_1 \cap \ins(\an(p,s,r))$ is contained in $\Sigma^*$ and the conclusion of the proposition follows since replacements are volume preserving stable.
\end{proof}

\begin{corollary}
    \label{F-stable in small balls}
    Assume that $\Sigma = \bd \Omega$ is a smooth, almost-embedded hypersurface with mean curvature $\eps_k h + h_0$ for some constant $h_0$. Assume that $(\vert \Sigma\vert, \Omega)$ is $(F_k,\mathcal F)$ almost-minimizing in an open set $U$. There is a radius $\rho > 0$ depending only on $M$ and an upper bound $c$ for $\vert \eps_k h + h_0\vert$ such that if $s < r < \rho$ and $\an(p,s,r) \subset U$,  then $\Sigma$ is volume preserving stable for $A^{\eps_k h + h_0}$ in the interior of $\an(p,s,r)$. 
\end{corollary}

\begin{proof}
Let $(V^*,\Omega^*)$ be a replacement for $(\vert \Sigma\vert,\Omega)$ in $\an(p,s,r)$. Then $V^*$ is induced by a smooth, almost-embedded hypersurface $\Sigma^*$ in the interior of $\an(p,s,r)$. Again Proposition \ref{F-gluing} implies that $\Sigma^*$ glues smoothly to $\Sigma \cap (M\setminus \an(p,s,r))$ along the boundary of $\an(p,s,r)$ 
By the maximum principle, if $\rho$ is sufficiently small depending on $M$ and an upper bound for the mean curvature of $\Sigma$, then $\Sigma$ and $\Sigma^*$ can have no closed components contained entirely in $\an(p,s,r)$. Therefore, by unique continuation, $\Sigma = \Sigma^*$ in $\an(p,s,r)$ and the result follows. 
\end{proof}

\begin{prop}
\label{F-compactness}
    Consider a sequence $\{(V_k,\Omega_k)\}\in \vc(M)$. Assume that $(V_k,\Omega_k)$ is stationary for $F_k$ and that $(V_k,\Omega_k)$ satisfies property $\operatorname{(R)}$ \eqref{item:property R for F} for $F_k$ with an integer $m$ that does not depend on $k$. Finally, suppose that $\|V_k\|(M)$ is bounded uniformly above, and that the varifolds $V_k$ have $c$-bounded first variation for a uniform constant $c>0$. Then, up to a subsequence, $(V_k,\bd\Omega_k) \to (V,T) \in \vz(M,\Z_2)$. Moreover, $(V,T)$ is stationary for $E$ and we have the following alternative. Either:
        \begin{itemize}
        \item[(i)] The varifold $V$ is induced by a smooth, almost-embedded, multiplicity one constant mean curvature surface $\Lambda$ with non-zero mean curvature; or
        \item[(ii)] The varifold $V$ is induced by a collection of minimal hypersurfaces with multiplicities. 
    \end{itemize}
    The following property also holds:
    \begin{itemize}
        \item[(R')]\label{item:property R'} There is a number $N = N(m)$ depending only on $m$ and a number $\rho > 0$ depending only on $c$, such that for any collection of $N$ concentric annuli $\an(x,s_1,r_1)$, $\hdots$, $\an(x,s_N,r_N)$ with $2r_j < s_{j+1}$ and $r_N < \rho$, $\supp \|V\|$ is volume preserving stable for the area as an immersion in at least one of the annuli. 
    \end{itemize}

\end{prop}

\begin{proof}
    Choose a sequence $\{(V_k,\Omega_k)\}$ as in the statement of the theorem and let $T_k = \bd \Omega_k$. Then $(V_k,T_k)\in \vz(M,\Z_2)$ and so, up to a subsequence, $(V_k,T_k)\to (V,T)$ in $\vz(M,\Z_2)$ since $\|V_k\|(M)$ is bounded uniformly from above. It follows immediately from the first variation formulas \eqref{eq:1st variation for E}\eqref{eq:1st variation for F} that $(V,T)$ is stationary for $E$. 

    Note that Proposition \ref{F-regularity} implies that $V_k$ is induced by a smooth, almost-embedded hypersurface $\Sigma_k = \bd \Omega_k$ with mean curvature $H_k = \eps_k h + h_k$, where $h_k$ is a constant. Moreover, the constants $h_k$ are uniformly bounded. Passing to a further subsequence, we can suppose that $h_k \to h_\infty$. 

    To prove property (R'), let $\rho$ be the constant from Corollary \ref{F-stable in small balls}, which depends only on $c$. Consider a collection of $N$ concentric annuli $\an(x,s_1,r_1)$, $\hdots$, $\an(x,s_N,r_N)$ with $2r_j < s_{j+1}$ and $r_N < \rho$.  Since $(V_k,\Omega_k)$ satisfies property (R), there is a choice of $j\in \{1,\hdots,N\}$ and a subsequence $(V_{k_i},\Omega_{k_i})$ such that $(V_{k_i},\Omega_{k_i})$ is $(F_{k_i},\mathcal F)$-almost minimizing in $\an = \an(x, s_j, r_j)$ for every $i$.  According to Proposition \ref{F-stable in small balls}, $\Sigma_{k_i}$ is volume preserving stable for $A^{\eps_k h + h_{k_i}}$ as an immersion in $\an$. 

    This stability implies a curvature estimate for $\Sigma_{k_i}$ in $\an$. If $h_\infty \neq 0$, then passing to a further subsequence, $\Sigma_{k_i}$ converges smoothly with multiplicity one to an almost-embedded limit $\Sigma$ with mean curvature $h_\infty$ in $\an$. If $h_\infty = 0$, then passing to a further subsequence, $\Sigma_{k_i}$ converges smoothly, possibly with multiplicity, to a collection of minimal hypersurfaces in $\an$. In both cases, this implies that $\an \cap \supp \|V\|$ is smooth and volume preserving stable for the area as an immersion. This proves that property (R') holds. 

    Finally, we need to show that either (i) or (ii) holds. First suppose that $h_\infty \neq 0$. Then the above argument shows that for any collection of $N$ concentric annuli $\an(x,s_1,r_1)$, $\hdots$, $\an(x,s_N,r_N)$ with $2r_j < s_{j+1}$ and $r_N < \rho$, V is induced by a smooth, almost-embedded, volume preserving stable CMC with mean curvature $h_\infty$ in one of the annuli. This implies that there is a radius $\sigma(x) > 0$ such that for any annulus $\an(x,s,r)$ with $s< r<\sigma(x)$, the varifold $V$ is induced by a smooth, almost-embedded, volume preserving stable CMC with mean curvature $h_\infty$ in $\an(x,s,r)$. This property implies that alternative (i) holds by the removable singularity part in \cite[Section 6, Step 4]{zhou2019min}. 
    
    If instead $h_\infty = 0$, then the above argument shows that for any collection of $N$ concentric annuli $\an(x,s_1,r_1)$, $\hdots$, $\an(x,s_N,r_N)$ with $2r_j < s_{j+1}$ and $r_N < \rho$, V is induced by a collection of smooth, volume preserving stable minimal surfaces with multiplicity in one of the annuli. This implies that there is a radius $\sigma(x) > 0$ such that for any annulus $\an(x,s,r)$ with $s< r<\sigma(x)$, the varifold $V$ is induced by a collection of smooth, volume preserving stable minimal surfaces with multiplicity. This property implies that alternative (ii) holds by the removable singularity part of the Almgren-Pitts regularity argument \cite{pitts2014existence}.  
\end{proof}

\subsection{Compactness for E}

Next we record another compactness property for the $E$ functional. Consider a sequence of smooth functions $f_k\colon[0,\vol(M)]\to \R$ satisfying (\ref{f is even}) and set 
\[
E_k(T) = \M(T) + f_k(\vol(T)).
\]
Let $\{(V_k,T_k)\}$ be a sequence such that $(V_k, T_k)$ is a critical point for $E_k$ for each $k\in\mathbb N$.  We show that if the mean curvature of the CMC portion of $(V_k,T_k)$ stays bounded then there is a convergence subsequence.  We will omit the proofs in this subsection, since they are very similar to the previous subsection. 

\begin{prop}
\label{E-gluing}
    Assume that $(V,T)$ has the regularity described in Proposition \ref{E-regularity}. Assume that $(V,T)$ is $(E_k,\mathcal F)$ almost-minimizing in an open set $U$ and that $\an(p,s,r)\subset U$ is a closed annulus. Let $(V^*,T^*)$ be a replacement for $(V,T)$ in $\an(p,s,r)$. Then $\supp \|V^*\| \cap \an(p,s,r)$ and $\supp \|V\| \cap (M\setminus \an(p,s,r))$ glue smoothly along the boundary of $\an(p,s,r)$. 
\end{prop}

\begin{corollary}
\label{E-am implies stability}
    Assume that $(V,T)$ has the regularity described in Proposition \ref{E-regularity}. Assume that $(V,T)$ is $(E_k,\mathcal F)$ almost-minimizing in an open set $U$ and that $\an(p,s,r) \subset U$ is a closed annulus. Assume that a local embedded sheet of $\supp \|V\|$ crosses the boundary of $\an(p,s,r)$. Then the component of $\supp \|V\|$ containing this sheet is volume preserving stable for the area in the interior of $\an(p,s,r)$. 
\end{corollary}

\begin{corollary}
    \label{E-stable in small balls} 
    Assume that $(V,T)$ has the regularity described in Proposition \ref{E-regularity}. Assume that $(V,T)$ is  $(E_k,\mathcal F)$ almost-minimizing in an open set $U$. There is a radius $\rho > 0$ depending only on $M$ and an upper bound $c$ for the first variation of $V$ such that if $s < r < \rho$ and $\an(p,s,r) \subset U$,  then every component of $\supp \|V\|$ intersecting $\an(p,s,r)$ has the same mean curvature. Moreover, $\supp \|V\|$ is volume preserving stable for the area in $\an(p,s,r)$. 
\end{corollary}

\begin{prop}
\label{E-compactness}
    Consider a sequence $(V_k,T_k)\in \vz(M,\Z_2)$. Assume that $(V_k,T_k)$ is stationary for $E_k$ and that $(V_k,T_k)$ satisfies property $\operatorname{(R)}$ \eqref{item:property R for E} for $E_k$ with an integer $m$ that does not depend on $k$. Suppose that $\|V_k\|(M)$ is uniformly bounded, and that the all the varifolds $V_k$ have $c$-bounded first variation for a uniform constant $c$. Then, up to a subsequence, $(V_k,T_k)\to (V,T)\in \vz(M,\Z_2)$. Moreover, $(V,T)$ has the regularity described in Proposition \ref{E-regularity} and the following property holds:
        \begin{itemize}
        \item[(R')] There is a number $N = N(m)$ depending only on $m$ and a number $\rho > 0$ depending only on $c$, such that for any collection of $N$ concentric annuli $\an(x,s_1,r_1)$, $\hdots$, $\an(x,s_N,r_N)$ with $2r_j < s_{j+1}$ and $r_N < \rho$, there is an annulus $\an$ in the collection such that every component of $\supp \|V\|$ intersecting $\an$ has the same mean curvature and is volume preserving stable for the area in $\an$. 
    \end{itemize}
    
\end{prop}

\section{Constructing Half-Volume CMCs}\label{S:constructiong half-volume CMCs}

The goal of this section is to prove Theorem \ref{main} and Corollary \ref{main2}. To begin, we recall the notion of the half-volume spectrum.

\subsection{The Half-Volume Spectrum} 
In \cite{mazurowski2023half}, the authors introduced the half-volume spectrum of a manifold. This is similar to the usual volume spectrum introduced by Gromov, except that all hypersurfaces in the construction are now additionally required to enclose half the volume of $M$. 

The Almgren isomorphism theorem \cite{almgren1962homotopy} implies that the space $\mathcal B(M,\Z_2)$ is weakly homotopy equivalent to $\RP^\infty$. Marques and Neves \cite{marques2021morse} later gave a substantially simpler proof of this fact. In \cite{mazurowski2023half}, it is shown that $\mathcal H(M,\Z_2)$ is homotopy equivalent to $\mathcal B(M,\Z_2)$. In particular, this means that the cohomology ring $H^*(\mathcal H(M,\Z_2),\Z_2)$ is isomorphic to $\Z_2[\lambda]$ where the generator $\lambda$ is in degree 1. 

\begin{defn}
Fix an integer $p\in \N$.  Let $X$ be a cubical complex. A flat continuous map $\Phi\colon X \to \mathcal H(M,\Z_2)$ is called a half-volume $p$-sweepout if $\Phi^* (\lambda^p) \neq 0$ in $H^*(X,\Z_2)$.  
\end{defn}

\begin{defn}
A flat continuous map $\Phi\colon X \to \mathcal H(M,\Z_2)$ is said to have {\it no concentration of mass} provided 
\[
\lim_{r\to 0} \left[\sup_{x\in X} \sup_{y\in M} \M(\Phi(x)\rest B(y,r))\right] = 0. 
\]
\end{defn}

\begin{defn}
Let $\mathcal Q_p$ be the set of all half-volume $p$-sweepouts of $M$ with no concentration of mass. Note that different half-volume $p$-sweepouts are allowed to have different domains. 
\end{defn}

\begin{defn}
For an integer $p\in \N$, the half-volume $p$-width of $M$ is 
\[
\tilde \omega_p = \inf_{\Phi\in \mathcal Q_p}\left[\sup_{x\in \text{dom}(\Phi)} \M(\Phi(x))\right].
\]
The sequence $\{\tilde \omega_p\}_{p\in \N}$ is called the half-volume spectrum of $M$. 
\end{defn}

The next proposition says that the definition of the half-volume $p$-width is unchanged if we restrict to only those half-volume $p$-sweepouts whose domain is a subset of $I(2p+1,k)$ for some $k\in\N$. This was proved for (ordinary) $p$-sweepouts by Y. Li \cite{li2023existence}. 

\begin{prop}
\label{prop:dimension}
For every $\eps > 0$, there is half-volume $p$-sweepout $\Psi\colon Y \to \mathcal H(M,\Z_2)$ whose domain $Y$ is a cubical subcomplex of $I(2p+1,k)$ for some $k\in \N$ and which satisfies
\[
\sup_{y\in Y} \M(\Psi(y)) \le \tilde \omega_p + \eps.
\]
Moreover, $\Psi$ has no concentration of mass. 
\end{prop}

\begin{proof}
This follows from the proof of \cite[Proposition 3.2]{li2023existence}. Indeed, given $\eps > 0$, we can find a half-volume $p$-sweepout $\Phi\colon X\to \mathcal H(M,\Z_2) \subset \mathcal B(M,\Z_2)$ with no concentration of mass and with 
\[
\sup_{x\in X} \M(\Phi(x)) \le \tilde \omega_p + \eps. 
\]
Applying the construction in \cite{li2023existence} verbatim to $\Phi$ produces a cubical subcomplex $Y$ of $I(2p+1,k)$ and a flat continuous $p$-sweepout $\Psi\colon Y\to \mathcal B(M,\Z_2)$ with the property that $\Psi(Y) \subset \Phi(X)$. In particular, the inclusion $\Psi(Y)\subset \Phi(X)$ implies that $\Psi$ takes values in $\mathcal H(M,\Z_2)$ and that 
\[
\sup_{y\in Y} \M(\Psi(y))\le \tilde \omega_p + \eps,
\]
and that $\Psi$ has no concentration of mass.
\end{proof}

\subsection{The Penalized Functionals}

Intuitively, we expect that each $\tilde \omega_p$ is associated with a critical point of the area functional restricted to the space of half-volume cycles. Therefore, $\tilde \omega_p$ should be achieved by a half-volume CMC. Rather than directly developing a min-max theory for the area on $\mathcal H(M,\Z_2)$, we first apply min-max on all of $\mathcal B(M,\Z_2)$ with a functional that consists of the area plus a term that penalizes the distance to the space of half-volume cycles. 

Fix a closed Riemannian manifold $(M^{n+1},g)$ with dimension $3\le n+1\le 7$.  Let $\hv = \frac{1}{2}\vol(M)$. For each $k\in \N$ define  $f_k\colon [0,\vol(M)]\to \R$ by 
\[
f_k(v) = k(v-\hv)^2. 
\]
Then define $E_k:\mathcal B(M,\Z_2)\to \R$ by 
\[
E_k(T) =  \M(T) + f_k(\vol(\Omega)) = \M(T) + k (\vol(\Omega) - \hv)^2
\]
where $\Omega\in \C(M)$ satisfies $\bd \Omega = T$. Since $\vert\vol(\Omega) - \hv\vert = \vert \vol(M\setminus \Omega) - \hv\vert$, it is easy to see that the definition does not depend on the choice of $\Omega$. 

Fix $p\in \N$.  For each $k\in \N$, select a half-volume $p$-sweepout $\Phi_k^* \colon X_k \to \mathcal H(M,\mathcal F,\Z_2)$ with no concentration of mass for which 
\[
\sup_{x\in X_k} \M(\Phi_k^*(x)) \le \tilde \omega_p + \frac{1}{k}. 
\] 
By Proposition \ref{prop:dimension}, we can further ensure that $X_k$ is a cubical subcomplex of $I(2p+1,\ell)$ for some $\ell\in \N$. 
By applying discretization \cite[Theorem 1.11]{zhou2020multiplicity} followed by interpolation \cite[Theorem 1.12]{zhou2020multiplicity}, we can replace $\Phi_k^*$ with a new $\mathbf F$-continuous map $\Phi_k \colon X_k\to \mathcal B(M,\mathbf F, \Z_2)$ such that 
\begin{gather*}
\sup_{x\in X_k} \M(\Phi_k(x)) \le \tilde \omega_p + \frac{2}{k},\quad
\sup_{x\in X_k} \vert \vol(\Phi_k(x)) - \hv\vert \le \frac{1}{k}. 
\end{gather*}
Recall here that $\vol(\Phi_k(x))$ stands for $\vol(\Omega)$ for any set $\Omega$ with $\bd \Omega=\Phi_k(x)$.

Let $\Pi_k$ be the $X_k$-homotopy class of the map $\Phi_k$.  Observe that  
\begin{align*}
L^{E_{k}}(\Pi_k) \le \sup_{x\in X_k} E_{k}(\Phi_k(x)) = \sup_{x \in X_k} \M(\Phi_k(x)) + k \sup_{x\in X_k} \vert \vol(\Phi_k(x)) - \hv\vert^2 \le \tilde \omega_p + \frac 3 k. 
\end{align*} 
Now suppose for contradiction that 
\[
\liminf_{k\to \infty} L^{E_{k}}(\Pi_k) < \tilde \omega_p.
\]
Then, after passing to a subsequence, we can find an $\eta > 0$ and maps $\Psi_k\colon X_k \to \mathcal B(M,\Z_2)$ homotopic to $\Phi_k$ such that 
\[
\sup_{x\in X_k} {E_{k}}(\Psi_k(x))\le \tilde \omega_p - \eta
\]
for all $k$. Note in particular this implies that for any $x\in X_k$ we have 
\[
\M(\Psi_k(x)) \le \tilde \omega_p - \eta, \quad \text{ and } \quad 
 \vert \vol(\Psi_k(x)) - \hv\vert \le \sqrt{\frac{\tilde \omega_p}{k}}
\]
Now recall the deformation retraction $\theta\colon \mathcal B(M,\Z_2)\times[0,1]\to \mathcal H(M,\Z_2)$ constructed in \cite{mazurowski2023half}.

\begin{lem}\label{lem:deformation retraction map}
There is a continuous function $w\colon [0,\infty) \to [0,\infty)$ with $w(0) = 0$ such that 
\[
\M(\theta(T,1)) \le \M(T) + w(\vert\vol(T) - \hv\vert)
\]
for any $T \in \mathcal B(M,\Z_2)$.
\end{lem}

\begin{proof}
Let $f \colon M \to [0,1]$ be a Morse function and let $U_t = \{f < t\}$ for $t\in [0,1]$. We first claim that for every $\eps > 0$ there is a $\delta > 0$ such that for all $t\in [0,1]$ and all $W\subset U_t$ we have
\[
\vol(W)\ge \vol(U_t) - \delta \implies  \M(\bd W) \ge \M(\bd U_t) - \eps.
\]
Suppose to the contrary that for some $\eps > 0$ there is no $\delta > 0$ which makes the assertion true. Then for $\delta = \frac 1 n$ there exist $t_n \in [0,1]$ and $W_n \subset U_{t_n}$ with $\vol(W_n) > \vol(U_{t_n}) - \frac 1 n$ but $\M(\bd W_n) < \M(\bd U_{t_n}) - \eps$. After passing to a subsequence, we can suppose that $t_n \to t_\infty \in [0,1]$ and that $W_n \to W_\infty \in \mathcal C(M)$. Note that $W_\infty \subset U_{t_\infty}$ and $\vol(W_\infty) = \vol(U_{t_\infty})$ and so $W_\infty = U_{t_\infty}$. Since $t\mapsto \M(\bd U_t)$ is continuous, this implies that 
\begin{align*}
    \M(\bd U_{t_\infty}) &\le \liminf \M(W_n)\\
    & \le \liminf \M(\bd U_{t_n}) - \eps = \M(\bd U_{t_\infty}) - \eps. 
\end{align*}
This is a contradiction. 

Now, to prove the lemma, it suffices to show that for each $\eps > 0$ there is a $\delta > 0$ such that 
\[
\vert \vol(T) - \hv\vert < \delta \implies \M(\theta(T,1)) \le \M(T) + \eps. 
\]
So let $\eps > 0$ be given and choose $\delta$ according to the previous claim. Assume that $T$ satisfies $\vert \vol(T) - \hv\vert < \delta$. Choose a set $\Omega$ with $\vol(\Omega) \le \hv$ and $\bd \Omega = T$ and note that $\hv - \vol(\Omega) < \delta$. The set $\theta(T,1)$ is defined as $\bd(\Omega \cup U_t)$ where $t$ is chosen so that $\vol(\Omega \cup U_t) = \hv$. Note that  $\Omega \cap U_t \subset U_t$ and that $\vol(\Omega \cap U_t) = \vol(\Omega)+\vol(U_t)-\vol(\Omega\cup U_t) \ge \vol(U_t) - \delta$. According to the previous claim, this implies that 
\[
\M(\bd (\Omega \cap U_t)) \ge \M(\bd U_t) - \eps. 
\]
Finally note that 
\begin{align*}
\M(\bd (\Omega \cup U_t)) &\le \M(\bd \Omega) + \M(\bd U_t) - \M(\bd (\Omega \cap U_t))\\
&\le \M(\bd \Omega) + \eps,
\end{align*}
as needed. 
\end{proof}

By the lemma, for $k$ large enough, the map $\Xi_k \colon X_k \to \mathcal H(M,\Z_2)$ given by $\Xi_k(x) = \theta(\Psi_k(x),1)$ is a half-volume $p$-sweepout with no concentration of mass (guaranteed by properties of $\theta$ \cite{mazurowski2023half}) which satisfies 
\[
\sup_{x\in X_k} \M(\Xi_k(x)) \le \tilde \omega_p - \eta + w\left(\sqrt{\frac{\tilde \omega_p}{k}}\right) < \tilde \omega_p.
\]
This contradicts the definition of $\tilde \omega_p$. It follows that 
\[ L^{E_{k}}(\Pi_k) \to \tilde \omega_p \]
as $k\to \infty$.

\subsection{Bounding the Mean Curvature}  We assume that $3\le \dim(M)\le 5$ from this point onward. 
Again $p\in \N$ is fixed. Applying the $E_k$-min-max theorem in the homotopy class $\Pi_k$ gives the existence of critical points for $E_k$. We would like to show that these critical points converge to a regular limit as $k\to \infty$. The key point is to show that the mean curvature of the critical points stays uniformly bounded as $k\to \infty$. Then we can appeal to the compactness results of the previous section. 

The following diameter bound for stable CMCs is due to Elbert-Nelli-Rosenberg \cite{elbert2007stable}.

\begin{theorem}[\cite{elbert2007stable}]
\label{theorem:diameter}
Assume $M^{n+1}$ is a Riemannian manifold of dimension $3\le n+1\le 5$. Asssume the sectional curvatures of $M$ are bounded below by $-\kappa$ for some $\kappa\ge 0$. Let $\Sigma$ be a stable, immersed $H$-CMC hypersurface in $M$. There is a constant $c = c(n,H,\kappa)$ so that $\dist_\Sigma(q,\bd \Sigma) \le c$ for all $q\in \Sigma$ provided $\vert H\vert > 2\sqrt{\kappa}$.  
\end{theorem}

The next proposition will be used to show that half-volume CMCs with bounded area cannot consist entirely of tiny components. 

\begin{prop}
\label{small-components}
Let $M^{n+1}$ be a closed Riemannian manifold. Let $\Lambda = \bd \Omega$ be a smooth, almost-embedded  hypersurface with non-vanishing mean curvature. Assume that the mean curvature vector of $\Lambda$ points consistently into or consistently out of $\Omega$. Assume that $\area(\Lambda) \le A$ and that $\frac{1}{3}\vol(M) \le \vol(\Omega)\le \frac{2}{3}\vol(M)$. There is a positive constant $\delta > 0$, depending only on $M$ and $A$, such that $\Lambda$ has a  component with extrinsic diameter at least $\delta$.
\end{prop}

\begin{proof}
Let $I_M$ denote the isoperimetric profile of $M$. By the asymptotics of the isoperimetric profile for small volumes, there are constants $c > 0$ and $V_0 > 0$ such that $I_M(v) \ge c v^{n/(n+1)}$ for all $v \in (0,V_0)$. Let $\alpha = \frac{1}{3}\vol(M)$.  
Shrinking $V_0$ if necessary, we can suppose that 
\[
v^{\frac n{n+1}} \ge \frac{2 A v}{c\alpha}
\]
for all $v\in (0,V_0)$. 
Now choose $\delta > 0$ so that any closed hypersurface in $M$ with extrinsic diameter less than $\delta$ encloses a region with volume less than $V_0$. 

Suppose for contradiction that every component of $\Lambda$ has extrinsic diameter smaller than $\delta$. Replacing $\Omega$ by $M\setminus \Omega$ if necessary, we can suppose that $\Omega$ is the union of the small volume regions enclosed by these components. Let $J$ be the number of connected components of $\Lambda$, and list the volumes of these components as $v_1,\hdots,v_J$. Note that 
\[
\sum_{j=1}^J v_j \ge \alpha
\]
and that $v_j \le V_0$ for all $j=1,\hdots,J$. Thus the area of $\Lambda$ satisfies  
\[
\area(\Lambda) \ge c \sum_{j=1}^J v_j^{n/(n+1)} \ge \frac{2A}{\alpha} \sum_{j=1}^J v_j \ge 2 A. 
\]
This contradicts the definition of $A$. 
\end{proof} 

{Now we can prove that the mean curvature does not blow up. Let $\mathscr S_k$ be the set of all $(V,T) \in \vz(M,\Z_2)$ such that 
\begin{itemize}
\item[(i)] $(V,T)$ is stationary for $E_k$,
\item[(ii)] $(V,T)$ satisfies property (R) \eqref{item:property R for E} for $E_k$ with $m=2p+1$,
\item[(iii)] $\|V\|(M) \le \tilde \omega_p + 3/k$. 
\item[(iv)] $\vert \vol(T) - \hv\vert \le \sqrt{2\tilde \omega_p/k}$.  
\end{itemize}
Note that every $(V, T)$ in $\mathscr S_k$ is regular in the sense of Proposition \ref{E-regularity}.
Also note that the critical points produced by the $E_k$ min-max theory in the homotopy class $\Pi_k$ belong to $\mathscr S_k$ when $k$ is sufficiently large. Define 
\[
H_k = \inf\{c \ge 0: \text{every $(V,T) \in \mathscr S_k$ has $c$-bounded first variation}\}.
\]
and note that $H_k \le k\vol(M)$. 

\begin{prop}
\label{h-bound}
Assume that $3\le n+1 \le 5$. Then $H = \sup\{H_k:k\in \N\}$ is finite. 
\end{prop}
}

\begin{proof} 
Suppose to the contrary that $H_k\to \infty$. Then we can find $(V_k,T_k)\in \mathscr S_k$ such that the mean curvature of the CMC portion $\Lambda_k$ of $\supp \|V_k\|$ goes to infinity as $k\to \infty$. 
By assumption, each $(V_k,T_k)$ has the following property:
\begin{itemize}
\item[(R)] For any collection of $N = N(2p+1)$ concentric annuli $\an (x, s_1,r_1)$, $\hdots$, $\an(x,s_N,r_N)$ with $2r_j < s_{j+1}$ for all $j$, $(V_k,T_k)$ is $(E_k,\mathcal F)$-almost-minimizing in at least one of the annuli. 
\end{itemize}
Let $\eta_{x,\rho}$ denote the map which rescales by a factor of $1/\rho$ centered at $x$. Choose $\rho_k\to 0$ so that $\eta_{x,\rho_k}(\Lambda_k)$ 
has mean curvature 1. 

Choose $\delta > 0$ according to Proposition \ref{small-components}. {Note that $\Lambda_k$ has a uniform upper bound on area, and $\Lambda_k = \bd \Omega_k$ with $\vol(\Omega_k)\to \hv$.} Therefore, to get a contradiction, it suffices to show that, for some large $k$, every connected component of $\Lambda_k$ has extrinsic diameter less than $\delta$. If this is not the case, then (passing to a subsequence) we can find points $x_k \in \Lambda_k$ such that the connected component $\Gamma_k$ of $\Lambda_k$ containing $x_k$ has extrinsic diameter at least $\delta$. Fix some positive numbers $s_1 < r_1 < s_2 < r_2 < \hdots < s_N < r_N$ satisfying $2r_j < s_{j+1}$ for all $j$. Choose a sequence $\sigma_k \to \infty$ such that $\rho_k \sigma_k \to 0$. 

For each $k$, consider the collection of $N(2p+1)$-concentric annuli $\an(x_k,s_1/\sigma_k,r_1/\sigma_k)$, $\hdots$, $\an(x_k,s_N/\sigma_k,r_N/\sigma_k)$ and note that this collection is admissible for property (R). 
Therefore (passing to a subsequence) there is a $j\in \{1,\hdots,N\}$ such that $(V_k,T_k)$ is $(E_k,\mathcal F)$-almost-minimizing 
in
\[
\an_k = \an(x_k,s_j/\sigma_k,r_j/\sigma_k)
\]
for all $k$. 
For notational convenience, let $s = s_j$ and $r = r_j$.
Then for $k$ large enough, since $\Gamma_k$ has extrinsic diameter at least $\delta$, it follows that $\Gamma_k$ crosses the boundary of $\an_k$.  Hence, by Proposition \ref{E-am implies stability}, it follows that $\Gamma_k$ is volume preserving stable in the interior of $\an_k$.

Define 
\[
C_k = \bd B\left(x_k,\frac{s+r}{2\sigma_k}\right) \subset \an_k. 
\]
Choose a point $y_k \in \Gamma_k \cap C_k$, which must exist since $\Gamma_k$ has diameter at least $\delta$. 
Now consider the balls 
\[
B_k = B\left(y_k,\frac{r-s}{2\sigma_k}\right) \subset \an_k, \quad D_k = B\left(y_k,\frac{r-s}{4\sigma_k}\right) \subset B_k. 
\]
Note that $\Gamma_k$ must intersect $\bd B(y_k,t)$ for every $0 < t < (r-s)/(2\sigma_k)$ since $\Gamma_k$ is connected and contains both $x_k$ and $y_k$. 

Define $\Lambda_k' = \eta_{y_k,\rho_k}(\Lambda_k)$, $\Gamma_k' = \eta_{y_k,\rho_k}(\Gamma_k)$, $B_k' = \eta_{y_k,\rho_k}(B_k)$, and $D_k'=\eta_{y_k,\rho_k}(D_k)$. 
Note that $\Lambda_k'$ has constant mean curvature 1 and is volume preserving stable in 
$
B_k' 
$.
Because $\rho_k \sigma_k\to 0$, the balls $B_k'$ resemble large, nearly Euclidean balls when $k$ is large. Note that $\Gamma_k'$ must be strongly stable as a 1-CMC immersion in either $D_k'$ or $B_k' \setminus D_k'$, but this violates the diameter estimate of Theorem \ref{theorem:diameter}. Indeed, there are points $q_k\in \Gamma_k' \cap D_k'$ with $\dist(q_k,\bd D_k')\to \infty$ and points $z_k \in \Gamma_k' \cap [B_k' \setminus D_k']$ with $\dist(z_k, \bd B_k' \cup \bd D_k')\to \infty$.
\end{proof}

This implies the following compactness property. 

\begin{prop}
    Assume that $3\le n+1\le 5$.  Consider a sequence $(V_k,T_k)\in \mathscr S_k$. Then, up to a subsequence, $(V_k,T_k)\to (V,T)$ in $\vz(M,\Z_2)$. Moreover, $(V,T)$ has the regularity described in Proposition \ref{E-regularity} and the mean curvature of the CMC part of $(V,T)$ is at most $H$. We have $\vol(T) = \hv$. Finally, $(V,T)$ satisfies property (R') \eqref{item:property R'} with $m = 2p+1$ and $\rho$ depending only on $H$. 
\end{prop}

\begin{proof}
    The mean curvature of $(V_k,T_k)$ is uniformly bounded by $H$. Choosing $\Omega_k$ so that $\bd \Omega_k = T_k$, one has $\vol(\Omega_k)\to \hv$. Therefore the result follows from Proposition \ref{E-compactness}. 
\end{proof}

\subsection{The Lifting Construction} 
At this point, we assume the metric $g$ on $M$ is generic so that, by the results of Appendix \ref{generic metrics}, we can assume that $M$ is bumpy, that $M$ has no half-volume minimal hypersurfaces, that half-volume CMCs are isolated in $M$, and that every closed almost-embedded half-volume CMC in $M$ is actually embedded.

Define $H = \sup\{H_k: k\in \N\}$ as in the previous subsection. Choose a scale $\rho > 0$ depending on $H$ according to Proposition \ref{h-bound}. Let $\mathscr S$ be the set of all pairs $(V,T)\in \vz(M,\Z_2)$ such that 
\begin{itemize}
\item[(i)] $\|V\|(M) \le \tilde \omega_p$,
\item[(i)] $(V,T)$ satisfies property (R') \eqref{item:property R'} with $m=2p+1$ and scale $\rho$, 
\item[(ii)] $\vol(T) = \hv$,
\item[(iii)] There is a smooth, embedded CMC hypersurface $\Lambda$ with non-zero mean curvature $\vert H_\Lambda\vert \le H$.
Moreover, there there exist a (possibly empty) collection of closed minimal hypersurfaces $\Sigma_1,\hdots,\Sigma_k$ and a collection of multiplicities $m_1,\hdots,m_k\in \N$ such that 
\[
V = \vert \Lambda\vert + \sum_{i=1}^k m_i \vert \Sigma_i\vert.
\]
Here the hypersurfaces $\Lambda,\Sigma_1,\hdots,\Sigma_k$ are all disjoint.
\end{itemize}

\begin{prop}
    Assume the metric $g$ on $M$ is generic and that $3\le n+1 \le 5$. Then the set $\mathscr S$ is finite. 
\end{prop}

\begin{proof}
    Suppose to the contrary that $\mathscr S$ is not finite. Then there is a sequence of distinct elements $(V_k,T_k)\in \mathscr S$. Let $\Lambda_k$ be the CMC portion of $(V_k,T_k)$. Since the mean curvature of $\Lambda_k$ stays uniformly bounded, we can argue as in Proposition \ref{E-compactness} to see that that, up to a subsequence, $(V_k,T_k)\to (V,T) \in \mathscr S$. Let $\Lambda$ denote the CMC portion of $(V,T)$. Note, in particular, $\Lambda$ is guaranteed to be embedded with non-zero mean curvature, since the metric $g$ admits no half-volume minimal hypersurfaces, and every almost-embedded half-volume CMC is embedded. 

    Since $\Lambda$ is genuinely embedded and the convergence of the CMC portion occurs with multiplicity one, Allard's regularity theorem implies that the convergence $\Lambda_k \to \Lambda$ is actually smooth. This contradicts that half-volume CMCs are isolated in $M$; see Corollary \ref{half-volume-isolated}.
    Thus $\Lambda_k = \Lambda$ for sufficiently large $k$. 

    Now let $V_k'$ be the minimal portion of $V_k$ and let $V'$ be the minimal portion of $V$. Then Property (R') implies that for each $x\in \supp \|V'\|$ there is a $\rho(x) > 0$ such that if $s < r <\rho(x)$ then the convergence of $V_k'$ to $V'$ is smooth in $\an(x,s,r)$. This implies that the convergence $V_k'\to V'$ is actually smooth away from finitely many points. Therefore, $V_k' = V'$ for sufficiently large $k$, as otherwise it would be possible to extract a Jacobi field on $V'$. This shows that in fact $(V_k,T_k) = (V,T)$ for sufficiently large $k$ and this is a contradiction. 
\end{proof}

The following lemma is due to Marques and Neves. 

\begin{lem}[\cite{marques2017existence} Corollary 3.6]
Let $\mathcal T$ be a finite subset of $\mathcal B(M,\Z_2)$. If $\eta > 0$ is sufficiently small, then every map $\Phi\colon S^1 \to B^{\mathcal F}_\eta(\mathcal T)$ is homotopically trivial. 
\end{lem}

Enumerate the set $\mathscr S = \{(W_1, S_1),\hdots,(W_Q,S_Q)\}$ and let $\mathcal T=\{S_1,\hdots,S_Q\}$. Choose $\eta > 0$ according to the previous lemma.   Recall the set $\mathscr S_k$ defined in the previous subsection. By Proposition \ref{h-bound} and Proposition \ref{E-compactness}, it follows that for any $\eta > 0$ there is a $K\in \N$ such that $\mathscr F((V,T), \mathscr S) \le \eta/4$ for all $(V,T)\in \mathscr S_k$ provided $k\ge K$. Note in particular that if $\mathscr F((V,T),\mathscr S) \le \eta/4$ then $\mathcal F(T,\mathcal T) \le \eta/4$. 

Now fix some $k\ge K$. Consider a pulled-tight $E_k$-min-max sequence $\Psi_{k,j}\colon X_k \to \B(M,\mathcal F,\Z_2)$ for the homotopy class $\Pi_k$. Choose a sequence $\ell_j\to \infty$ so that 
\[
\sup \{\mathbf F(\Psi_{k,j}(x),\Psi_{k,j}(y)): \alpha \in X_k(\ell_j),\ x,y\in \alpha\} < \frac{\eta}{2}. 
\]
Let $\tilde Z_{k,j} \subset X_k$ be the union of all the cells $\alpha \in X_k(\ell_j)$ such that 
\[
\mathscr F(\Psi_{k,j}(x), \mathscr S) \ge \eta/2 
\]
for all vertices $x\in \alpha$. Here and in the following, we use $\Psi_{k,j}(x)$ to denote its image under the natural inclusion $\B(M, \Z_2)\to \vz(M, \Z_2)$. Then 
\[ 
\mathscr F(\Psi_{k,j}(x), \mathscr S) \ge \eta/2 
\]
for all $x\in \tilde Z_{k,j}$. Let $\tilde Y_{k,j} = \cl{X_k \setminus \tilde Z_{k,j}}$. Consider the restricted sequence $\Psi_{k,j}|_{\tilde Z_{k,j}}$. Then either 
\begin{itemize}
\item[(i)] $L^{E_k}(\{\Psi_{k,j}|_{\tilde Z_{k,j}}\}) < L^{E_k}(\Pi_k)$, or
\item[(ii)] no element in $\mathcal K(\{\Psi_{k,j}|_{\tilde Z_{k,j}}\})$ satisfies property (R) for $E_k$ with $m=2p+1$.  
\end{itemize}
Indeed, if both (i) and (ii) failed then there would be some element 
\[
(W,S)\in \mathcal K(\{\Psi_{k,j}|_{\tilde Z_{k,j}}\}) \cap \mathscr S_k.
\]
But then $\mathscr F((W,S),\mathscr S) \le \eta/4$ contradicting the choice of $\tilde Z_{k,j}$. 

We can now argue exactly as in \cite[Section 5, Step 2]{zhou2020multiplicity} to obtain a new min-max sequence $\Psi_{k,j}'\colon X_k \to \mathcal B(M,\Z_2)$ such that 
\begin{itemize}
    \item[(i)] $X_k$ can be decomposed into $Y_{k,j}$ and $Z_{k,j}$ where $Y_{k,j} = \cl{X_k\setminus Z_{k,j}}$. Moreover, we have 
\[
\qquad \qquad(\Psi_{k,j}'|_{Y_{k,j}})^* \lambda = 0 \text{ in } H^1(Y_i,\Z_2), \quad (\Psi_{k,j}'|_{Z_{k,j}})^*(\lambda^{p-1}) \neq 0 \text{ in } H^{p-1}(Z_{k,j},\Z_2)
\]
provided $j$ is large enough, 
\item[(ii)] $L^{E_k}(\{\Psi_{k,j}'\}) = L^{E_k}(\Pi_k)$,
\item[(iii)] $\limsup_{j\to \infty} \sup\{ E_k(\Psi_{k,j}'(z)): z\in Z_{k,j}\} < L^{E_k}(\Pi_k)$.
\end{itemize}
Define $\widetilde E_k \colon \C(M) \to \R$ by 
\[
\widetilde E_k(\Omega) = \M(\bd \Omega) + f_k(\vol(\Omega)).  
\]
Let $\widetilde X_k$ be the double cover of $X_k$ associated to the cohomology class $(\Psi_{k,j}')^*\lambda$.  Let $\widetilde \Pi_{k,j}$ be the $(\widetilde X_k, \widetilde Z_{k,j})$ relative homotopy class of $\Psi_{k,j}'$. 
By the same argument as \cite[Section 5, Step 3]{zhou2020multiplicity}, if $j$ is large enough then 
\[
L^{\widetilde E_k}(\widetilde\Pi_{k,j}) > \sup_{z\in \widetilde Z_{k,j}} \widetilde E_k(\Psi_{k,j}'(z)).
\]
Moreover, we have $L^{\widetilde E_k}(\widetilde\Pi_{k,j})\to L^{E_k}(\Pi_k)$ as $j\to \infty$.

\subsection{Conclusion of the Proof} 
Let $h\colon M\to\R$ be a smooth Morse function satisfying property (T) \eqref{property (T)}. Choose a sequence $\eps_j \to 0$. Define $F_{k,j}\colon \C(M)\to \R$ by 
\[
F_{k,j}(\Omega) = \M(\bd \Omega) -\eps_j \int_{\Omega} h + f_k(\vol(\Omega)).
\]
Then $F_{k, j}(\Omega) = \widetilde{E}_k(\Omega) - \eps_j \int_{\Omega} h$, and we have  
\[
L^{F_{k,j}}(\widetilde \Pi_{k,j}) > \sup_{z\in \widetilde Z_{k,j}} F_{k,j}(\Psi_{k,j}'(z))
\]
for large enough $j$, and $L^{F_{k,j}}(\widetilde \Pi_{k,j}) \to L^{E_k}(\Pi_k)$ as $j\to \infty$. Therefore, by applying min-max theory for $F_{k,j}$ (Theorem \ref{F-min-max}) in the relative homotopy class $\tilde \Pi_{k,j}$, there exist critical points $(V_{k,j},\Omega_{k,j})$ for $F_{k,j}$ with $F_{k,j}(V_{k,j},\Omega_{k,j})\to L^{E_k}(\Pi_k)$.  By Proposition \ref{F-compactness}, after passing to a subsequence, $(V_{k,j},\bd \Omega_{k,j})\to (V_k,T_k) \in \vz(M,\Z_2)$, $(V_k,T_k)$ is stationary for $E_k$, $E_k(V_k,T_k) = L^{E_k}(\Pi_k)$, and $V_k$ is induced by either 
\begin{itemize}
        \item[(i)] a smooth, almost-embedded, multiplicity one constant mean curvature surface with non-zero mean curvature; or
        \item[(ii)]  a collection of minimal hypersurfaces with multiplicities. 
    \end{itemize}
It remains to take a limit as $k\to \infty$.  Again the key step is to show that the mean curvature doesn't blow up. 

\begin{prop}
\label{h-bound2}
Consider the pairs $(V_k,T_k)\in \vz(M,\Z_2)$ defined above. Then the mean curvature of $\Lambda_k = \supp \|V_k\|$ stays uniformly bounded as $k\to \infty$.
\end{prop}

\begin{proof}
    Suppose to the contrary that the mean curvature $H_k$ of $\Lambda_k$ goes to $\infty$. Choose $\rho_k\to 0$ so that $\eta_{x,\rho_k}(\Lambda_k)$ has mean curvature 1. 
    
    Let $\Lambda_{k,j}$ denote $\supp \|V_{k,j}\|$ and let $H_{k,j}$ be the mean curvature of $\Lambda_{k,j}$. Then for fixed $k$, we have $H_{k,j}\to H_k$ as $j\to \infty$. Choose $\delta > 0$ according to Proposition \ref{small-components}. Then, if $j$ is large enough depending on $k$, $\Lambda_{k,j}$ has a connected component $\Gamma_{k,j}$ with extrinsic diameter at least $\delta$. Pick a point $x_{k,j}\in \Gamma_{k,j}$. 

    Let $N = N(2p+1)$ and fix $s_1 < r_1 < s_2<r_2< \hdots < s_N < r_N$ with $2 r_i < s_{i+1}$. Choose a sequence $\sigma_k\to \infty$ such that $\rho_k\sigma_k\to 0$. Now fix an integer $k\in \N$. For each $j$, consider the collection of concentric annuli $\an(x_{k,j}, s_1/\sigma_k, r_1/\sigma_k)$, $\hdots$, $\an(x_{k,j}, s_N/\sigma_k,r_N/\sigma_k)$.  Since $(V_{k,j},\Omega_{k,j})$ satisfies property (R) \eqref{item:property R for F}, there is an $i(k)\in \{1,\hdots,N\}$ and a (non-relabeled) subsequence of $j$'s such that $(V_{k,j},\Omega_{k,j})$ is $(F_{k,j},\mathcal F)$-almost minimizing in the annulus $\an(x_{k,j},s_{i(k)}/\sigma_k,r_{i(k)}/\sigma_k)$ for all $j$. Since $\Gamma_{k,j}$ has extrinsic diameter at least $\delta$, it follows that $\Gamma_{k,j}$ must cross the boundary of $\an(x_{k,j},s_{i(k)}/\sigma_k,r_{i(k)}/\sigma_k)$. By Corollary \ref{F-am implies stability}, it follows that $\Gamma_{k,j}$ is volume preserving stable for $A^{\eps_j h}$ in $\an(x_{k,j},s_{i(k)}/\sigma_k,r_{i(k)}/\sigma_k)$.

    Now consider the rescaled surfaces $\Gamma_{k,j}' = \eta_{x_{k,j},\rho_k}(\Gamma_{k,j})$. Passing to a further subsequence in $j$, we can suppose that $x_{k,j}\to x_k$ as $j\to \infty$. For a fixed $k$, the surfaces $\Gamma_{k,j}'$ have a uniform upper bound on mean curvature, a uniform upper bound on area, and are volume preserving stable for $A^{\eps_j \rho_k h}$ in $\an_{k,j}' = \eta_{x_{k,j},\rho_k}(\an(x_{k,j},s_{i(k)}/\sigma_k, r_{i(k)}/\sigma_k))$. Therefore, these surfaces have a uniform curvature estimate, 
    and converge smoothly to a volume preserving stable CMC hypersurface $\Gamma_k'$ in $\an_k' = \eta_{x_k,\rho_k}(\an(x_{k},s_{i(k)}/\sigma_k, r_{i(k)}/\sigma_k))$ with mean curvature 1. Note that the surface $\Gamma_k'$ is connected and intersects both the inner and outer boundary of $\an_k'$. Therefore, we can now argue exactly as in Proposition \ref{h-bound} to get a contradiction for large enough $k$. 
\end{proof}

Once we know that the mean curvature doesn't blow up, it follows that $(V_k,T_k)$ satisfies property (R') \eqref{item:property R'} at a uniform scale $\rho$ that does not depend on $k$. This implies that $(V_k,T_k)\to (V,T)$,  $\|V\|(M) \le \tilde \omega_p$, $\vol(T) = \hv$, and $V$ is induced by either 
\begin{itemize}
        \item[(i)] a smooth, almost-embedded, multiplicity one constant mean curvature surface with non-zero mean curvature; or
        \item[(ii)]  a collection of minimal hypersurfaces with multiplicities. 
    \end{itemize}
Choose $\Omega$ so that $\bd \Omega = T$. Note that case (ii) cannot occur because then some collection of minimal hypersurfaces in $M$ would bound $\Omega$. But we have assumed the metric $g$ is generic so that this cannot happen by Proposition \ref{half-volume-minimal}. Therefore case (i) occurs and $\bd \Omega$ is an almost-embedded constant mean curvature hypersurface. 

It remains to show that $\|V\|(M) = \tilde \omega_p$. Since $E_k(V_k,T_k)\to \tilde \omega_p$, it is equivalent to show that 
\[
k(\vol(\Omega_k)-\hv)^2 \to 0
\]
as $k\to \infty$, where $\bd \Omega_k = T_k$. But this follows from the fact that $(V_k,T_k)$ is stationary for $E_k$ and the fact that the mean curvature does not blow up. Indeed, let $H_k$ denote the mean curvature of $\supp \|V_k\|$. Then we know that $
\vert H_k\vert = 2k\vert \vol(\Omega_k)-\hv\vert$ since $(V_k,T_k)$ is stationary for $E_k$. 
It follows that 
\[
k(\vol(\Omega_k)-\hv)^2 = \frac{\vert H_k\vert^2}{4k},
\]
and this goes to 0 as $k\to \infty$ since $H_k$ is uniformly bounded. Therefore, $\|V\|(M) = \tilde \omega_p$, as needed. This proves Theorem \ref{main}, and Corollary \ref{main2} follows immediately.

\subsection{Positive Ricci Curvature}

In this subsection, we prove Theorem \ref{positive-Ricci} on the existence of half-volume CMCs in manifolds with positive Ricci curvature.  Fix a closed manifold $M^{n+1}$ of dimension $3\le n+1\le 5$ and let $g$ be a metric on $M$ with positive Ricci curvature. Choose a sequence of generic metrics $g_i$ on $M$ such that $g_i\to g$ smoothly. 

Fix an integer $p\in \N$. We claim that $\tilde \omega_p(M,g_i)\to \tilde \omega_p(M,g)$ as $i\to \infty$. To see this, first note that there are constants $\eta_i\to 1$ such that 
\begin{equation}
\label{eq:equivalent metrics}
\eta_i^{-2} g_i(v,v) \le g(v,v)\le \eta_i^2 g_i(v,v)
\end{equation}
for all $v \in TM$. Also notice that if $\vol(\Omega,g_i) = \frac{1}{2}\vol(M,g_i)$ then 
\begin{align*}
\left\vert \vol(\Omega,g) - \frac{1}{2}\vol(M,g)\right\vert &\le \left\vert \vol(\Omega,g)-\vol(\Omega,g_i)\right\vert  + \left\vert \frac{1}{2}\vol(M,g_i)-\frac{1}{2}\vol(M,g)\right\vert.
\end{align*}
Thus, by equation (\ref{eq:equivalent metrics}), there are constants $a_i \to 0$ as $i\to \infty$ such that if $\vol(\Omega,g_i) = \frac{1}{2}\vol(M,g_i)$ then $\vert \vol(\Omega,g) - \frac{1}{2}\vol(M,g)\vert \le a_i$. 
Now, given any $\eps > 0$, one can select a half-volume $p$-sweepout $\Phi_i$ of $(M,g_i)$ with 
\[
\sup_{x\in \text{dom}(\Phi_i)} \M(\Phi_i(x)) \le \tilde \omega_p(M,g_i) + \eps. 
\]
Note that $\Phi_i$ may not be a half-volume $p$-sweepout of $(M,g)$. However, if $\theta$ denotes the deformation retraction to half-volume cycles in $(M,g)$ in Lemma \ref{lem:deformation retraction map}, then $\theta(\Phi_i(\cdot),1)$ is a half-volume $p$-sweepout of $(M,g)$. Moreover, according to Lemma \ref{lem:deformation retraction map}, we have 
\[
\sup_{x\in \text{dom}(\Phi_i)} \M(\theta(\Phi_i(x),1)) \le \tilde \omega_p(M,g_i) + \eps + w(a_i). 
\]
Hence letting $i\to \infty$ and then letting $\eps\to 0$, one obtains $\tilde \omega_p(M,g) \le \liminf_{i\to \infty} \tilde \omega_p(M,g_i)$. The reverse inequality $\tilde \omega_p(M,g) \ge \limsup_{i\to\infty} \tilde \omega_p(M,g_i)$ can be proved similary. 

Applying Theorem \ref{main} to $(M,g_i)$ gives the existence of a Caccioppoli set $\Omega_i$ with $\vol(\Omega_i,g_i) = \frac{1}{2}\vol(M,g_i)$ such that $\Sigma_i = \bd \Omega_i$ is smooth and almost-embedded with constant mean curvature $H_i\neq 0$. Moreover, we have $\area(\Sigma_i,g_i) = \tilde \omega_p(M,g_i)$.  Arguing similarly to Proposition \ref{h-bound} and Proposition \ref{h-bound2}, one can show that the mean curvature $H_i$ stays uniformly bounded as $i\to \infty$. Thus, after passing to a subsequence, we can suppose that $H_i\to H\in \R$ as $i\to \infty$. 

If $H\neq 0$, then arguing as in the proof of Proposition \ref{E-compactness}, one can show that $(\vert \Sigma_i\vert,\Omega_i)$ converges to a limit $(V,\Omega)$ with $\vol(\Omega,g) = \frac{1}{2}\vol(M,g)$ and $\|V\|(M,g) = \tilde \omega_p(M,g)$. Moreover, $\Sigma = \bd \Omega$ is a smooth, almost-embedded CMC with mean curvature $H$ and $V$ is induced by $\Sigma$ with multiplicity one. This concludes the proof in the case where $H\neq 0$. 

Suppose instead that $H = 0$. Again arguing as in the proof of Proposition \ref{E-compactness}, one can show that $(\vert \Sigma_i\vert,\Omega_i)$ converges to a limit $(V,\Omega)$ with $\vol(\Omega,g) = \frac{1}{2}\vol(M,g)$ and $\|V\|(M,g) = \tilde \omega_p(M,g)$. Moreover, $V$ is induced by a collection of smooth, embedded, connected, pairwise disjoint minimal hypersurfaces with multiplicity, and some subcollection of these minimal hypersurfaces bounds $\Omega$. Since $g$ has positive Ricci curvature, the Frankel property implies that every pair of minimal hypersurfaces in $(M,g)$ must intersect. Thus $\Sigma = \bd \Omega$ is in fact a connected minimal hypersurface and $V = m\vert \Sigma\vert$ for some $m\in \N$. 

To complete the proof of Theorem \ref{positive-Ricci}, we need to show that $m = 1$. Note that the convergence $\Sigma_i\to \Sigma$ is locally smooth away from finitely many points. If $m\ge 3$, then for each $i$ it is possible to find two sheets of $\Sigma_i$ with mean curvature pointing in the same direction. This implies that $\Sigma$ carries a positive solution $\varphi$ to $J_\Sigma \varphi = 0$; c.f. \cite[Theorem 4.1, Part 4]{zhou2020multiplicity}. Since $(M,g)$ has positive Ricci curvature, this is impossible. Finally, the case $m=2$ is also impossible. Indeed, if $m=2$ then it is straightforward to see that either $\Omega = \emptyset$ or $\Omega = M$ which violates the half-volume constraint. Thus $m=1$, as needed. 

\appendix
\section{Condition (T)}
\label{h-generic}

Let $h\colon M\to \R$ be a smooth Morse function. Given a regular point $x\in M$ for $h$, let $\Gamma(x)$ be the level set of $h$ passing through $x$. Then define $v(h,x)$ to be the vanishing order at $x$ of the mean curvature $H_{\Gamma(x)}$, regarded as a function on $\Gamma(x)$.   
Recall that our min-max theory requires $h$ to satisfy the following property (Definition \ref{property (T)}):
\begin{itemize}
\item[(T)] For every regular point $x$ of $h$, we have $v(h,x) < \infty$. 
\end{itemize}
This property is used to show that the touching set of an almost-embedded $(h+h_0)$-PMC is contained in a countable union of $(n-1)$-dimensional manifolds. The goal of this appendix is to show that an arbitrary Morse function $h$ can be perturbed slightly in the smooth topology to a nearby Morse function that satisfies condition (T).

Note the following lemmas. 

\begin{lem}
\label{usc1}
Assume that $u\colon M\to\R$ is a smooth Morse function. Let $C$ be the set of all critical points of $u$. Let $G \subset M\setminus C$ be a closed set such that for each $x\in G$ we have $v(u,x) < \infty$. Then there is a finite $k\in \N$ such that $v(u,x)\le k$ for all $x\in G$. Moreover, if $w$ is a sufficiently small smooth perturbation of $u$ then $v(w,x) \le k$ for all $x\in G$ as well. 
\end{lem}

\begin{proof}
The function $v(u,x)$ is upper semicontinuous, i.e., 
\[
\limsup_{w \to ,u\, y\to x} v(w,y) \le v(u,x). 
\]
If there were points $x_n\in G$ with $v(u,x_n)\to \infty$, then there would be a convergent subsequence $x_{n_k}\to x\in G$. But then the semi-continuity implies $v(u,x) = \infty$, contrary to assumption. Likewise, if there is a sequence $u_n\to u$ such that there are points $x_n\in G$ with $v(u_n,x_n)\ge k+1$, then we can find a convergent subsequence $x_{n_k}\to x \in G$. The semi-continuity implies that $v(u,x) \ge k+1$, contrary to assumption. 
\end{proof}

\begin{lem}
\label{usc2}
Assume that $u\f M\to \R$ is a smooth Morse function. Let $C$ be the set of all critical points of $u$. Let $G\subset M\setminus C$ be a closed set such that for each $x\in G$ we have $v(u,x)<\infty$. Then there is an $\eps > 0$ such that $v(u,x)< \infty$ for every $x\in M$ with $\dist(x,G)\le \eps$. 
\end{lem}

\begin{proof}
Again this follows from the upper semicontinuity of the vanishing order. Suppose to the contrary that there is no such $\eps$. Then there is a sequence of points $x_n \to x\in G$ with $v(u,x_n)\to \infty$. But this implies that $v(u,x) = \infty$. 
\end{proof}

Now we can prove the main result of this appendix. 

\begin{prop}
Let $h\colon M\to \R$ be a smooth Morse function. Then there exist Morse functions satisfying property $\operatorname{(T)}$ which are arbitrarily close to $h$ in $C^\infty(M)$. 
\end{prop}

\begin{proof}
The topology on $C^\infty(M)$ is complete and metrizable. Fix a complete metric $\rho$ on $C^\infty(M)$ inducing the topology of $C^\infty(M)$.  Consider the following two constructions. 

{\bf Construction 1:} Let $u\f M\to \R$ be a Morse function. Let $B=\{b_1,\hdots,b_k\}$ be the set of all critical values of $u$ and let $C$ denote the set of all critical points of $u$. Fix a small radius $r > 0$. Let $\Sigma_i$ be a smooth, closed, embedded surface which coincides with $\{u=b_i\}$ outside $N_r(C) = \{x\in M: \dist(x,C)< r\}$. Choose a short time $s$, and let $\Sigma_{i,t}$, $t\in[0,s]$ be the surface obtained by running mean curvature flow for time $t$ starting from $\Sigma_i$. There exists a smooth, time-dependent vector field $X_t$ on $M$ which coincides with the mean curvature vector of $\Sigma_{i,t}$ on $\Sigma_{i,t}$ at time $t$. We can suppose that $X_t$ is supported in $\{x\in M: \dist(u(x),B) < r\}$ for all $t$. Let $\zeta$ be a smooth function on $M$ which is identically 1 outside $N_{2r}(C)$ and identically 0 inside $N_r(C)$.  Let $\phi_t$ be the flow of $\zeta X_t$ and define 
\[
u_t = u \circ \phi_t\inv.
\]
If $t$ is small enough, the following properties hold:
\begin{itemize}
\item[(i)] $u_t$ is Morse and has exactly the same critical points and critical values as $u$,
\item[(ii)] if $x\in(u_t)\inv(B) \setminus N_{2r}(C)$ then  $v(u_t,x) < \infty$ (c.f. \cite[Proposition 3.9]{zhou2020existence}),
\item[(iii)] if $\dist(b,B)\ge r$ then $\{u_t = b\} = \{u = b\}$. 
\end{itemize}
Moreover, $u_t\to u$ in $C^\infty(M)$ as $t\to 0$. We select a small time $\tau$ and then set $\tilde u = u_\tau$. $\blacksquare$

{\bf Construction 2:} Let $u\f M\to \R$ be a Morse function. Let $B = \{b_1,\hdots,b_k\}$ be the set of critical values of $u$.  Let $\mathcal N_\rho(B) = \{b\in \R: \dist(b,B) < \rho\}$.  Fix a small number $r > 0$. By the argument in \cite[Proposition 3.9]{zhou2020existence}, there exists a small positive number $s$ and smooth family of smooth functions $w_t\colon M\to \R$, $t\in [0,s]$ such that $\Sigma_{b,t} = \{w_t = b\}$ is the surface obtained from running mean curvature flow for time $t$ starting from $\{u=b\}$ for all $b\in \R \setminus \mathcal N_{r/2}(B)$ and all $t\in [0,s]$. Choose a smooth cut-off function $\zeta$ which is identically 1 outside $u\inv(\mathcal N_{r}(B))$ and identically 0 inside $u\inv(\mathcal N_{r/2}(B))$. Then set 
\[
u_t = \zeta w_t + (1-\zeta)u.
\]
If $t$ is small enough then the following properties hold:
\begin{itemize}
\item[(i)] $u_t$ is Morse and has exactly the same critical points and critical values as $u$,
\item[(ii)] if $x\in M\setminus (u_t)\inv(\mathcal N_{2r}(B))$ then $v(u_t,x) < \infty$ (c.f. \cite[Proposition 3.9]{zhou2020existence}),
\item[(iii)] $(u_t)\inv(B) = u\inv(B)$.
\end{itemize}
Moreover, $u_t\to u$ in $C^\infty(M)$ as $t\to 0$. We select a small time $\tau$ and then set $\tilde u = u_\tau$. $\blacksquare$

Now choose a smooth Morse function $h\f M\to \R$. Let $C$ be the set of critical points of $h$ and let $B= \{b_1,\hdots,b_k\}$ be the set of critical values. Fix some $\eps_0 > 0$. To begin, apply Construction 1 with $u= h$ and $r = r_0 = 1$ and $\tau$ small enough to ensure that 
\[
\rho(u,\tilde u) < \frac{\eps_0}{2}. 
\]
Then define $h_1 =\tilde u$ and let $G_1$ be a closed set containing a neighborhood of the set $(h_1)\inv(B) \setminus N_2(C)$ with the property that $v(h_1,x)<\infty$ for all $x\in G_1$. Note that such a $G_1$ exists by Lemma \ref{usc2}. We further select $\eps_1 > 0$ small enough that if $\rho(w,h_1) < \eps_1$ then 
\begin{itemize}
\item[(i)] $v(w,x) < \infty$ for all $x\in G_1$,
\item[(ii)] $w\inv(B) \setminus N_{2}(C) \subset G_1$.
\end{itemize}
This is possible by Lemma \ref{usc1}. 

Next we apply Construction 2 with $u = h_1$ and $r= r_1 = 2^{-1}$ and $\tau$ small enough to ensure that 
\[
\rho(u,\tilde u) < \min\left\{\frac{\eps_1}{2},\frac{\eps_0}{4}\right\}.
\]
Then we set $h_2 = \tilde u$. We let $G_2$ be a closed set containing $G_1$ and a neighborhood of $M\setminus (h_2)\inv(\mathcal N_{1}(B))$ with the property that $v(h_2,x)<\infty$ for all $x\in G_2$. Such a set $G_2$ exists by Lemma \ref{usc2}. Note that $G_2$ also contains a neighborhood of $(h_2)\inv(B)\setminus N_2(C) = (h_1)\inv(B)\setminus N_2(C)$. We select $\eps_2 > 0$ small enough that if $\rho(w,h_2)<\eps_2$ then 
\begin{itemize}
\item[(i)] $v(w,x) < \infty$ for all $x\in G_2$,
\item[(ii)] $w\inv(B) \setminus N_{2}(C) \subset G_2$.
\item[(iii)] $M \setminus w\inv(\mathcal N_{1}(B)) \subset G_2$.
\end{itemize}
This is possible by Lemma \ref{usc1}. 

Now, for $n\ge 2$, suppose inductively that $h_n \colon M\to \R$ is a smooth Morse function with the same critical points and critical values as $h$. Let $r_n = 2^{-n}$. Assume that $G_n$ is a closed subset of $M\setminus C$ such that $v(h_n,x) < \infty$ for all $x\in G_n$. Suppose that $G_n$ contains a neighborhood of $(h_n)\inv(B) \setminus N_{4 r_n}(C)$ when $n$ is odd, and that $G_n$ contains a neighborhood of $(h_n)\inv(B)\setminus N_{8r_n}(C)$ when $n$ is even.  Further suppose that $G_n$ contains a neighborhood of $M\setminus (h_n)\inv(\mathcal N_{4r_n}(B))$ when $n$ is even, and that $G_n$ contains a neighborhood of $M\setminus (h_n)\inv(\mathcal N_{8r_n}(B))$ when $n$ is odd. Choose $\eps_n$ small enough that if $\rho(w,h_n) < \eps_n$ then 
\begin{itemize}
\item[(i)] $v(w,x)<\infty$ for all $x\in G_n$,
\item[(ii)]  $w\inv(B) \setminus N_{8r_n}(C) \subset G_n$. 
\item[(iii)] $M\setminus w\inv(\mathcal N_{8r_n}(B)) \subset G_n$. 
\end{itemize} 
This is possible by Lemma \ref{usc1}. 

Now if $n$ is even, apply Construction $1$ with $u = h_n$ and $r = r_n$ and $\tau$ small enough to ensure that 
\[
\rho(u,\tilde u) < \min\left\{\frac{\eps_n}{2}, \frac{\eps_{n-1}}{2^2}, \hdots, \frac{\eps_0}{2^{n+1}}\right\}.
\]
Then set $h_{n+1} = \tilde u$. Let $G_{n+1}$ be a closed set containing $G_n$ together with a neighborhood of $(h_{n+1})\inv(B) \setminus N_{4r_{n+1}}(C)$ with the property that $v(h_{n+1},x)<\infty$ for all $x\in G_{n+1}$. Again such a choice is possible by Lemma \ref{usc2}. Further note that $G_{n+1}$ contains a neighborhood of $M\setminus (h_{n+1})\inv(\mathcal N_{8r_{n+1}}(B))$ since $M\setminus (h_{n+1})\inv(\mathcal N_{8r_{n+1}}(B)) = M\setminus (h_n)\inv (\mathcal N_{4r_n}(B))$ and $G_n$ contains a neighborhood of $M\setminus (h_n)\inv (\mathcal N_{4r_n}(B))$.

If $n$ is odd, apply Construction 2 with $u = h_n$ and $r = r_n$ and $\tau$ small enough to ensure that 
\[
\rho(u,\tilde u) < \min\left\{\frac{\eps_n}{2}, \frac{\eps_{n-1}}{2^2}, \hdots, \frac{\eps_0}{2^{n+1}}\right\}.
\]
Then set $h_{n+1} = \tilde u$. Let $G_{n+1}$ be a closed set containing $G_n$ and a neighborhood of $M\setminus h_{n+1}\inv(\mathcal N_{4 r_{n+1}})$ with the property that $v(h_{n+1},x)<\infty$ for all $x\in G_{n+1}$. Again such a choice is possible by Lemma \ref{usc2}.  Note that $G_{n+1}$ contains a neighborhood of $(h_{n+1})\inv(B) \setminus N_{8r_{n+1}}(C)$ since $(h_{n+1})\inv(B) \setminus N_{8r_{n+1}}(C) = (h_n)\inv(B) \setminus N_{4r_n}(C)$ and $G_n$ contains a neighborhood of $(h_n)\inv (B) \setminus N_{4r_n}(C)$. 

Now observe that $(h_n)$ is a Cauchy sequence with respect to $\rho$ and so it converges to some smooth function $h^*$ with $\rho(h,h^*) < \eps_0$. If $\eps_0$ is chosen sufficiently small, then $h^*$ is also Morse. Since each critical point $y$ of $h$ is also a critical point for $h^*$ with $h(y) = h^*(y)$, it follows that $h^*$ has the same critical points and critical values as $h$. Moreover, we know that $v(h^*,x) < \infty$ for all $x\in G = \cup_{n=1}^\infty G_n$. To complete the proof, it remains to show that $G = M\setminus C$. 

First, we claim that $G$ contains $(h^*)\inv(B) \setminus C$. It suffices to show that $G$ contains $(h^*)\inv(B) \setminus N_{8r_n}(C)$ for all $n$. To see this, observe that by construction one has 
\[
\rho(h^*,h_n) < \sum_{j=1}^\infty \frac{\eps_n}{2^j} = \eps_n. 
\]
By the choice of $\eps_n$, this implies that $(h^*)\inv (B) \setminus N_{8r_n}(C) \subset G_n \subset G$. This proves the claim. 

Second, we claim that $G$ contains $M\setminus (h^*)\inv(B)$.  It suffices to show that $G$ contains $M\setminus (h^*)\inv(\mathcal N_{8r_n}(B))$ for all $n$. To see this, again observe that 
\[
\rho(h^*,h_n) < \sum_{j=1}^\infty \frac{\eps_n}{2^j} = \eps_n. 
\]
By the choice of $\eps_n$, this implies that $M\setminus (h^*)\inv(\mathcal N_{8r_n}(B)) \subset G_n \subset G$. This completes the proof. 
\end{proof}

\section{Generic Metrics}
\label{generic metrics}

In this appendix, we prove three theorems about generic metrics. The first result says that, generically, half-volume minimal hypersurfaces do not exist.  The second result says that, generically, half-volume CMCs are isolated. Finally, the third result says that, generically, every almost-embedded half volume CMC is actually embedded.

\begin{prop}
\label{half-volume-minimal}
Let $M$ be a closed manifold. Then, for a generic smooth metric $g$ on $M$, there is no  smooth, closed, (not necessarily connected) minimal hypersurface which encloses half the volume of $M$. 
\end{prop}

\begin{proof}
Given $C > 0$ and $I\in \N$ and a smooth metric $g$ on $M$, let $\mathcal M_{C,I}(g)$ be the set of all smooth, closed, embedded (not necessarily connected) minimal hypersurfaces in $(M,g)$ with area at most $C$ and index at most $I$.  Let $\mathcal G_{C,I}$ be the set of all smooth metrics $g$ on $M$ such that the set $\mathcal M_{C,I}(g) = \{\Sigma_1,\hdots,\Sigma_n\}$ is finite, each surface $\Sigma_i$ is non-degenerate, and none of the surfaces $\Sigma_i$ encloses half the volume of $M$.  

Let $\mathcal G$ be the set of all smooth metrics on $M$. It suffices to show that for each fixed $C>0$ and $I\in \N$, the set $\mathcal G_{C,I}$ is open and dense in $\mathcal G$. First, we claim that $\mathcal G_{C,I}$ is open. Indeed, choose some $g\in \mathcal G_{C,I}$ and write $\mathcal M_{C,I} = \{\Sigma_1,\hdots,\Sigma_n\}$. Then, since each $\Sigma_i$ is non-degenerate, there is a neighborhood $U_1$ of $g$ in $\mathcal G$ such that for every $\tilde g\in \mathcal G$ and every $i=1,\hdots,N$ there is a unique small perturbation of $\Sigma_i$ to a hypersurface $\Sigma_i(\tilde g)$ which is minimal and non-degenerate in $(M,\tilde g)$. Then, by Sharp's compactness theorem \cite{sharp2017compactness}, it follows that for $\tilde g$ in a possibly smaller neighborhood $U_2$ of $g$ we have 
\[
\mathcal M_{C,I}(\tilde g) \subset \{\Sigma_1(\tilde g),\hdots,\Sigma_n(\tilde g)\}. 
\]
Finally, since none of the hypersurfaces $\Sigma_i$ enclose half the volume of $(M,g)$, it follows that for $\tilde g$ in a potentially smaller neighborhood $U_3$ of $g$, none of the hypersurfaces $\Sigma_i(\tilde g)$ enclose half the volume of $(M,\tilde g)$. This proves that $\mathcal G_{C,I}$ is open. 

It remains to show that $\mathcal G_{C,I}$ is dense in $\mathcal G$. Choose an arbitrary metric $g\in \mathcal G$. Then there exists a small perturbation $g_1$ of $g$ which is bumpy. By Sharp's compactness theorem, it follows that the set $\mathcal M_{C,I}(g_1) = \{\Sigma_1,\hdots,\Sigma_n\}$ is finite and consists entirely of non-degenerate minimal hypersurfaces.  Again there is a small neighborhood $U_1$ of $g_1$ in $\mathcal G$ such that for every $\tilde g\in U$ there is a small perturbation of $\Sigma_i$ to a hypersurface $\Sigma_i(\tilde g)$ which is minimal and non-degenerate in $(M,\tilde g)$. As above, by Sharp's compactness theorem, for $\tilde g$ in a possibly smaller neighborhood $U_2$ of $g_1$ we have 
\[
\mathcal M_{C,I}(\tilde g)\subset \{\Sigma_1(\tilde g),\hdots,\Sigma_n(\tilde g)\}. 
\]
We now inductively perturb $g_1$ to ensure that no surfaces in this collection encloses half the volume. 

First consider $\Sigma_1(g_1)$. If $\Sigma_1(g_1)$ does not enclose half the volume of $(M,g_1)$ then we set $g^1 = g_1$. Suppose instead that $\Sigma_1 = \bd \Omega_1$ does enclose half the volume of $M$.  Then we set $g^1 = e^{2\phi} g_1$ where $\phi$ is non-negative, and vanishes on the closure of $\Omega_1$, and is positive at some point in $M\setminus \Omega_1$. By selecting $\phi$ close enough to 0, we can ensure that $g^1\in U_2$.  
 Then $\Sigma_1(g^1) = \Sigma_1$ as sets but $\Sigma_1(g^1)$ no longer encloses half the volume of $(M,g^1)$.  
 
 Next consider $\Sigma_2(g^1)$. If $\Sigma_2(g^1)$ does not enclose half the volume of $(M,g^1)$ then set $g^2 = g^1$. Suppose instead that $\Sigma_2(g^1) = \bd \Omega_2$ does enclose half the volume of $(M,g^1)$. Then we set $g^2 = e^{2\phi}g^1$ where $\phi$ is non-negative, and vanishes on the closure of $\Omega_2$ together with the support of $\Sigma_1(g^1)$, and is positive at some point in $M\setminus \Omega_2$. By selecting $\phi$ close enough to $0$, we can ensure that $g^2\in U_2$ and that $\Sigma_1(g^2)$ does not enclose half the volume of $(M,g^2)$. Note that $\Sigma_2(g^2) = \Sigma_2(g^1)$ as sets but $\Sigma_2(g^2)$ no longer encloses half the volume of $(M,g^2)$. 
 
Continuing this process inductively, one arrives at a metric $g^n\in U_2$ with the property that none of the hypersurfaces $\Sigma_1(g^n),\hdots,\Sigma_n(g^n)$ enclose half the volume of $(M,g^n)$. Thus $g^n \in \mathcal G_{C,I}$ and so $g^n$ is the required perturbation of $g$. 
 \end{proof}
 
Let $g$ be a smooth Riemannian metric on $M$.  Assume that $u\colon \Sigma \to M$ is a two-sided, null-homologous embedding with a preferred choice of normal vector $\nu$ with respect to $g$. Then there is a unique region $\Omega$ whose boundary is $u(\Sigma)$ and such that $\nu$ points into $\Omega$. We define the volume enclosed by $u$ to be the volume of $\Omega$. Note that the enclosed volume depends only on the image $u(\Sigma)$ and the preferred choice of normal vector, and not the particular immersion $u$.

Now suppose that $u\colon \Sigma \to M$ is a two-sided, null-homologous CMC almost-embedding with non-zero mean curvature and a preferred choice of normal vector $\nu$ with respect to  $g$.  Consider another immersion $w\colon \Sigma \to M$ such that 
\[
w(x) = \exp^g_{u(x)}(f(x)\nu)
\]
for some $f\in C^{j,\alpha}(\Sigma)$ with small norm. Regard $w(\Sigma)$ as an integer multiplicity current (oriented by the preferred choice of normal vector). 
Then exactly one of the following is true (depending on whether the normal vector and the mean curvature vector point in the same direction or opposite directions): 
\begin{itemize}
\item[(i)] There exists a unique $\Theta = \theta(x) \rest \mathcal H^{n+1}$ with $\theta(x) \in \{0,1,2\}$ such that $\bd \Theta = w(\Sigma)$.
\item[(ii)] There exists a unique $\Theta = \theta(x) \rest H^{n+1}$ with $\theta(x) \in \{-1,0,1\}$ such that $\bd \Theta = w(\Sigma)$.
\end{itemize}
We define the volume enclosed by $w$ to be 
\[
\int_M \theta(x)\, d\mathcal H^{n+1}(x). 
\]
Observe that the enclosed volume of $w$ actually depends only on $w(\Sigma)$ and the preferred choice of normal vector, and not on the particular immersion $w$. 

Let $M^{n+1}$ and $\Sigma^n$ be smooth, closed manifolds. Fix $q,j\in \N$ and $\alpha\in(0,1)$ with $q \ge j+3$. Let $\mathcal P$ be the set of all triples $(g,[u],H)$ where $g$ is a $C^q$ Riemannian metric on $M$, and $[u]$ is the equivalence class of a simple, two-sided, null-homologous, $C^{j,\alpha}$ immersion $u\colon \Sigma \to M$ with constant mean curvature $H$ and a preferred choice of normal vector. Let $\mathcal P' \subset \mathcal P$ be a neighborhood of the triples $(g,[u],H)$ where $u$ is an almost-embedding. We can suppose that $\mathcal P'$ is small enough that the enclosed volume function is defined on $\mathcal P'$. Then let $\mathcal M \subset \mathcal P'$ be the set of all triples $(g,[u],H)$ where the volume enclosed by $u$ is equal to $\frac{1}{2}\vol(M,g)$. Define $\Pi\colon \mathcal M\to \Gamma$ by $\Pi(g,[u],H) = g$. Let $\mathcal M_\infty$ be the subset of $\mathcal M$ consisting of triples $(g,[u],H)$ where $g$ and $u$ are smooth.

\begin{prop}
There is an open subset $\mathcal U\subset \mathcal M$ containing $\mathcal M_\infty$ such that $\mathcal U$ has the structure of a separable $C^{q-j}$ Banach manifold. Moreover, the differential of $\Pi$ restricted to $\mathcal U$ is Fredholm with index 0.
\end{prop}

\begin{proof} The proof is almost identical to White's manifold structure theorem \cite{white1991space}. We will provide a sketch of the argument.  Let $\Gamma$ be the space of all $C^q$ Riemannian metrics on $M$. Fix an arbitrary smooth metric $g\in \Gamma$.  Fix a smooth, simple, two-sided, null-homologous, almost-embedding $u\colon \Sigma \to M$ with a preferred choice of normal vector $\nu$, constant mean curvature $H$ with respect to $g$, and enclosed volume $\frac{1}{2}\vol(M,g)$ with respect to $g$. Here we compute the mean curvature of $H$ with respect to $\nu$ and the enclosed volume is defined as above.

Define a map 
$
\mathcal H: \Gamma \times C^{j,\alpha}(\Sigma) \to C^{j-2,\alpha}(\Sigma)
$
by letting $\mathcal H(h,f)$ be the mean curvature of the immersion 
\[
x \mapsto \exp^g_{u(x)}(f(x)\nu),
\]
computed with respect to $h$. 
 The differential of the function $\mathcal H$ at the point $(g,0)$ is well-known.  If $\varphi$ is a smooth function on $M$ and $h(t) = e^{2t \varphi} g$ then 
\[
D_1\mathcal H(h'(0)) = -2\varphi H - n\la \grad \varphi,\nu\ra. 
\] 
If $w\in C^{j,\alpha}(\Sigma)$ then 
\[
D_2 \mathcal H(w) = Jw = \lap w + (\vert A\vert^2 + \ric(\nu,\nu))w,
\]
where the Laplacian and second fundamental form are those of $u(\Sigma)$ with respect to $g$. 

Now consider the volume functional. Define 
$
\mathcal V\f \Gamma \times C^{j,\alpha}(\Sigma)\to \R
$
by letting $\mathcal V(h,f)$ be the volume enclosed by the immersion 
\[
x \mapsto \exp^g_{u(x)}(f(x)\nu),
\]
computed with respect to $h$. As explained above, this is well-defined if $\|f\|_{j,\alpha}$ is small-enough, and it depends only on the image of the immersion and the preferred choice of normal vector. The differential of $\mathcal V$ is also well-known. Let $\Omega$ be the region enclosed by $u(\Sigma)$. If $\psi$ is a smooth function on $M$ supported in $\Omega$ and $h(t) = e^{2t\psi}g$ then 
\[
D_1\mathcal V(h'(0)) = (n+1) \int_\Omega \psi \, dV_g. 
\]
Also one has 
\[
D_2\mathcal V(w) = -\int_\Sigma w\, dA_{u^*g}. 
\]
for $w\in C^{j,\alpha}(\Sigma)$. 

Now define 
\begin{gather*}
\mathcal F\colon \Gamma \times C^{j,\alpha}(\Sigma) \times \R \to C^{j-2,\alpha}(\Sigma)\times \R,\\
\mathcal F(h,f,a) = \left(\mathcal H(h,f)-a,\mathcal V(h,f) - \frac{\vol(M,h)}{2}\right). 
\end{gather*}
Also define  
\begin{gather*}
L\colon  C^{j,\alpha}(\Sigma)\times \R \to C^{j-2,\alpha}(\Sigma)\times \R,\\
L(w,b) = D \mathcal F(0,w,b) = \left(Jw - b, -\int_\Sigma w\, dA_g\right).
\end{gather*}
The key point is that $L$ is formally self-adjoint, with finite dimensional kernel $K$, and image $(K^\perp) \cap (C^{j-2,\alpha}(\Sigma)\times \R)$. Moreover, it is easy to see from the explicit formulas for $D_1 \mathcal H$ and $D_1\mathcal V$ along conformal paths of metrics that no non-zero element of $K$ is orthogonal to the image of $D_1\mathcal F$. As in \cite{white1991space}, these facts imply that $D\mathcal F$ is onto, and that the kernel of $D\mathcal F$ is complemented.

It now follows by the implicit function theorem that $\mathcal N = \mathcal F\inv(0,0)$ has the structure of a $C^{q-j}$ Banach manifold in a neighborhood of $(g,0,H)$. 
Now consider the map 
\begin{gather*}
\Pi\colon \Gamma\times C^{j,\alpha}(\Sigma)\times \R \to \Gamma,\\
\Pi(h,f,a) = h.
\end{gather*}
Again, the same argument as in \cite{white1991space} shows that the differential $D(\Pi \vert_{\mathcal N})$, evaluated at $(g,0,H)$, is Fredholm with index 0. 

The above construction gives a chart for $\mathcal M$ in a neighborhood of every $(g,[u],H)$ in which $g$ is smooth and in which $u$ is a smooth almost-embedding.  Let $\mathcal U$ be the open subset of $\mathcal M$ covered by these charts. As in \cite{white1991space}, one can show that the transition maps between charts are $C^{q-j}$ and therefore that this gives $U$ the claimed manifold structure. 
\end{proof}

\begin{corollary}
Let $M^{n+1}$ be a smooth, closed manifold. Then for a (Baire) generic set of smooth Riemannian metrics on $M$, there are no almost-embedded half-volume CMC hypersurfaces $\Sigma$ in $(M,g)$ which carry a non-zero function $\varphi$ such that $J_\Sigma \varphi$ is constant and $\int_\Sigma \varphi = 0$.
\end{corollary} 

\begin{proof} 
Fix $q$, $j$, and $\alpha$ as before. Let $\Gamma_q$ be the space of $C^q$ metrics on $M$. Let $\Gamma_\infty$ be the space of smooth Riemannian metrics on $M$. Fix a smooth, closed $n$-dimensional manifold $\Sigma$ and consider the manifold $\mathcal U\subset \mathcal M$ containing $\mathcal M_\infty$ from the previous proposition.  Since $\Pi\colon \mathcal U\to \Gamma_q$ is Fredholm with index 0 and $\mathcal U$ is separable, it follows from Smale's infinite dimensional Sard theorem \cite{smale1965infinite} that the set $\mathcal R$ of regular values of $\Pi$ is of second category in $\Gamma_q$. 
According to \cite[Theorem 2.10]{white2017bumpy}, the set $\mathcal R \cap \Gamma_\infty$ is therefore of second category in $\Gamma_\infty$. Now observe that if $g\in \mathcal R\cap \Gamma_\infty$ and $u\colon \Sigma\to M$ is a smooth half-volume almost-embedded CMC then there is no non-zero function $\varphi$ on $\Sigma$ such that $J_\Sigma \varphi$ is constant and $\int_\Sigma \varphi = 0$. Since there are only countably many possible diffeomorphism classes for $\Sigma$, the result follows. 
\end{proof}

\begin{corollary}
\label{half-volume-isolated}
Let $M^{n+1}$ be a smooth, closed manifold. Then for a (Baire) generic set of smooth Riemannian metrics on $M$, every almost-embedded half-volume CMC hypersurface $\Sigma$ in $(M,g)$ is isolated. 
\end{corollary} 

\begin{proof}
Let $\Sigma$ be an almost-embedded half-volume CMC hypersurface in $(M,g)$. Consider the operator $\mathcal F\colon C^{j,\alpha}(\Sigma)\times \R \to C^{j-2,\alpha}(\Sigma)\times \R$ given by 
\[
F(f,a) = \left(\mathcal H(g,f)-a, \mathcal V(g,f)-\frac{\vol(M,g)}{2}\right). 
\]
Then the linearization of $\mathcal F$ at $(0,0)$ is 
\[
L(w,b) = \left(J_\Sigma w - b,-\int_\Sigma w\right). 
\]
Therefore the previous corollary implies that, for generic $g$, the map $L$ is injective and hence an isomorphism. The inverse function theorem now implies that $\Sigma$ is isolated. 
\end{proof}

\begin{prop}
Assume that $M^{n+1}$ is a smooth, closed manifold. Then for a (Baire) generic set of smooth Riemannian metrics on $M$, every almost-embedded half-volume CMC hypersurface in $(M,g)$ is actually embedded. 
\end{prop}

\begin{proof}
Again we can follow White's self-transversality argument \cite{white2019generic} almost identically.  Let $\Pi\colon \mathcal M_\infty\to \Gamma_\infty$ be the projection to the metric, and let $\mathcal M_{\text{reg}}$ be the set of points $(g,[u],H)$ in $\mathcal M_\infty$ at which the differential $D\Pi$ has trivial kernel.  Pick a point $(g,[u],H)\in \mathcal M_{\text{reg}}$ and assume that $u\colon \Sigma\to M$ is an almost-embedding.  Pick a small open set $W\subset M$ such that $u(\Sigma)$ is embedded in $W$.  Let 
\[
V_0 = \{f \in C^\infty(\Sigma): Jf \text{ is supported in } W\}. 
\]
Let $\Omega^2 \Sigma = \{(x_1,x_2)\in \Sigma\times\Sigma: x_1\neq x_2\}$ and then let 
\[
C = \{(x_1,x_2)\in \Omega^2\Sigma: u(x_1) = u(x_2)\}
\]
which is a compact subset of $\Omega^2(\Sigma)$. According to \cite[Theorem 29]{white2019generic}, there is a finite dimensional subspace $V\subset V_0$ such that given any $(x_1,x_2)\in C$ and $a_1,a_2\in \R$, there is an $f\in V$ with $f(x_i) = a_i$ for $i=1,2$. 

Let $f_1,\hdots,f_d$ be a basis for $V$. Next, we can argue as in \cite[Proposition 8]{white2019generic} to find maps 
\begin{gather*}
\gamma\colon B^d(0,\eps)\to C^\infty(M),\\
\mathcal W\colon B^d(0,\eps) \times \Sigma \to M
\end{gather*}
such that 
\begin{itemize}
\item[(i)] $\gamma(0) = 0$,
\item[(ii)] $\mathcal W(0,\cdot) = u(\cdot)$,
\item[(iii)] for each $\tau\in B^k(0,\eps)$, the map $\mathcal W(\tau,\cdot)$ is a half-volume CMC immersion with respect to the metric $e^{2\gamma(\tau)}g$,
\item[(iv)] For each $i=1,\hdots,d$ we have 
\[
\frac{d}{dt}\eval_{t=0} \mathcal W(te_i,x) = f_i(x). 
\]
\end{itemize}
Indeed, since $Jf_i$ is supported in $W$, it is clear from the formula for $D_1\mathcal F$ that there exists $\varphi_i\in C^\infty(M)$ such that $D_1\mathcal F(v_i) = (-J f_i,\int_\Sigma f_i)$, where $h_i(t) = e^{2t\varphi_i}g$ and $v_i = h_i'(0)$. Then define 
\begin{gather*}
\gamma\colon \R^d \times M\to \R,\\
\gamma(\tau,x) = \sum_{i=1}^d 2 \tau_i \varphi_i(x). 
\end{gather*}
Since $(g,[u],H)\in \mathcal M_{\text{reg}}$, it follows that for a sufficiently small $\eps > 0$, for each $\tau \in B^d(0,\eps)$ there is a unique function $w(\tau)$ with small norm on $\Sigma$ such that 
\[
\mathcal W(\tau,x) = \exp^g_{u(x)}(w(\tau) \nu)
\]
is a half-volume CMC immersion with respect to $e^{2\gamma(\tau)}g$ with mean curvature $H(\tau)$. 
Moreover, the map 
\[
\mathcal W: B^d(0,\eps)\times \Sigma \to M
\]
is smooth. Finally, observe that $\mathcal F(e^{2\gamma(\tau)}g, w(\tau), H(\tau)) \equiv 0$, and hence 
\begin{align*}
(0,0) &= \frac{d}{d\tau_i}  \mathcal F(e^{2\gamma(\tau)}g, w(\tau), H(\tau)) \\
&= D_1\mathcal F(v_i) + L\left(\frac{d}{d\tau_i} w, \frac{d}{d\tau_i} H\right),
\end{align*}
where $L$ is defined as above. This implies that 
\[
L\left(\frac{d}{d\tau_i} w, \frac{d}{d\tau_i} H\right) = L\left(f_i,0\right)
\]
Since $(g,[u],H)\in \mathcal M_{\text{reg}}$ the operator $L$ is invertible, and therefore this implies that 
\[
\frac{d}{dt}\eval_{t=0} \mathcal W(te_i,x) = \frac{d}{d\tau_i} w = f_i,
\]
as needed. 

Let $\Delta_2 M = \{(x,x)\in M^2: x\in M\}$. The same argument as \cite[Theorem 14]{white2019generic} now shows that the map 
\begin{gather*}
\widetilde {\mathcal U}\colon B^d(0,\eps) \times \Omega^2 \Sigma \to M\times M,\\
\widetilde {\mathcal U}(\tau,x_1,x_2) = \big(\mathcal U(\tau,x_1),\mathcal U(\tau,x_2)\big)
\end{gather*}
is a submersion at every point in $\widetilde {\mathcal U}\inv(\Delta^2 M)$. From this construction, it now follows as in \cite[Theorem 15]{white2019generic}  that, for a generic smooth metric $g$ on $M$, every half-volume CMC immersion is self-transverse. In particular, for such metrics, every almost-embedded half-volume CMC is embedded. 
\end{proof}

\bibliographystyle{plain}
\bibliography{bibliography}

\begin{thebibliography}{10}

\bibitem{almgren1962homotopy}
Frederick Almgren.
\newblock The homotopy groups of the integral cycle groups.
\newblock {\em Topology}, 1(4):257--299, 1962.

\bibitem{almgren1965theory}
Frederick Almgren.
\newblock The theory of varifolds.
\newblock {\em Mimeographed notes}, 1965.

\bibitem{arnold2004arnold}
Vladimir~I Arnold.
\newblock {\em Arnold's problems}.
\newblock Springer, 2004.

\bibitem{barbosa2012stability}
J~Lucas Barbosa, Manfredo~do Carmo, and Jost Eschenburg.
\newblock Stability of hypersurfaces of constant mean curvature in riemannian
  manifolds.
\newblock In {\em Manfredo P. do Carmo--Selected Papers}, pages 291--306.
  Springer, 2012.

\bibitem{bellettini2019curvature}
Costante Bellettini, Otis Chodosh, and Neshan Wickramasekera.
\newblock Curvature estimates and sheeting theorems for weakly stable cmc
  hypersurfaces.
\newblock {\em Advances in Mathematics}, 352:133--157, 2019.

\bibitem{bellettini2020inhomogeneous}
Costante Bellettini and Neshan Wickramasekera.
\newblock The inhomogeneous allen--cahn equation and the existence of
  prescribed-mean-curvature hypersurfaces.
\newblock {\em arXiv preprint arXiv:2010.05847}, 2020.

\bibitem{chodosh2020minimal}
Otis Chodosh and Christos Mantoulidis.
\newblock Minimal surfaces and the allen--cahn equation on 3-manifolds: Index,
  multiplicity, and curvature estimates.
\newblock {\em Annals of Mathematics}, 191(1):213--328, 2020.

\bibitem{Dey23}
Akashdeep Dey.
\newblock Existence of multiple closed {CMC} hypersurfaces with small mean
  curvature.
\newblock {\em J. Differential Geom.}, 125(2):379--403, 2023.

\bibitem{elbert2007stable}
Maria Elbert, Barbara Nelli, and Harold Rosenberg.
\newblock Stable constant mean curvature hypersurfaces.
\newblock {\em Proceedings of the American Mathematical Society},
  135(10):3359--3366, 2007.

\bibitem{gaspar2018allen}
Pedro Gaspar and Marco~AM Guaraco.
\newblock The allen--cahn equation on closed manifolds.
\newblock {\em Calculus of Variations and Partial Differential Equations},
  57:1--42, 2018.

\bibitem{gaspar2019weyl}
Pedro Gaspar and Marco~AM Guaraco.
\newblock The weyl law for the phase transition spectrum and density of limit
  interfaces.
\newblock {\em Geometric and Functional Analysis}, 29:382--410, 2019.

\bibitem{gromov2006dimension}
Mikhael Gromov.
\newblock Dimension, non-linear spectra and width.
\newblock In {\em Geometric Aspects of Functional Analysis: Israel Seminar
  (GAFA) 1986--87}, pages 132--184. Springer, 2006.

\bibitem{irie2018density}
Kei Irie, Fernando Marques, and Andr{\'e} Neves.
\newblock Density of minimal hypersurfaces for generic metrics.
\newblock {\em Annals of Mathematics}, 187(3):963--972, 2018.

\bibitem{li2023existence}
Yangyang Li.
\newblock Existence of infinitely many minimal hypersurfaces in
  higher-dimensional closed manifolds with generic metrics.
\newblock {\em Journal of Differential Geometry}, 124(2):381--395, 2023.

\bibitem{liokumovich2018weyl}
Yevgeny Liokumovich, Fernando Marques, and Andr{\'e} Neves.
\newblock Weyl law for the volume spectrum.
\newblock {\em Annals of Mathematics}, 187(3):933--961, 2018.

\bibitem{Marques-Neves21}
{F. C.} Marques and A.~Neves.
\newblock Morse theory for the area functional.
\newblock {\em S\~{a}o Paulo J. Math. Sci.}, 15(1):268--279, 2021.

\bibitem{marques2016morse}
Fernando~C Marques and Andr{\'e} Neves.
\newblock Morse index and multiplicity of min-max minimal hypersurfaces.
\newblock {\em Cambridge Journal of Mathematics}, 4(4):463--511, 2016.

\bibitem{marques2017existence}
Fernando~C Marques and Andr{\'e} Neves.
\newblock Existence of infinitely many minimal hypersurfaces in positive ricci
  curvature.
\newblock {\em Inventiones mathematicae}, 209(2):577--616, 2017.

\bibitem{marques2021morse}
Fernando~C Marques and Andr{\'e} Neves.
\newblock Morse index of multiplicity one min-max minimal hypersurfaces.
\newblock {\em Advances in Mathematics}, 378:107527, 2021.

\bibitem{marques2019equidistribution}
Fernando~C Marques, Andr{\'e} Neves, and Antoine Song.
\newblock Equidistribution of minimal hypersurfaces for generic metrics.
\newblock {\em Inventiones mathematicae}, 216(2):421--443, 2019.

\bibitem{Mazurowski22}
Liam Mazurowski.
\newblock C{MC} doublings of minimal surfaces via min-max.
\newblock {\em J. Geom. Anal.}, 32(3):Paper No. 104, 28, 2022.

\bibitem{mazurowski2023half}
Liam Mazurowski and Xin Zhou.
\newblock The half-volume spectrum of a manifold.
\newblock {\em arXiv preprint arXiv:2302.07722}, 2023.

\bibitem{morgan2003regularity}
Frank Morgan.
\newblock Regularity of isoperimetric hypersurfaces in riemannian manifolds.
\newblock {\em Transactions of the American Mathematical Society},
  355(12):5041--5052, 2003.

\bibitem{pitts2014existence}
Jon~T Pitts.
\newblock {\em Existence and regularity of minimal surfaces on Riemannian
  manifolds.(MN-27)}, volume~27.
\newblock Princeton University Press, 2014.

\bibitem{schoen1981regularity}
Richard Schoen and Leon Simon.
\newblock Regularity of stable minimal hypersurfaces.
\newblock {\em Communications on Pure and Applied Mathematics}, 34(6):741--797,
  1981.

\bibitem{sharp2017compactness}
Ben Sharp.
\newblock Compactness of minimal hypersurfaces with bounded index.
\newblock {\em Journal of Differential Geometry}, 106(2):317--339, 2017.

\bibitem{simon1983lectures}
Leon Simon.
\newblock {\em Lectures on geometric measure theory}, volume~3 of {\em
  Proceedings of the Centre for Mathematical Analysis, Australian National
  University}.
\newblock Australian National University, Centre for Mathematical Analysis,
  Canberra, 1983.

\bibitem{smale1965infinite}
S.~Smale.
\newblock An infinite dimensional version of sard's theorem.
\newblock {\em American Journal of Mathematics}, 87(4):861, 1965.

\bibitem{song2023existence}
Antoine Song.
\newblock {Existence of infinitely many minimal hypersurfaces in closed
  manifolds}.
\newblock {\em Annals of Mathematics}, 197(3):859 -- 895, 2023.

\bibitem{wang2023existence}
Zhichao Wang and Xin Zhou.
\newblock Existence of four minimal spheres in {$S^3$} with a bumpy metric.
\newblock {\em arXiv preprint arXiv:2305.08755}, 2023.

\bibitem{white1991space}
Brian White.
\newblock The space of minimal submanifolds for varying riemannian metrics.
\newblock {\em Indiana University Mathematics Journal}, pages 161--200, 1991.

\bibitem{white2017bumpy}
Brian White.
\newblock On the bumpy metrics theorem for minimal submanifolds.
\newblock {\em American Journal of Mathematics}, 139(4):1149--1155, 2017.

\bibitem{white2019generic}
Brian White.
\newblock Generic transversality of minimal submanifolds and generic regularity
  of two-dimensional area-minimizing integral currents.
\newblock {\em arXiv preprint arXiv:1901.05148}, 2019.

\bibitem{yau1982seminar}
Shing-Tung Yau.
\newblock {\em Seminar on differential geometry}.
\newblock Princeton University Press, 1982.

\bibitem{zhou2020multiplicity}
Xin Zhou.
\newblock On the multiplicity one conjecture in min-max theory.
\newblock {\em Annals of Mathematics}, 192(3):767--820, 2020.

\bibitem{Zhou-ICM22}
Xin Zhou.
\newblock Mean curvature and variational theory.
\newblock In {\em I{CM}---{I}nternational {C}ongress of {M}athematicians.
  {V}ol. {IV}. {S}ections 5--8}, pages 2696--2717. EMS Press, Berlin, [2023]
  \copyright 2023.

\bibitem{zhou2019min}
Xin Zhou and Jonathan~J Zhu.
\newblock Min--max theory for constant mean curvature hypersurfaces.
\newblock {\em Inventiones mathematicae}, 218:441--490, 2019.

\bibitem{zhou2020existence}
Xin Zhou and Jonathan~J Zhu.
\newblock Existence of hypersurfaces with prescribed mean curvature i--generic
  min-max.
\newblock {\em Cambridge Journal of Mathematics}, 8(2):311--362, 2020.

\end{thebibliography}

\end{document}